\documentclass[12pt]{amsart}
\usepackage{amsthm,amsfonts,amsmath, amscd, amssymb, hyperref,amssymb,upgreek}
\usepackage{fullpage}
\usepackage[pdftex]{graphicx}
\usepackage{enumerate}
\usepackage{cite}
\usepackage[usenames,dvipsnames]{color}
\usepackage{subcaption}
\usepackage{graphicx, upgreek}
\usepackage[margin=1in]{geometry}

\usepackage{pdfpages}
\usepackage{xcolor}

\newtheorem{theorem}{Theorem}

\makeatletter
\@namedef{subjclassname@2020}{
  \textup{2020} Mathematics Subject Classification}
\makeatother

\usepackage{datetime,mathtools}
\usepackage{xcolor}
\usepackage{hyperref}
\definecolor{darkgreen}{rgb}{0,0.4,0}
\definecolor{BrickRed}{rgb}{0.65,0.08,0}
\hypersetup{colorlinks=true,linkcolor=blue,citecolor=red,filecolor=BrickRed,urlcolor=darkgreen}

\mathtoolsset{showonlyrefs} 

\newcount\m \newcount\n
\def\hours{\n=\time \divide\n 60
	\m=-\n \multiply\m 60 \advance\m \time
	\twodigits\n:\twodigits\m}
\def\twodigits#1{\ifnum #1<10 0\fi \number#1}
\date{\today \quad
	\textit{Time}: \hours \quad 
	\textit{Version:} 33
	\\
	Research supported in part by an NSF grant (Focus Research Group DMS-1854398), and the American Institute of Mathematics.
}

\numberwithin{equation}{section}

\title{On correlation of the 3-fold divisor function with itself}

\author[D. T. Nguyen]{David T. Nguyen}
\address{Previous Addres: American Institute of Mathematics, 600 E. Brokaw Rd., San Jose, CA 95112, USA.
}
\email{dtn@aimath.org}

\address{Current Address: Department of Mathematics and Statistics, Queen's University, Jeffery Hall, 48 University Ave, Kingston, Ontario, K7L-3N6, Canada.}
\email{d.nguyen@queensu.ca}

\newtheorem{thm}{Theorem}
\newtheorem{cor}{Corollary}
\newtheorem{remark}{Remark}
\newtheorem{conjecture}{Conjecture}
\newtheorem{lem}{Lemma}
\newtheorem{prop}{Proposition}

\theoremstyle{definition}

\numberwithin{equation}{section}

\begin{document}

\maketitle

\begin{abstract}
	Let $\zeta^k(s) = \sum_{n=1}^\infty \tau_k(n) n^{-s}, \Re s > 1$.
	We present three conditional results on the ternary additive correlation sum $$\sum_{n\le X} \tau_3(n) \tau_3(n+h),\quad (h\ge 1),$$ and give numerical verifications of our method. The first is a conditional proof for the full main term of the above correlation sum for any composite shift $1 \le h \le X^{2/3}$, on assuming an averaged level of distribution for the three-fold divisor function $\tau_3(n)$ in arithmetic progressions to level two-thirds. The second is a conditional derivation for the leading order main term asymptotics of this correlation sum, also valid for any composite shift $1 \le h \le X^{2/3}$. The third result gives a complete expansion of the polynomial for the full main term for the special case $h=1$ from both our method and from the delta-method, showing that our answers match. 
	
	Our method is essentially elementary, especially for the $h=1$ case, uses congruences, and, as alluded to earlier, gives the same answer as in prior prediction of Conrey and Gonek \cite{ConreyGonnek2002} (Duke Math. J. \textbf{107} (3) pp. 577-604, 2002), previously computed by Ng and Thom \cite{NgThom2019} (Funct. Approx. Comment. Math. \textbf{60}(1): 97-142, 2019), and unpublished heuristic probabilistic arguments of Tao \cite{TaoBlog}. Our procedure is general and works to give the full main term with a power-saving error term for any correlations of the form $\sum_{n\le X} \tau_k(n) f(n+h)$, to any composite shift $h$, and for a wide class of arithmetic function $f(n)$.
\end{abstract}

\setcounter{tocdepth}{1}
\tableofcontents

\section{Introduction and statements of results}

For $k\ge 1$ let
\begin{equation}
	\zeta^k(s) = \sum_{n=1}^\infty \frac{\tau_k(n)}{n^s},\ (\Re s > 1).
\end{equation}
The additive correlation sums
\begin{equation} \label{eq:DklXh}
	D_{k,\ell} (X,h)
	=
	\sum_{ n\le X}
	\tau_k(n) \tau_\ell(n+h)
\end{equation}
of the $k$-fold divisor functions $\tau_k(n)$	are instrumental in the study of moments of $L$-functions, dating back to 1918 from G. Hardy and J. Littlewood in their pioneering work on the Second moment of the magnitude of the  Riemann zeta function on the vertical line with real part one-half, corresponding to the case $k=\ell=2$. Despite its importance, no one to this day has been able to rigorously prove even an asymptotic formula for this correlation when both $k$ and $\ell$ are three or larger, though it is widely believed (see, e.g., \cite[Conjecture 1.1]{NgThom2019}, \cite[Conjecture 1]{TaoBlog}, \cite[Conjecture 3]{ConreyGonnek2002}, and \cite[Conjecture 1.1 (ii)]{MatomakiRadziwillTao2019}), that
\begin{equation} \label{eq:CorrelationSum}
	\sum_{n \le X} \tau_3(n) \tau_3(n+1)
	\sim 
	\frac{1}{4}
	\prod_p
	\left(
	1- \frac{4}{p^2}
	+ \frac{4}{p^3}
	- \frac{1}{p^4}
	\right) 
	X \log^4 X,
\end{equation} 
as $X\to \infty$. More generally, the \textit{additive divisor correlation problem} asks for an asymptotic of the form
\begin{equation} \label{eq:correlationProblem}
	\sum_{n \le X} \tau_\ell(n) \tau_k(n+1)
	= 
	M_{\ell, k}(X) + E_{\ell, k}(X),
\end{equation}
where $M_{\ell,k}(X)$ is a main term of order exactly $X (\log X)^{\ell + k - 2}$ and $E_{\ell, k}(X)$ is an error term of order strictly smaller than $M_{\ell, k}(X)$. In Table \ref{table:1} we summarize results on the error term $E_{\ell,k}(X)$ for various $\ell$ and $k$. 

\begin{table}[h!]
	\caption{Progress on the error term $E_{\ell,k}(X)$ in the asymptotic $\sum_{n \le X} \tau_\ell(n) \tau_k(n+1) = M_{\ell, k}(X) + E_{\ell, k}(X)$, as $X\to \infty$, where $E_{\ell, k}(X)$ is of order strictly smaller than $X (\log X)^{\ell + k - 2}$.}
	\label{table:1}
	\begin{tabular}{ l l l l}
		\hline \hline
		$\ell$ & $k$ & References & $E_{\ell,k}(X)$\\
		\hline \hline
		$2$& $2$ & Ingham \cite[(8.5) p. 205]{Ingham1927} (1927) & $\ll X \log X$ \\
		& & Estermann \cite[p. 173]{Estermann1931} (1931) & $\ll X^{11/12} (\log X)^{17/6}$\\
		& & Heath-Brown \cite[Theorem 2, p. 387]{HeathBrown1979} (1979) & $\ll X^{5/6+\epsilon}$\\
		& & Deshouillers \& Iwaniec \cite[Theorem, p. 2]{DeshouillersIwaniec1982} (1982) & $\ll X^{2/3+\epsilon}$\\
		\hline
		$2$ & $3$ & Hooley \cite[Theorem 1, p. 412]{Hooley1957} (1957) & $\ll X (\log X \log \log X)^2$\\
		& & Friedlander \& Iwaniec \cite[p. 320]{FriedlanderIwanice1985} (1985) & $\ll X^{1-\delta}\ (\delta >0)$\\
		& & Heath-Brown \cite[Theorem 3, p. 32]{HeathBrown1986} (1986) & $\ll X^{1-\frac{1}{102}+\epsilon}$\\
		& & Bykovskii, Vinogradov \cite[p. 3004]{BykovskiiVinogradov1987} (1987) & $\ll X^{8/9 + \epsilon}$  \\
		\hline
		$2$ & $\ge 4$ & Linnik 
		\cite[Teopema 3, p. 961]{Linnik1958}
		\cite{Linnik1958} (1958)  & $\ll X (\log X)^{k - 1} (\log \log X)^4$
		\\
		& & Bredikhin \cite[Teopema, p. 778]{Bredikhin1963} (1963) & $\ll X (\log X)^{k - 1} (\log \log X)^4$\\
		& & Motohashi \cite[Theorem 1, p. 43]{Motohashi1980} (1980) & $\ll X (\log \log X)^{c(k)} (\log X)^{-1}$\\
		& & Fouvry, Tenenbaum \cite[Theoreme 1, p. 44]{FouvryTenenbaum1985} (1985) & $\ll X \exp{(-c(k) (\log X)^{1/2})}$\\
		& & Bykovskii, Vinogradov \cite[p. 3004]{BykovskiiVinogradov1987} (1987) & $\ll X^{1 - \frac{1}{2k} + \epsilon}$ \\
		& & Drappeau \cite[Theorem 1.5, p. 687]{Drappeau2017} (2017) & $\ll X^{1-\delta/k}\ (\delta > 0)$\\
		& & Topacogullari \cite[Theorem 1.1, p. 7682]{Topacogullari2018} (2018) & $\ll X^{1-\frac{4}{15k-9} + \epsilon} + X^{1- \frac{1}{57} + \epsilon}$
		\\
		\hline \hline
		$3$ & $3$ & Open--no unconditional bound on $E_{\ell,k}(X)$ is known.& \\
		\hline \hline 
	\end{tabular}
\end{table}

An approach to the shifted convolution $\tau_k(n) \tau_\ell(n+h)$ is through what is called a ``level of distribution". It is a folklore conjecture that $\tau_k(n)$ all have a level of distribution up to $1-\epsilon$, for any $\epsilon>0$. Some known level, or exponent, of distribution for $\tau_k(n)$ was summarized in \cite[Table 1, p. 33]{Nguyen2021}. One of the purposes of this paper is to provide a conditional proof for the full asymptotic expansion for \eqref{eq:CorrelationSum}, on assuming the following upper bound for the averaged level of distribution of $\tau_3(n)$ in arithmetic progressions up to level $2/3$ for $k=\ell=3$, and to indicate the barrier in the additive divisor correlation problem. This obstacle is summarized in the following

\begin{conjecture} \label{conj:LevelofDistribution}
	Let $\epsilon>0$. Then, for any $k\ge 1$, we have, uniformly in $1\le h\le X^{\frac{k-1}{k}}$, the upper bound
	\begin{equation} \label{eq:412}
		\sum_{q \le X^{\frac{k-1}{k}}}
		\left|
		\sum_{\substack{n\le X\\ n \equiv h (\bmod q)}} \tau_k(n)
		- \frac{1}{\varphi\left( \frac{q}{(h, q)} \right)}
		\sum_{\substack{n\le X\\ \left(n, \frac{q}{(h, q)} \right) =1}} \tau_k(n)
		\right|
		\ll_\epsilon
		X^{\frac{1}{2} + \epsilon},
	\end{equation}
	as $X\to \infty$, where the implied constant is independent of $h$ and only depends on $\epsilon$.
\end{conjecture}

\begin{remark}
	Numerical evidence for this conjectural upper bound is provided in the last section, where we numerically determine an upper bound for the exponent of the error term and also the size of the implied constant for the two error terms $E_{3,3}(X,1)$ and $E_{2,2}(X,1)$.
\end{remark}

Our first result gives the full main term for the shifted convolution $D_{3,3}(X,1)$, on assuming a special case of this conjecture.

\setcounter{thm}{0}
\begin{thm} \label{thm:D33XhFullPoly}
	Assume Conjecture \ref{conj:LevelofDistribution} for $k=3$. 
	Let $D_{3,3}(X,h)$ be defined as in \eqref{eq:DklXh}.
	Let $\epsilon>0$. We have, for any composite shift $1 \le h \le X^{2/3}$,
	\begin{equation} \label{eq:D33b}
		D_{3,3}(X,h) 
		= M_{3,3}(X,h) +E_{3,3}(X,h),
		\quad
		(\text{as } X\to \infty),
	\end{equation}
	where
	\begin{align} \label{eq:M33}
		\quad \quad 
		&M_{3,3}(X,h)
		\\ & \quad =
		3\underset{\substack{s=1\\ w_1=w_2=0}}{\mathrm{Res}}
		\left(
		\frac{X^{\frac{1}{3}(w_1+2w_2+3s)}}{s w_1w_2}
		\zeta^3(s)
		\zeta(w_1+w_2+1)
		\zeta(w_2+1)
		A_1(s,w_1,w_2)
		\right)
		\\& \quad  \quad 
		- 3
		\underset{\substack{s=1\\ w_2=1, w_1=0}}{\mathrm{Res}}
		\left(
		\frac{X^{ \frac{1}{3}(w_1+2w_2+s)}}{s w_1w_2}
		\zeta^3(s)
		\zeta(w_1+w_2+1-s)
		\zeta(w_2+1-s)
		A_2(s,w_1,w_2)
		\right)
		\\& \quad \quad \quad 
		+
		\underset{\substack{s=1\\ w_1=w_2=1}}{\mathrm{Res}}
		\left(
		\frac{X^{\frac{1}{3} (w_1+w_2+s)}}{s w_1w_2}
		\zeta^3(s)
		\zeta(w_1+1-s)
		\zeta(w_2+1-s)
		A_3(s,w_1,w_2)
		\right)
		\\& \quad \quad \quad \quad 
		+ O(X^{0.897}),
	\end{align}
	with
	\begin{align} \label{eq:A1}
		\quad\quad
		&A_1(s,w_1,w_2)
		=
		\prod_p
		\left(1-\frac{1}{p^{w_1+w_2+1}}\right)
		\left(1-\frac{1}{p^{w_2+1}}\right)
		\\ & \quad
		\times
		\left(
		1
		+
		\frac{\left(1-\frac{1}{p^s}\right)^3}{1-\frac{1}{p}}
		\left(
		\frac{1}{p^{w_1+w_2+1}-1}
		+
		\frac{1}{p^{w_2+1}-1}
		+
		\frac{1}{(p^{w_1+w_2+1}-1)(p^{w_2+1}-1)}
		\right)
		\right),
	\end{align}
	\begin{align} \label{eq:A2}
		A_2&(s,w_1,w_2)
		=
		\prod_p
		\left(1-\frac{1}{p^{w_1+w_2+1-s}}\right)
		\left(1-\frac{1}{p^{w_2+1-s}}\right)
		\\ &
		\times
		\left(
		1
		+
		\frac{\left(1-\frac{1}{p^s}\right)^3}{1-\frac{1}{p}}
		\left(
		\frac{1}{p^{w_1+w_2+1-s}-1}
		+
		\frac{1}{p^{w_2+1-s}-1}
		+
		\frac{1}{(p^{w_1+w_2+1-s}-1)(p^{w_2+1-s}-1)}
		\right)
		\right),
	\end{align}
	and
	\begin{align} \label{eq:A3}
		\quad \quad
		A_3&(s,w_1,w_2)
		=
		\prod_p
		\left(1-\frac{1}{p^{w_1+1-s}}\right)
		\left(1-\frac{1}{p^{w_2+1-s}}\right)
		\\ &
		\times
		\left(
		1
		+
		\frac{\left(1-\frac{1}{p^s}\right)^3}{1-\frac{1}{p}}
		\left(
		\frac{1}{p^{w_1+1-s}-1}
		+
		\frac{1}{p^{w_2+1-s}-1}
		+
		\frac{1}{(p^{w_1+1-s}-1)(p^{w_2+1-s}-1)}
		\right)
		\right),
	\end{align}
	and the error term satisfies
	\begin{equation}
		E_{3,3}(X,1)
		\ll_\epsilon
		X^{\frac{1}{2}+\epsilon}.
	\end{equation}
	The functions $\zeta^3(s) A_1(s,w_1,w_2)$, $\zeta^3(s) A_2(s,w_1,w_2)$, and $\zeta^3(s) A_3(s,w_1,w_2)$ are analytic in the wider regions 
	\begin{align} \label{eq:niceregions}
		&\Re(s) > 1/2,\ \Re(w_2) > -1/2,\ \text{and } \Re(w_1) > -1/2 - \Re(w_2);
		\\&
		\Re(s) > 1/2,\ \Re(w_2) > \Re(s) - 1/2,\ \text{and } \Re(w_1) > \Re(s) - \Re(w_2) - 1/2;
		\\&
		\Re(s), \Re(w_1), \Re(w_2) > 1/2,
	\end{align} 
	respectively.
\end{thm}

\begin{remark}
	Our method applies equally to correlations between the von Mangoldt function $\Lambda(n)$ and $\tau_k(n)$ of the form
	\begin{equation} \label{eq:Elliott–Halberstam}
		P_k(X,h)=
		\sum_{n\le X}
		\tau_k(n)
		\Lambda (n+h).
	\end{equation}
	In particular, by assuming the Elliott-Halberstam Conjecture for $\Lambda(n)$, the full main-term for the prime correlation \eqref{eq:Elliott–Halberstam} can be derived and numerically tested, similar to the case for $D_{3,3}(X,1)$ and $D_{2,2}(X,1)$ demonstrated here. In this sense, Conjecture \ref{conj:LevelofDistribution} can be seen as an Elliott-Halberstam Conjecture, but for the $k$-fold divisor function $\tau_k(n)$.
\end{remark}

\begin{remark}
	The error term in \eqref{eq:M33} could likely be improved by using smooth weights. However, due to the regions \eqref{eq:niceregions} of analyticity of the Euler factors $A_i$, the best error term for the main term \eqref{eq:M33} we seem to get from our method is $O(X^{2/3 + \epsilon})$.
\end{remark}

We give a numerical verification of our prediction \eqref{eq:D33b}, which also seems to suggest squareroot cancellation in the error term. This, in particular, gives the first quantitative confirmation of any prediction on the additive correlation sum $D_{3,3}(X,1)$, as the coefficients of these polynomials are not too easy to compute. The result is

\begin{cor} \label{cor:mypolynomialM33}
	Let $M_{3,3}(X,1)$ be defined by \eqref{eq:M33}.
	Then, we have, with at least sixty-eight digits accuracy in the coefficients,
	\begin{align} \label{eq:m3u1bc}
		&M_{3,3}(X,1)
		=
		X \left(0.054444679154884094580751878529861703282699438750338984412069100
		\right.
		\\& \left. \quad\quad\quad\quad\quad\quad\quad\quad
		88090 66227780631551394813609558909414229584839437008 \log^4 X
		\right.
		\\& \left.
		+ 0.710113929053644747553958926673505372958197119463757504939845715359739 \log ^3 X
		\right.
		\\& \left.
		+ 2.02119605787987777943324240784753809467091508369917789267040603543881 \log ^2 X
		\right.
		\\& \left.
		+ 0.677863310832980388541571083062733656003222322704135348688102425159897 \log X
		\right.
		\\& \left.
		+ 0.287236647746619417221664617814645950166036274397222249618913907447198\right)
		+ O(X^{0.897}).
	\end{align}
\end{cor}

Corollary \ref{cor:mypolynomialM33} is derived from the main term in Theorem \ref{thm:D33XhFullPoly} with the help of Mathematica\footnote{Mathematica files available at \href{https://aimath.org/~dtn/papers/correlations/}{https://aimath.org/$\sim$dtn/papers/correlations/}} to carry out the residues computations. The coefficients of \eqref{eq:m3u1bc} can be computed to any degree of accuracy--see the proof of Corollary \ref{cor:mypolynomialM33} in the Appendix \ref{section:appendix} for more.

A numerical computation provided by B. Conrey shows that, for $X=10^9$, the data
\begin{equation}
	\sum_{n \le 10^9}
	\tau_3(n) \tau_3(n+1)
	= \boldsymbol{17, 243, 3}58, 889, 275
\end{equation}
compares extremely well with the prediction \eqref{eq:m3u1bc}
\begin{equation}
	[ M_{3,1}(10^9,1) ]
	=
	\boldsymbol{17, 243, 3}95, 216, 318,
\end{equation}
with the first 6 of 14 digits match exactly, which is almost half the number of digits. A graphical comparison between the data $D_{3,3}(X,1)$ and our prediction $M_{3,3}(X,1)$ is provided in Figure \ref{figure:allThree}, showing great alignment. 
\begin{figure}
	\includegraphics[scale=0.9]{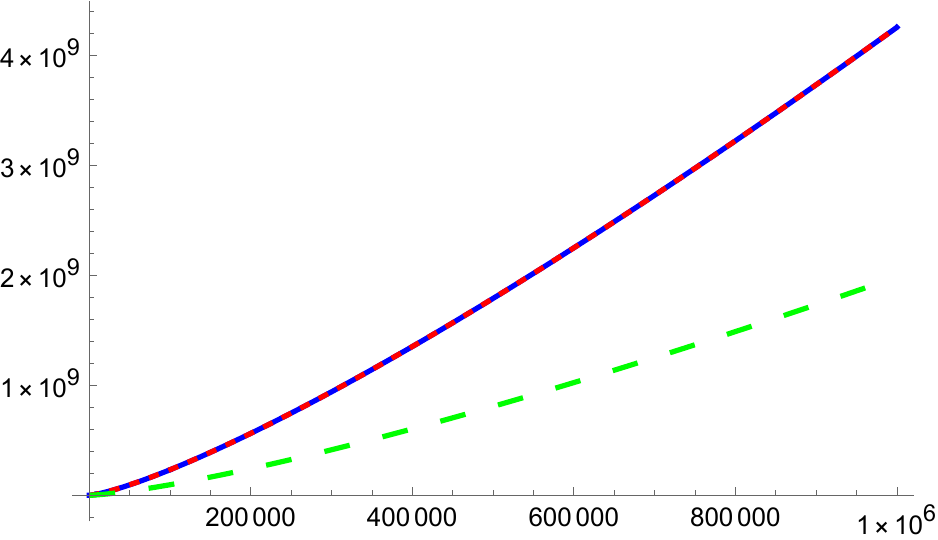}
	\caption{A plot of the three functions $D_{3,3}(X,1)$ in \eqref{eq:DklXh} (solid blue), $M_{3,3}(X,1)$ in \eqref{eq:m3u1bc} (dotted red), and $\frac{1}{4}
		\prod_p
		\left(
		1- \frac{4}{p^2}
		+ \frac{4}{p^3}
		- \frac{1}{p^4}
		\right) 
		X \log^4 X$ (green large dash), for $X \le 10^6$.}
	\label{figure:allThree}
\end{figure}
In Figure \ref{figure:deltamethod}, a plot of the error term $E_{3,3}(X,1) = D_{3,3}(X,1) -  M_{3,3}(X,1)$ is shown, for $X \le 10^6$.
\begin{figure}
	\includegraphics[scale=0.7]{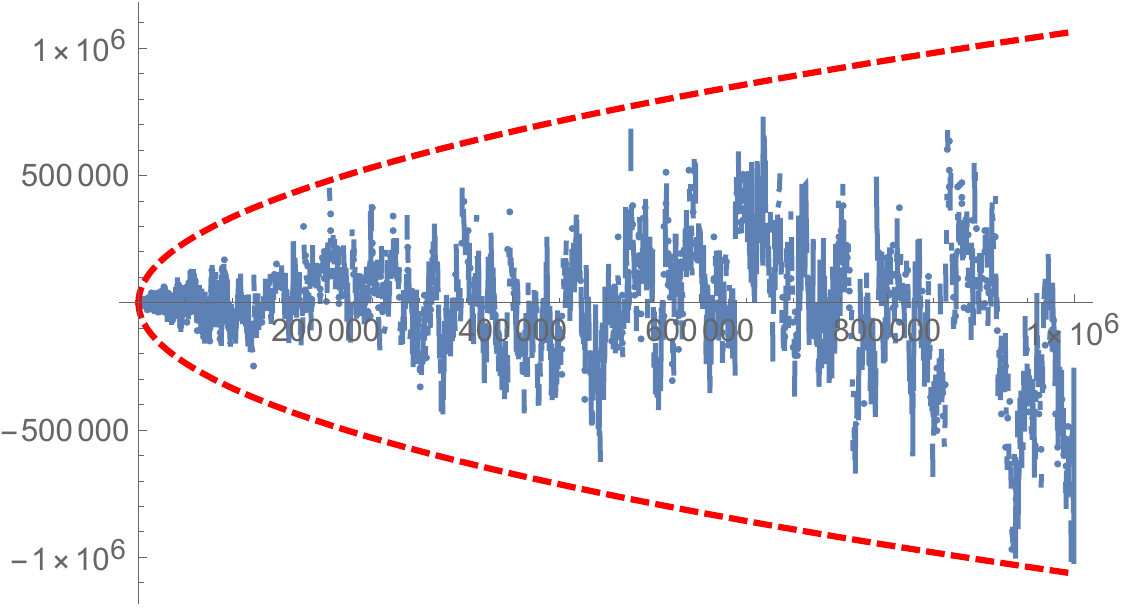}
	\caption{A plot of the error term $E_{3,3}(X,1)$ in \eqref{eq:D33b} in solid blue, and the bounds $\pm 1050 X^{0.51}$ in dashed red, for $X$ up to a million.}
	\label{figure:deltamethod}
\end{figure}

We work out in our next result the leading order main term in $M_{k,\ell}(X,h)$ for any $k,\ell$ and composite shift $h$, and verify, for the special case $k=\ell=3$ and any composite shift $h$, that our answer matches previous computations of Ng and Thom \cite{NgThom2019} and Tao \cite{TaoBlog}. 

\begin{cor} \label{cor:correlation}
	Assume Conjecture \ref{conj:LevelofDistribution} for all $\ell$. 
	Let $D_{k,\ell}(X,h)$ be defined as in \eqref{eq:DklXh}.
	We have, for any $k,\ell\ge 2$ and composite shift $1 \le h \le X^{(\ell-1)/\ell}$,
	\begin{equation} \label{eq:CorrelationSumforhPrime}
		D_{k,\ell}(X,h)
		\sim 
		\frac{C_{k,\ell}
			f_{k,\ell}(h)}{(k-1)! (\ell-1)!}
		X (\log X)^{k+\ell-2},
	\end{equation}
	where 
	\begin{equation}
		C_{k,\ell}
		= 
		\prod_p
		\left(
		\left(
		1 - \frac{1}{p}
		\right)^{k-1}
		+ 
		\left(
		1 - \frac{1}{p}
		\right)^{\ell-1}
		-
		\left(
		1 - \frac{1}{p}
		\right)^{k+\ell-2}
		\right),
	\end{equation}
	and $f_{k,\ell}(h)$ is given by equation \eqref{eq:fklh} below.
	
	In particular, for $k=\ell=3$ and any $1  \le h \le X^{2/3}$, we have
	\begin{equation} \label{eq:leadingTermAsymptotic}
		\sum_{n \le X} \tau_3(n) \tau_3(n+h)
		\sim 
		\frac{1}{4}
		\prod_p
		\left(
		1- \frac{4}{p^2}
		+ \frac{4}{p^3}
		- \frac{1}{p^4}
		\right) 
		f_{3,3}(h)
		X \log^4 X,
	\end{equation} 
	where
	\begin{align} \label{eq:f33}
		f_{3,3}(h)
		=
		&\prod_{p \mid h}
		\left(
		-\nu_p(h)^2 (p-1)^2 (p+1)+p^{\nu_p(h)+2}+4 p^{\nu_p(h)+3}
		\right.
		\\ & \left. \quad
		+p^{\nu_p(h)+4}
		+\nu_p(h) \left(-4 p^3+6 p-2\right)
		-4 p^3-5 p^2+4 p-1
		\right)
		\\ &  \quad \quad
		/
		\left(
		p^{\nu_p(h)}(p-1)^2 \left(p^2+2 p-1\right)
		\right),
	\end{align}
	with $\nu_p(h)$ the highest power of $p$ that divides $h$.
\end{cor}

We expect that our answers \eqref{eq:CorrelationSumforhPrime} also agree for all $k,\ell$ and composite shifts $h$. We are unable to show that uniformly at the moment, but we give an algorithm to check it case by case.

\begin{remark}
	The conditional asymptotic \eqref{eq:leadingTermAsymptotic} confirms a recent Conjecture in \cite[Conjecture, page 35]{NgThom2019} for $k=\ell=3$ and $1\le h\le X^{2/3}$.
\end{remark}

Corollary \ref{cor:correlation} above is derived from assuming Conjecture \ref{conj:LevelofDistribution} together with the following unconditional
\setcounter{thm}{1}
\begin{thm} \label{thm:2b}
	For $k, \ell \ge 1$ and $h$ any composite number, we have
\begin{align} \label{eq:h=prime}
	\sum_{\ell_1 \le X^{1/k}}
	&\sum_{\ell_2 \le \frac{X^{(k-1)/k}}{\ell_1}}
	\cdots
	\sum_{\substack{\ell_{k-1} \le \frac{X^{(k-1)/k}}{\ell_1 \cdots \ell_{k-2}}}} 
	\sum_{\ell_k \le \frac{X}{\ell_1 \cdots \ell_{k-1}}}
	\frac{1}{\varphi\left(q_1\right)}
	\underset{\substack{s=1}}{\mathrm{Res}}
	\left(
	\frac{(X/\delta)^s}{s}
	\sum_{\left( n, q_1 = 1 \right)}
	\frac{\tau_\ell (n \delta )}{n^s}
	\right)
	\\& \sim
	\frac{C_{k,\ell} f_{k,\ell}(h) }{k! (\ell-1)!}
	X \log^{k+\ell-2} X,
	\quad
	(X\to \infty),
\end{align}
where $q = \ell_1 \cdots \ell_{k-1}$, $\delta = (h, q)$, and $q_1 = q/\delta$.
\end{thm}
We give an elementary proof, essentially, for \eqref{eq:h=prime} for the special case $k=\ell=3$ and $h=1$ in Section \ref{section:h=1}. For the general situation $k, \ell \ge 1$ and $h>1$, it turns out to be more robust to use generating functions, which we do in Section \ref{section:compositeh}.

For comparison with our method, in Section \ref{section:ConreyGonek}, we explicitly work out all the main terms in full details from a previously conjectured formula of Conrey and Gonek \cite[Conjecture 3]{ConreyGonnek2002} for the specific case $k=3$ and $h=1$, showing complete agreement in our answers to at least 68 digits down to the constant term. This is
\setcounter{thm}{2}
\begin{thm} \label{thm:deltamethodb}
	Let $\epsilon>0$. Let $m_3(X,1)$ be defined via the delta method by \eqref{eq:m3prime}. Then, we have, as $X\to \infty$, with at least 71 digits accuracy in the coefficients,
	\begin{align} \label{eq:m3u1b}
		&m_3(X,1)
		\\&=
		0.05444467915488409458075187852986170328269943875033898441206910088090
		\\& \qquad
		66227780631551394813609558909414229584839437008 X \log ^4(X)
		\\&\quad
		+0.710113929053644747553958926673505372958197119463757504939845715359
		\\& \qquad \quad
		739076661971842253983213149206 X \log ^3(X)
		\\&\quad \quad +2.0211960578798777794332424078475380946709150836991778926704060354
		\\& \qquad \quad \quad
		3880 548628848354775122568369734 X \log ^2(X)
		\\&\quad \quad \quad
		+0.67786331083298038854157108306273365600322232270413534868810242
		\\& \qquad \quad \quad \quad
		515989 727867201461267995359769 X \log (X)
		\\& \quad \quad \quad \quad
		+ 0.287236647746619417221664617814645950166036274397222249618913
		\\& \qquad \quad \quad \quad \quad
		90744731 664345218868780687078219 X
		+ O(X^\epsilon).
	\end{align}
\end{thm}

In the last Section \ref{section:D22Xh}, we provide further numerical evidence for Conjecture \ref{conj:LevelofDistribution} for the case $k=2$. More precisely, we refine an unconditional result of Heath-Brown \cite[Theorem 2]{HeathBrown1979} on the shifted correlation $D_{2,2}(X,h)$ of the usual divisor function, giving
\setcounter{thm}{3}
\begin{thm} 
	Let $\epsilon>0$. We have, uniformly for all $1\le h\le X^{1/2}$, the asymptotic equality
	\begin{equation} \label{eq:D22Xh}
		\sum_{n\le X}
		\tau(n) \tau(n+h)
		=
		M_{2,2}(X,h)
		+ 
		E_{2,2}(X,h),
	\end{equation}
	where
	\begin{equation}
		M_{2,2}(X,h)
		=X \left(
		c_2(h) \log^2 X
		+ c_1(h) \log X
		+ c_0(h)
		\right),
	\end{equation}
	with
	\begin{equation}
		c_2(h)
		=
		\frac{6}{\pi^2} \sum_{d \mid h} \frac{1}{d},
	\end{equation}
	\begin{equation}
		c_1(h)
		=
		(4 \gamma - 2) f_h(1,0)
		+
		2 f_h^{(0,1)}(1,0)+f_h^{(1,0)}(1,0),
	\end{equation}
	and
	\begin{align}
		c_0(h)
		&=
		2 \left(-f_h^{(0,1)}(1,0)+\gamma  \left(2 f_h^{(0,1)}(1,0)+f_h^{(1,0)}(1,0)-f_h(1,0)\right)+f_h^{(1,1)}(1,0)+2 \gamma ^2 f_h(1,0)\right)
		\\&  \quad
		+ 
		f_h^{(1,0)}(1,0)
		+2 (\gamma -1) f_h(1,0),
	\end{align}
	with the constants $f_h$, $f_h^{(0,1)}$, $f_h^{(1,0)}$, and $f_h^{(1,1)}$ at $(1,0)$ depending only on $h$ given in Lemmas \ref{lemma:fh10} and \ref{lemma:fhDerivatives}, and with the error term satisfying
	\begin{equation}
		E_{2,2}(X,h)
		\ll_\epsilon
		X^{5/6+\epsilon}.
	\end{equation}
\end{thm}
As a consequence of this result, we obtain the following
\begin{cor} \label{cor:3b}
	We have, for any $\epsilon>0$, with at least 148 digits accuracy in the coefficients,
	\begin{align}
		&M_{2,2}(X,1)
		=
		X \left(\frac{6}{\pi ^2}\log ^2(X)
		\right.
		\\& \left. 
		+1.5737449203324910789070569280484417010544014980534581993991047787172106559673
		\right.
		\\& \left. \quad
		1173018329789033856157663793482022187619702084359231966550508901828044158 \log (X)
		\right.
		\\& \left. 
		-0.5243838319228249988207213304174247109766097340170991428485246582967458363611
		\right.
		\\& \left. \quad
		4606090215515124475866524185215534024889460792901985996741204565400064583\right)
		+ O(X^\epsilon).
	\end{align}
\end{cor}
For example, our $M_{2,2}(X, 1)$ given above for the main term of $D_{2,2}(X, 1)$ for $X=20, 220, 000$ yields
	\begin{equation}
		M_{2,2}(20.22 \times 10^6, 1)
		\approx
		\textbf{4, 003, 240}, 490,
	\end{equation}
	which is just 25 parts-per-billion of the answer
	\begin{equation} \label{eq:2022}
		\sum_{n\le 20, 220, 000}
		\tau(n) \tau(n+1)
		= \textbf{4, 003, 240}, 588;
	\end{equation}
	whereas the corresponding leading order asymptotic
	\begin{equation}
		\frac{6}{\pi^2}
		(20, 220, 000)
		\log^2 (20, 220, 000)
		\approx
		3, 478, 542, 795
	\end{equation}
is far from \eqref{eq:2022}. 

A graph of the error term $E_{2,2}(X,1)$ is plotted in Figure \ref{figure:D22X1}. In Figure \ref{figure:D22X1LogLogPlot}, a log-log-plot of this error term is shown, numerically suggesting that this error is bounded by $|E_{2,2}(X,1)| \le 7 X^{0.51}$, which is in favor of the conjectural bound \eqref{eq:412}.
\begin{figure}
	\includegraphics[scale=1]{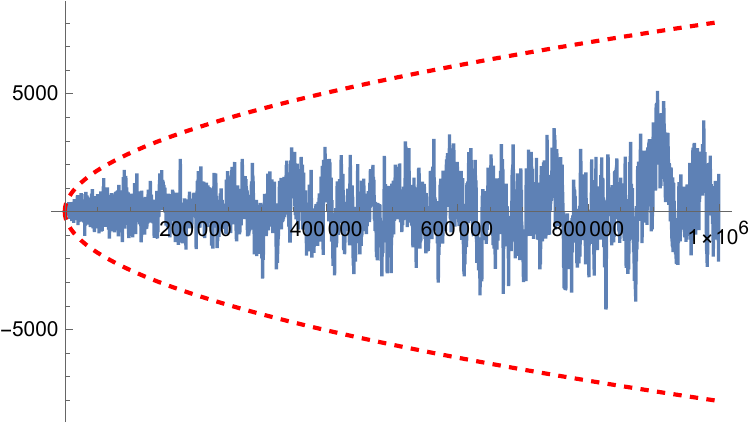}
	\caption{A plot of the error term $E_{2,2}(X,1)$ in solid blue, and $\pm 7 X^{0.51}$ in dashed red, for $X$ up to one million.}
	\label{figure:D22X1}
\end{figure}

\begin{remark}
	Unconditional lower bounds for the additive divisor sum $D_{k,\ell}(X,h)$ have been sharpened from Ng and Thom \cite{NgThom2019} by Andrade and Smith \cite{AndradeSmith2019}, who approximate, in our notation, the general divisor function $\tau_k(n)$ by partial divisor functions
	\begin{equation} \label{eq:partialdivisorfunctions}
		\tau_\ell(n,A)
		=
		\sum_{q\mid n: q\le n^A}
		\tau_{\ell-1}(q)
	\end{equation}
	parametrized by $A \in (0,1]$.
\end{remark}

\begin{remark}
	A similar quantity to the left side of \eqref{eq:412} was investigated for a special set of moduli $d=rq$ in \cite[Theorem 1, p. 35]{Nguyen2021} using the method of \cite{Zhang2014} with $d < X^{\frac{1}{2}+ \frac{1}{584}}$ for a fixed residue class $n\equiv h (d)$. This is one approach towards bounding this error term $E_{\ell, k}(X)$--maybe a weaker form of \eqref{eq:412} is sufficient for certain applications.
\end{remark}
\begin{remark}
	It would be interesting to also sum over $h$ and investigate the variance of divisor sums, such as
	\begin{equation}
		\sum_{h\le H} 
		\left|
		\sum_{n\le X} \tau_3(n) \tau_3(n+h)
		-
		M_{3,3}(X,h)
		\right|^2,
	\end{equation}
	with $M_{3,3}(X,h)$ given more precisely by \eqref{eq:1156} below and with $H=X^c$ for various ranges of $c \in (0,1]$. An analogous variance, but of the $k$-fold divisor function in arithmetic progressions, was studied by the author in \cite{Nguyen2022}.
\end{remark}

In summary, we collect in Table \ref{table:2} the conditional and unconditional results of this paper and where to find their proofs.
\begin{table}[h!]
	\caption{Summary of results and their proofs}
	\label{table:2}
	\begin{tabular}{ l l l l}
		\hline \hline
		Conditional results 
		& \vline\
		Proves in
		& \vline\ \vline \quad
		Unconditional results
		& \vline \
		Proves in
		\\ \hline \hline
		Theorem \ref{thm:D33XhFullPoly}
		& \vline\
		Section \ref{sectiion:D33XhFullPoly}
		& \vline\ \vline\quad
		Theorem \ref{thm:2b}
		& \vline\
		Section \ref{section:compositeh}
		\\
		Corollary \ref{cor:correlation}
		& \vline\
		Sections \ref{section:h=1}
		and \ref{sec:hcomposite}
		& \vline\ \vline\quad
		Theorem \ref{thm:deltamethodb}
		& \vline \
		Section \ref{section:ConreyGonek}
		\\
		Proposition \ref{prop:Sigma1111}
		& \vline \
		Section \ref{section:sigma11}
		&\vline\ \vline\quad
		Theorem \ref{theorem:D22Xh}
		& \vline\
		Section \ref{section:D22Xh}
		\\ 
		& \vline \
		&\vline\ \vline\quad
		Corollary \ref{cor:mypolynomialM33}
				& \vline\
				Appendix Mathematica
		\\
		& \vline \
		&\vline\ \vline\quad
		Corollary \ref{cor:3b}
		& \vline\
		Section \ref{section:D22Xh}
		\\
		& \vline \
		&\vline\ \vline\quad
				Proposition \ref{prop:1} 
		& \vline \ 
		Section \ref{sectiion:D33XhFullPoly}
		\\
		& \vline \
		&\vline\ \vline\quad
				Proposition \ref{prop:sigma21sigma31} 
		& \vline \ 
		Section \ref{section:4.2}
		\\
		\hline \hline
	\end{tabular}
\end{table}

\section{Lemmata}

We start by first generalizing a combinatorial Lemma of Hooley \cite[Lemma 4, p. 405]{Hooley1957} for $\tau_k(n)$.

\begin{lem}
	\label{lemma:keyDecompositionk}
	For any $n\le X$, we have
	\begin{equation} \label{eq:keyDecompositionk}
		\tau_k(n) = k \Sigma_k(n) 
		+
		O(E(n)),
	\end{equation}
	where
	\begin{align} 
		\Sigma_k(n)
		= 
		\sum_{\substack{\ell_1 \ell_2 \cdots \ell_k =n\\ \ell_1\ell_2 \cdots \ell_{k-1} \le X^{(k-1)/k};\ \ell_1 \le X^{1/k}}} 1
	\end{align}
	and
	\begin{equation}
		E(n)
		=
		\sum_{\substack{\ell_1 \ell_2 \cdots \ell_k =n\\ \ell_1\ell_2 \cdots \ell_{k-1} \le X^{(k-1)/k};\ \ell_k \le X^{1/k}}} 1.
	\end{equation}
\end{lem}
\begin{proof}
	This follows from the identity
	\begin{align}
		\sum_{\substack{\ell_1 \cdots \ell_k = n\\ \ell_1, \cdots, \ell_k \le X^{1/k}}}
		&=
		\sum_{\substack{\ell_1 \cdots \ell_k = n}} 1
		-
		\sum_{1 \le i \le k}
		\sum_{\substack{\ell_1 \cdots \ell_k = n\\ \ell_i > X^{1/k}}} 1
		+
		\sum_{1 \le i_1 < i_2 \le k}
		\sum_{\substack{\ell_1 \cdots \ell_k = n\\ \ell_{i_1}, \ell_{i_2} > X^{1/k}}} 1
		+ 
		\cdots
		\\&
		+
		(-1)^j
		\sum_{1 \le i_{1} < \cdots < i_{j} \le k}
		\sum_{\substack{\ell_1 \cdots \ell_k = n\\ \ell_{i_1}, \cdots, \ell_{i_j} > X^{1/k}}} 1
		+ \cdots
		+
		(-1)^k
		\sum_{\substack{\ell_1 \cdots \ell_k = n\\ \ell_{1}, \cdots, \ell_{k} > X^{1/k}}} 1.
	\end{align}
\end{proof}

\begin{lem} \label{lemma:tauk(nh)}
	For any $h\ge 1$, we have
	\begin{equation} \label{eq:tau_k(nh)}
		\sum_{n=1}^\infty
		\frac{\tau_k(nh)}{n^s}
		=
		\zeta^k(s)
		A_h(s),\ (\sigma >1),
	\end{equation}
	where
	\begin{equation} \label{eq:A_h}
		A_h(s)
		=
		\prod_{p\mid h}
		\left(1-\dfrac{1}{p^s}\right)^{k}
		\binom{k+\nu_p(h) - 1}{k-1} {}_2 F_1(1,k+\nu_p(h); 1 + \nu_p(h) ; p^{-s}),
	\end{equation}
	where ${}_2 F_1$ is a hypergeometric function.
\end{lem}
\begin{proof}
	By multiplicativity and Euler products, we have
	\begin{align}
		\sum_{n=1}^\infty
		\frac{\tau_k(nh)}{n^s}
		&=
		\prod_{p\mid h}
		\left(
		\sum_{j=0}^\infty
		\frac{\tau_k(p^{j+\nu_p(h)})}{p^{js}}
		\right)
		\prod_{p\nmid h}
		\left(
		\sum_{j=0}^\infty
		\frac{\tau_k(p^{j})}{p^{js}}
		\right)
		\\&=
		\prod_{p\mid h}
		\frac{\displaystyle\sum_{j=0}^\infty
			\binom{k+j+\nu_p(h)-1}{k-1}
			\frac{1}{p^{js}}}{
			\displaystyle\sum_{j=0}^\infty
			\frac{\tau_k(p^{j})}{p^{js}}}
		\prod_{p}
		\left(
		\sum_{j=0}^\infty
		\frac{\tau_k(p^{j})}{p^{js}}
		\right).
	\end{align}
	By a hypergeometric relation, we have
	\begin{equation}
		\sum_{j=0}^\infty
		\binom{k+j+\nu_p(h)-1}{k-1}
		\frac{1}{p^{js}}
		=
		\binom{k+\nu_p(h) - 1}{k-1} 
		{}_2 F_1(1,k+\nu_p(h); 1 + \nu_p(h) ; p^{-s}).
	\end{equation}
	This, together with
	\begin{equation}
		\prod_{p}
		\left(
		\sum_{j=0}^\infty
		\frac{\tau_k(p^{j})}{p^{js}}
		\right)
		=
		\zeta^k(s),
	\end{equation}
	give \eqref{eq:tau_k(nh)}.
\end{proof}

\begin{lem} \label{lemma:tauknhcoprime}
	For any $h \ge 1$, we have
	\begin{equation}
		\sum_{(n,h)=1}
		\frac{\tau_k(n)}{n^s}
		=
		\zeta^k(s)
		\prod_{p \mid h}
		\left(1-\frac{1}{p^s}\right)^k,\ (\sigma>1).
	\end{equation}
\end{lem}
\begin{proof}
	Going to Euler products gives
	\begin{equation}
		\sum_{(n,h)=1}
		\frac{\tau_k(n)}{n^s}
		=
		\prod_{p \nmid h}
		\sum_{j=0}^\infty
		\frac{\tau_k(p^j)}{p^{js}}
		= \zeta^k(s)
		\prod_{p \mid h}
		\left(1-\frac{1}{p^s}\right)^k.
	\end{equation}
\end{proof}

\section{Full main term for $D_{3,3}(X,h)$: Proof of Theorem \ref{thm:D33XhFullPoly}} \label{sectiion:D33XhFullPoly}

We start with Hooley's identity \eqref{eq:keyDecompositionk} specializing to $k=3$.

\begin{lem}
	\label{lemma:keyDecomposition}
	For any $n\le X$, we have
	\begin{equation} \label{eq:keyDecomposition}
		\tau_3(n) = 3 \Sigma_1(n) - 3 \Sigma_2(n) + \Sigma_3(n),
	\end{equation}
	where
	\begin{align} 
		\Sigma_1(n) \label{eq:Sigma1}
		&= \sum_{\substack{\ell_1 \ell_2 \ell_3 =n\\ \ell_1\ell_2 \le X^{2/3};\ \ell_1 \le X^{1/3}}} 1,\\
		\Sigma_2(n) \label{eq:Sigma2}
		&= \sum_{\substack{\ell_1 \ell_2 \ell_3 =n\\ \ell_1\ell_2 \le X^{2/3};\ \ell_1, \ell_3 \le X^{1/3}}} 1,\\
		\Sigma_3(n)  \label{eq:Sigma3}
		&= \sum_{\substack{\ell_1 \ell_2 \ell_3 =n\\ \ell_1, \ell_2 , \ell_3\le X^{1/3}}} 1.
	\end{align}
\end{lem}

Substituting \eqref{eq:keyDecomposition} in for $\tau_3(n)$ in $D_{3,3}(X,h)$, we have
\begin{align} \label{eq:Sigma123}
	&D_{3,3}(X,h)
	\\&= 3 \sum_{n\le X} \tau_3(n+h) \Sigma_1(n)
	- 3 \sum_{n\le X} \tau_3(n+h) \Sigma_2(n)
	+ \sum_{n\le X} \tau_3(n+h) \Sigma_3(n)\\
	&= 3 \Sigma_{11}(X) - 3 \Sigma_{21}(X) + \Sigma_{31}(X),
\end{align}
say. Interchanging the order of summations in $\Sigma_{11}(X)$, we have
\begin{align}
	\Sigma_{11}(X)
	&= \sum_{\ell_1 \le X^{1/3}}
	\sum_{\ell_2 \le \frac{X^{2/3}}{\ell_1}}
	\sum_{\ell_3 \le \frac{X}{\ell_1 \ell_2}}
	\tau_3(\ell_1 \ell_2 \ell_3 + h).
\end{align}
Making a change of variables in the $\ell_3$ sum, we get
\begin{equation} \label{eq:426}
	\Sigma_{11}(X)
	=
	\sum_{\ell_1 \le X^{1/3}}
	\sum_{\ell_2 \le \frac{X^{2/3}}{\ell_1}}
	\sum_{\substack{n \le X+h\\ n \equiv h (\ell_1 \ell_2)}} \tau_3(n).
\end{equation}
Similarly, we obtain
\begin{equation} \label{eq:1106b}
	\Sigma_{21}(X)
	= \sum_{\ell_1 \le X^{1/3}}
	\sum_{\ell_2 \le \frac{X^{2/3}}{\ell_1}}
	\sum_{\substack{n\le \ell_1 \ell_2 X^{1/3}+h\\ n \equiv h (\ell_1 \ell_2)}}
	\tau_3(n)
\end{equation}
and
\begin{equation} \label{eq:1106bc}
	\Sigma_{31}(X)
	= 
	\sum_{\ell_1 \le X^{1/3}}
	\sum_{\ell_3 \le X^{1/3}}
	\sum_{\substack{n\le \ell_1 \ell_3 X^{1/3} +h\\ n\equiv h (\ell_1 \ell_3)}}
	\tau_3(n).
\end{equation}
We have
\begin{equation}
	\sum_{\substack{n\le Y\\ n\equiv h (q)}}
	\tau_3(n)
	=
	\frac{1}{\varphi(q)}
	\underset{\substack{s=1}}{\mathrm{Res}}
	\frac{Y^s}{s}
	\zeta^3(s)
	f_q(s)
	+ 
	E_3(Y; h, q),
\end{equation}
where
\begin{equation} \label{eq:fqs}
	f_q(s)
	= \prod_{p\mid q}
	\left(1-\frac{1}{p^s}\right)^3
\end{equation}
and, by \eqref{eq:412},
\begin{equation}
	\sum_{q \le Y^{2/3}}
	E_3(Y; h, q)
	\ll Y^{1/2 + \epsilon}.
\end{equation}
Thus, by \eqref{eq:412}, \eqref{eq:426}, \eqref{eq:1106b}, and \eqref{eq:1106bc}, $D_{3,3}(X,h)$ becomes
\begin{align} 
	D_{3,3}(X,h)
	&=
	M_{3,3}(h) +
	O_\epsilon(X^{1/2 + \epsilon}),
\end{align}
where
\begin{align} \label{eq:1156}
	M_{3,3}(h)
	&=
	3
	\underset{\substack{s=1}}{\mathrm{Res}}
	\left(
	\frac{(X+h)^s}{s} \zeta^3(s)
	\sum_{\ell_1 \le X^{1/3}}
	\sum_{\ell_2 \le \frac{X^{2/3}}{\ell_1}}
	\frac{f_{\ell_1 \ell_2} (s)}{\varphi(\ell_1 \ell_2)}
	\right)
	\\& \quad
	- 3
	\underset{\substack{s=1}}{\mathrm{Res}}
	\left(
	\frac{\zeta^3(s)}{s} 
	\sum_{\ell_1 \le X^{1/3}}
	\sum_{\ell_2 \le \frac{X^{2/3}}{\ell_1}}
	\frac{f_{\ell_1 \ell_2} (s)}{\varphi(\ell_1 \ell_2)}
	(\ell_1 \ell_2 X^{1/3} +h)^s
	\right)
	\\& \quad \quad
	+
	\underset{\substack{s=1}}{\mathrm{Res}}
	\left(
	\frac{\zeta^3(s)}{s} 
	\sum_{\ell_1 \le X^{1/3}}
	\sum_{\ell_2 \le X^{1/3}}
	\frac{f_{\ell_1 \ell_2} (s)}{\varphi(\ell_1 \ell_2)}
	(\ell_1 \ell_2 X^{1/3} +h)^s
	\right).
\end{align}
We treat the three double sums from the above by truncated Perron's formula. This involves tedious, but routine, estimates on horizontal and vertical contours, which we provide full details for ease of checking. The procedure is similar for the three, so we show full details only for the first. The result is
\begin{prop} \label{prop:1}
	Let $D_{3,3; a}(X,h)$, $D_{3,3; b}(X,h)$, and $D_{3,3; c}(X,h)$ denote the three quantities on the right side of \eqref{eq:1156}, respectively.
	We have
	\begin{align} \label{eq:520a}
		&D_{3,3; a}(X,h)
		\\ &\quad
		=
		3\underset{\substack{s=1\\ w_1=w_2=0}}{\mathrm{Res}}
		\left(
		\frac{(X+h)^s X^{\frac{1}{3}(w_1+2w_2)}}{s w_1w_2}
		\zeta^3(s)
		\zeta(w_1+w_2+1)
		\zeta(w_2+1)
		A_1(s,w_1,w_2)
		\right)
		\\ & \quad \quad
		+ O(X^{0.897}),
	\end{align}
	\begin{align} \label{eq:520b}
		\\
		&D_{3,3;b}(X,h)
		\\
		&\quad =
		- 3
		\underset{\substack{s=1\\ w_2=1, w_1=0}}{\mathrm{Res}}
		\left(
		\frac{X^{ \frac{1}{3}(w_1+2w_2+s)}}{s w_1w_2}
		\zeta^3(s)
		\zeta(w_1+w_2+1-s)
		\zeta(w_2+1-s)
		A_2(s,w_1,w_2)
		\right)
		\\& \quad \quad
		+ O(X^{0.692}),
	\end{align}
	and
	\begin{align} \label{eq:520c}
		\\
		&D_{3,3; c}(X,h) 
		\\ & \quad 
		=
		\underset{\substack{s=1\\ w_1=w_2=1}}{\mathrm{Res}}
		\left(
		\frac{X^{\frac{1}{3} (w_1+w_2+s)}}{s w_1w_2}
		\zeta^3(s)
		\zeta(w_1+1-s)
		\zeta(w_2+1-s)
		A_3(s,w_1,w_2)
		\right)
		+ O(X^{5/7 + \epsilon}).
	\end{align}
\end{prop}

\begin{proof}
We first fix a notation. Let $\lambda \ge 0$ be a number such that $\left| \zeta(1/2 + it) \right| \ll (1+|t|)^{\lambda + \epsilon}$ for every $\epsilon>0$. By Weyl's bound we may assume that $\lambda \le 1/6$. By Phragm\'en-Lindel\"of  convexity principle, one has, for $1/2 \le \sigma \le 1$ and every $\epsilon >0$, that
\begin{equation}
	\left| \zeta(\sigma + it) \right|
	\ll_\epsilon
	(1+|t|)^{2 \lambda (1-\sigma) + \epsilon},
	\quad
	(1/2 \le \sigma \le 1).
\end{equation}

By multiplicativity and going to Euler products, we have
\begin{equation} \label{eq:1152b}
	\sum_{\ell_1, \ell_2=1}^\infty
	\frac{f_{\ell_1 \ell_2} (s)}{\varphi(\ell_1 \ell_2)}
	\frac{1}{\ell_1^{w_1+w_2}}
	\frac{1}{\ell_2^{w_2}}
	= \prod_p
	\sum_{j_1, j_2}
	\frac{f_{p^{j_1+j_2}}(s)}{\varphi(p^{j_1+j_2})}
	\frac{1}{p^{j_1 (w_1 + w_2) + j_2 w_2}}.
\end{equation}
By \eqref{eq:fqs} and definition of $\varphi(n)$, the $j$'s sums become
\begin{align} \label{eq:1152}
	\\
	\sum_{j_1, j_2}
	&\frac{f_{p^{j_1+j_2}}(s)}{\varphi(p^{j_1+j_2})}
	\frac{1}{p^{j_1 (w_1  + w_2) + j_2 w_2}}
	=
	1
	+ \frac{\left(1-\frac{1}{p^s}\right)^3}{1-\frac{1}{p}}
	\sum_{j_1, j_2\ge 1}
	\frac{1}{p^{j_1( w_1 + w_2)}}
	\frac{1}{p^{ j_2 w_2}}
	\\&
	=
	1
	+ \frac{\left(1-\frac{1}{p^s}\right)^3}{1-\frac{1}{p}}
	\left(
	\frac{1}{p^{w_1 + w_2 +1} - 1}
	+
	\frac{1}{p^{w_2 +1} - 1}
	+
	\frac{1}{(p^{w_1 + w_2 +1} - 1)(p^{w_2 +1} - 1)}
	\right).
\end{align}
Thus, by \eqref{eq:1152b} and \eqref{eq:1152}, we get
\begin{equation} \label{eq:1152c}
	\sum_{\ell_1, \ell_2=1}^\infty
	\frac{f_{\ell_1 \ell_2} (s)}{\varphi(\ell_1 \ell_2)}
	\frac{1}{\ell_1^{w_1+w_2}}
	\frac{1}{\ell_2^{w_2}}
	= \zeta(w_1 + w_2 +1)
	\zeta(w_2 +1)
	A_1(s, w_1, w_2),
\end{equation}
where $A_1(s, w_1, w_2)$ is given as in \eqref{eq:A1}. The function above is analytic in the region
\begin{equation}
	\Re(w_2) > 0
	\text{ and }
	\Re(w_1) > - \Re(w_2),
\end{equation}
with $A_1(s, w_1, w_2)$ analytic in larger regions from \eqref{eq:niceregions}. Hence, by Perron's formula, we have
\begin{align} \label{eq:131}
	&\sum_{\ell_1 \le X^{1/3}}
	\sum_{\ell_2 \le \frac{X^{2/3}}{\ell_1}}
	\frac{f_{\ell_1 \ell_2} (s)}{\varphi(\ell_1 \ell_2)}
	\\& \quad =
	\frac{1}{(2\pi i)^2}
	\int\limits_{\epsilon - iT_1}^{\epsilon + i T_1}
	\int\limits_{\epsilon - iT_2}^{\epsilon + i T_2}
	\frac{X^{w_1/3}}{w_1}
	\frac{X^{2w_2/3}}{w_2}
	\zeta(w_1 + w_2 +1)
	\zeta(w_2 +1)
	A_1(s, w_1, w_2)
	d w_2 dw_1
	\\& \quad \quad
	+ O\left( \frac{X^{\epsilon}}{ T_1 T_2}\right),
\end{align}
for parameters $T_1$ and $T_2$ to be chosen later. 
We shift first the $w_2$ contour in the above left to the vertical segment from $\sigma_2 - iT_2$ to $\sigma_2 + iT_2$, where $-1/2 < \sigma_2 < 0$ is to be determined. We pick up the residue at $w_2 = 0$, two horizontal contours each of size $\ll X^{\epsilon} T_2^{-1} + X^{2\sigma_2/3} T_2^{-1 + 2 \lambda |\sigma_2| + \epsilon}$, and the left vertical contour at real part $\sigma_2$ of size $\ll X^{2\sigma_2/3} T_2^{2 \lambda |\sigma_2| +\epsilon}$. Since $X^{2\sigma_2/3} T_2^{-1 + 2 \lambda |\sigma_2| + \epsilon} \ll X^{2\sigma_2/3} T_2^{2 \lambda |\sigma_2| +\epsilon}$, we will ignore it. We will also ignore the error term in \eqref{eq:131}, since it is $\ll X^{\epsilon}(T_1^{-1} + T_2^{-1})$. Setting $T_2^{-1} = X^{2\sigma_2/3} T_2^{2 \lambda |\sigma_2|}$, we get
\begin{equation}
	T_2 = X^{2 |\sigma_2|/(3 + 6 \lambda |\sigma_2|)}.
\end{equation}
With this, \eqref{eq:131} becomes
\begin{align} \label{eq:131b}
	&\sum_{\ell_1 \le X^{1/3}}
	\sum_{\ell_2 \le \frac{X^{2/3}}{\ell_1}}
	\frac{f_{\ell_1 \ell_2} (s)}{\varphi(\ell_1 \ell_2)}
	\\& \quad =
	\frac{1}{2\pi i}
	\int\limits_{\epsilon - iT_1}^{\epsilon + i T_1}
	\frac{X^{w_1/3}}{w_1}
	\left[
	\underset{\substack{w_2=0}}{\mathrm{Res}}
	\left(
	\frac{X^{2w_2/3}}{w_2}
	\zeta(w_1 + w_2 +1)
	\zeta(w_2 +1)
	A_1(s, w_1, w_2)
	\right)
	\right.
	\\& \left. \quad\quad\quad\quad\quad\quad\quad\quad\quad\quad
	+ O \left( X^{- 2 |\sigma_2|/(3 + 6 \lambda |\sigma_2|) +\epsilon}  \right)
	\right]
	dw_1
	\\& \quad =
	\underset{\substack{w_2=0}}{\mathrm{Res}}
	\left(
	\frac{X^{2w_2/3}}{w_2}
	\zeta(w_2 +1)
	\frac{1}{2\pi i}
	\int\limits_{\epsilon - iT_1}^{\epsilon + i T_1}
	\frac{X^{w_1/3}}{w_1}
	\zeta(w_1 + w_2 +1)
	A_1(s, w_1, w_2)
	dw_1
	\right)
	\\& \quad \quad\quad
	+ O \left( X^{- 2 |\sigma_2|/(3 + 6 \lambda |\sigma_2|) +\epsilon} T_1^\epsilon  \right).
\end{align}
Similarly, we now shift the remaining $w_1$ contour in the above left to the vertical segment from $\sigma_1 + \epsilon - iT_1$ to $\sigma_1 + \epsilon + iT_1$, with $\sigma_1 = - \sigma_2 - 1/2$, so that $\sigma_1 + \epsilon + \sigma_2 + 1 = 1/2 + \epsilon$. We pick up the residue at $w_1 = 0$, two horizontal contours each of size $\ll X^\epsilon T_1^{-1} + X^{(-\sigma_2 - 1/2)/3+\epsilon} T_1^{-1 + \lambda + \epsilon}$ and the left vertical contour at real part $\sigma_1 +\epsilon$ of size $\ll X^{(-\sigma_2 - 1/2)/3+\epsilon} T_1^{ \lambda + \epsilon}$. As before, ignoring the second error term and setting $T_1^{-1} = X^{(-\sigma_2 - 1/2)/3+\epsilon} T_1^{ \lambda }$, we obtain
\begin{equation}
	T_1 = X^{(1-2 |\sigma_2|)/(6+6\lambda)}.
\end{equation}
With this, \eqref{eq:131b} becomes
\begin{align} \label{eq:131c}
	\quad \quad
	&\sum_{\ell_1 \le X^{1/3}}
	\sum_{\ell_2 \le \frac{X^{2/3}}{\ell_1}}
	\frac{f_{\ell_1 \ell_2} (s)}{\varphi(\ell_1 \ell_2)}
	=
	\underset{\substack{w_1=w_2=0}}{\mathrm{Res}}
	\left(
	\frac{X^{2w_2/3}}{w_2}
	\zeta(w_2 +1)
	\frac{X^{w_1/3}}{w_1}
	\zeta(w_1 + w_2 +1)
	A_1(s, w_1, w_2)
	\right)
	\\& \quad
	+ O \left( X^{- 2 |\sigma_2|/(3 + 6 \lambda |\sigma_2|) - (1-2 |\sigma_2|)/(6+6\lambda) + \epsilon} \right).
\end{align}
Setting $2 |\sigma_2|/(3 + 6 \lambda |\sigma_2|) = (1-2 |\sigma_2|)/(6+6\lambda)$, we get
\begin{equation} \label{eq:126}
	|\sigma_2| = 
	\begin{cases}
	1/6, & \text{if } \lambda=0,
	\\
	\dfrac{\sqrt{\lambda^2+10 \lambda+9}-\lambda-3}{4 \lambda},& \text{if } \lambda \in (0,1/6),
	\\
	0.1553, & \text{if } \lambda=1/6	,
	\end{cases}
\end{equation}
Thus, with the above choice for $\sigma_2$, \eqref{eq:131c} becomes
\begin{align} \label{eq:131d}
	&\sum_{\ell_1 \le X^{1/3}}
	\sum_{\ell_2 \le \frac{X^{2/3}}{\ell_1}}
	\frac{f_{\ell_1 \ell_2} (s)}{\varphi(\ell_1 \ell_2)}
	\\& \quad =
	\underset{\substack{w_1=w_2=0}}{\mathrm{Res}}
	\left(
	\frac{X^{2w_2/3}}{w_2}
	\zeta(w_2 +1)
	\frac{X^{w_1/3}}{w_1}
	\zeta(w_1 + w_2 +1)
	A_1(s, w_1, w_2)
	\right)
	+ O \left( X^{-\nu_\lambda + \epsilon}
	\right),
\end{align}
where
\begin{equation}
	\nu_\lambda = 
	\begin{cases}
	-1/9, & \text{if } \lambda = 0,
	\\
	-
	\dfrac{\sqrt{\lambda^2+10 \lambda+9}-\lambda-3}{6 \lambda}, & \text{if } \lambda \in (0,1/6),
	\\
	-  0.1035 , & \text{if } \lambda = 1/6.
	\end{cases}
\end{equation}
Thus, by \eqref{eq:131d}, the first term on the right side of \eqref{eq:1156} becomes
\begin{align} \label{eq:450}
	&3\underset{\substack{s=1\\ w_1=w_2=0}}{\mathrm{Res}}
	\left(
	\frac{(X+h)^s X^{\frac{1}{3}(w_1+2w_2)}}{s w_1w_2}
	\zeta^3(s)
	\zeta(w_1+w_2+1)
	\zeta(w_2+1)
	A_1(s,w_1,w_2)
	\right)
	\\& \quad + O \left( X^{1-\nu_\lambda + \epsilon}
	\right).
\end{align}
For, e.g, $\lambda = 1/6$, the above error term is $\ll X^{0.897}$. This gives \eqref{eq:520a}. 

To treat the second double sum on the right side of \eqref{eq:1156}, we first break $h$ into three cases: $1 \le h \le X^{1/3}$, $X^{1/3} < h \le X^{2/3}$, and $h > X^{2/3}$, then split up the $\ell_1, \ell_2$ sums according to $\ell_1 \ell_2 \ge h/X^{1/3}$ or $\ell_1 \ell_2 < h/X^{1/3}$.

Case 1: $1 \le h < X^{1/3}$. In this case, there are no $\ell_1, \ell_2\ge 1$ such that $\ell_1 \ell_2 < h/X^{1/3}$ so this possibility does not occur. Thus, $\ell_1 \ell_2 X^{1/3} > h$ for all $\ell_1, \ell_2 \ge 1$, and we have
\begin{align} \label{eq:454}
	(\ell_1 \ell_2 X^{1/3} + h)^s
	&=
	(\ell_1 \ell_2 X^{1/3} + h)^s
	\left( 1 + \frac{h}{\ell_1 \ell_2 X^{1/3}} \right)^s
	\\&=
	(\ell_1 \ell_2 X^{1/3} + h)^s
	\left(1 +
	\sum_{j=1}^\infty
	\binom{s}{j}
	\left(\frac{h}{\ell_1 \ell_2 X^{1/3}}\right)^j
	\right).
\end{align}
We note that the series above is absolutely convergent since all terms are non-negative and $h/\ell_1 \ell_2 X^{1/3} <1$ for all $\ell_1, \ell_2\ge 1$. Thus, with this, we can write the double sum of the second term on the right side of \eqref{eq:1156} as
\begin{align} \label{eq:442}
	&\sum_{\ell_1 \le X^{1/3}}
	\sum_{\ell_2 \le \frac{X^{2/3}}{\ell_1}}
	\frac{f_{\ell_1 \ell_2} (s)}{\varphi(\ell_1 \ell_2)}
	(\ell_1 \ell_2 X^{1/3} +h)^s
	\\& \quad = 
	(X^{1/3} + h)^s
	\sum_{\ell_1 \le X^{1/3}}
	\sum_{\ell_2 \le \frac{X^{2/3}}{\ell_1}}
	\frac{f_{\ell_1 \ell_2} (s)}{\varphi(\ell_1 \ell_2)}
	(\ell_1 \ell_2)^s
	\left(1 +
	\sum_{j=1}^\infty
	\binom{s}{j}
	\left(\frac{h}{\ell_1 \ell_2 X^{1/3}}\right)^j
	\right).
\end{align}
The $j\ge 1$ terms from the above will contribute a negligible amount and therefore be absorbed into the error term. For $j=0$, following the same procedure as for the first double sum, we find
\begin{align}
	&
	\sum_{\ell_1 \le X^{1/3}}
	\sum_{\ell_2 \le \frac{X^{2/3}}{\ell_1}}
	\frac{f_{\ell_1 \ell_2} (s)}{\varphi(\ell_1 \ell_2)}
	(\ell_1 \ell_2)^s
	\\& \quad
	=
	\frac{1}{(2 \pi i)^2}
	\int\limits_{1 - i T_1}^{1+iT_1}
	\int\limits_{1 +\epsilon  - i T_1}^{1+ \epsilon +iT_1}
	\frac{X^{w_1/3}}{w_1}
	\frac{X^{2w_2/3}}{w_2}
	\zeta(w_1 + w_2 +1 -s)
	\zeta(w_2 + 1 - s)
	A_2(s, w_1 - s, w_2-s)
	dw_2 dw_1
	\\ & \quad\quad
	+ O(X^{1+\epsilon} (T_1 T_2)^{-1}).
\end{align}
However, unlike the previous double sum, the error term above cannot be ignored so we keep it until the end. For this double sum we shift the $w_2$ integral in the above left to $\sigma_2=1/2+\epsilon$ then shift the $w_1$ integral left to $\sigma_1 = 1-\epsilon$. The four horizontal contours contribute $\ll X^{2/3+\epsilon} T_2^{-1} 
+ X^{2\sigma_2/3 + \epsilon} T_2^{-1+2 \lambda (1-\sigma_2) + \epsilon} 
+ X^{1/3} T_1^{-1} 
+ X^{(3/2-\sigma_2)/3} T_1^{-1 + \lambda + \epsilon} $. The two left vertical contours contribute $\ll X^{2\sigma_2/3 + \epsilon} T_2^{2 \lambda (1-\sigma_2)} 
+ X^{(3/2-\sigma_2)/3} T_1^{\lambda + \epsilon}$. Setting $X^{2/3} T_2^{-1} = X^{2\sigma_2/3} T_2^{2 \lambda (1-\sigma_2)}$, and $X^{1/3} T_1^{-1}  = X^{(3/2-\sigma_2)/3} T_1^{\lambda}$, we find $T_1 \gg X^{4/13 - \epsilon}$ and $T_2 \gg X^{4/13 - \epsilon}$, for $\lambda=1/6$ and $\sigma_2=1/2+\epsilon$. Thus, all error terms add up to
\begin{equation}
	\ll X^\epsilon
	\left( 
	X^{1-\frac{4}{13}-\frac{4}{13}}
	+ X^{2/3 - 4/13}
	+ X^{1/3 - 4/13}
	\right)
	\ll X^{14/39 + \epsilon}.
\end{equation}
Multiplying this error term by $(X+h)^{1/3}$, the error term \eqref{eq:442} is $$\ll X^{14/39+1/3 + \epsilon} =  X^{9/13 + \epsilon},$$ with the main term given by the corresponding residues. Since this error term, which is roughly $\ll X^{0.692}$, is way $\ll X^{0.897}$ from the error term of the first double sum \eqref{eq:450}, we can ultimately ignore it. The other two cases can be handled similarly. We indicate the main differences.

In the second case, where $X^{1/3} < h < X^{2/3}$, we have $h/(\ell_1 X^{1/3}) \ge 1$ iff $\ell_1 \le h/ X^{1/3}$. Hence, we write the double sum in the second term of the right side of \eqref{eq:1156} as
\begin{align} \label{eq:506}
	\\
	&\sum_{\ell_1 \le X^{1/3}}
	\sum_{\ell_2 \le \frac{X^{2/3}}{\ell_1}}
	\frac{f_{\ell_1 \ell_2} (s)}{\varphi(\ell_1 \ell_2)}
	(\ell_1 \ell_2 X^{1/3} +h)^s
	\\& \quad
	=
	\sum_{\ell_1 \le \frac{h}{X^{1/3}}}
	\sum_{\frac{h}{\ell_1 X^{1/3}} \le \ell_2 \le \frac{X^{2/3}}{\ell_1}}
	\frac{f_{\ell_1 \ell_2} (s)}{\varphi(\ell_1 \ell_2)}
	(\ell_1 \ell_2 X^{1/3} +h)^s
	+
	\sum_{\ell_1 \le X^{1/3}}
	\sum_{\ell_2 < \frac{h}{\ell_1 X^{1/3}}}
	\frac{f_{\ell_1 \ell_2} (s)}{\varphi(\ell_1 \ell_2)}
	(\ell_1 \ell_2 X^{1/3} +h)^s.
\end{align}
For the first term on the right of the above, we have $\ell_1 \ell_2 X^{1/3} \ge h$ and we factor $(\ell_1 \ell_2 X^{1/3} +h)^s$ as in \eqref{eq:442}. For the second term on the right of the above, we have $\ell_1 \ell_2 X^{1/3} < h$, so we write $(\ell_1 \ell_2 X^{1/3} +h)^s$ as
\begin{equation}
	h^s \left( 1+ \frac{\ell_1 \ell_2 X^{1/3}}{h} \right)^s
	=
	h^s
	\left(
	1
	+ \sum_{j=1}^{\infty}
	\binom{s}{j}
	\left(\frac{\ell_1 \ell_2 X^{1/3}}{h}\right)^j
	\right).
\end{equation}

In the last case, where $X^{2/3} \le h \le X$, we have $h/(\ell_1 X^{1/3}) \ge 1$ always, so we write the double sum in the second term of the right side of \eqref{eq:1156} as 
\begin{align} \label{eq:506b}
	\sum_{\ell_1 \le X^{1/3}}
	\sum_{\ell_2 < \frac{h}{\ell_1 X^{1/3}}}
	\frac{f_{\ell_1 \ell_2} (s)}{\varphi(\ell_1 \ell_2)}
	(\ell_1 \ell_2 X^{1/3} +h)^s
	+
	\sum_{\ell_1 \le X^{1/3} }
	\sum_{\frac{h}{\ell_1 X^{1/3}} \le \ell_2 \le \frac{X^{2/3}}{\ell_1}}
	\frac{f_{\ell_1 \ell_2} (s)}{\varphi(\ell_1 \ell_2)}
	(\ell_1 \ell_2 X^{1/3} +h)^s,
\end{align}
and factor $(\ell_1 \ell_2 X^{1/3} +h)^s$ as in the second case. The error terms from these two cases will be no more than that of the first case, which is $\ll X^{0.692}$, since $h \le X$ for all three cases. This gives \eqref{eq:520b}. 

Similarly, we obtain \eqref{eq:520c}, noting that the error term here comes from the choices $\sigma_1 = \sigma_2 = 1/2 +\epsilon$, $T_1 = T_2 = X^{1/7 - \epsilon}$, which yields
\begin{equation}
	\ll X^{1+\epsilon} (T_1 T_2)^{-1}
	\ll X^{5/7 + \epsilon}
\end{equation}
for the error term of the last term on the right side of \eqref{eq:1156}.
\end{proof}
This completes the proof of Theorem \ref{thm:D33XhFullPoly}. \qed

\section{Conditional proof of the leading order asymptotic for the correlation sum $D_{3,3}(X,h)$: Proof of Corolary \ref{cor:correlation}}
\label{section:h=1}

Let $h=1$ (the case for $h>1$ is treated in the next section). Recall that $\Sigma_{11}(X)$, $\Sigma_{21}(X)$, and $\Sigma_{31}(X)$ are given by \eqref{eq:426}, \eqref{eq:1106b}, and \eqref{eq:1106bc}, respectively. In this section, we will evaluate $\Sigma_{11}(X)$ asymptotically (see Proposition \ref{prop:Sigma1111}), and give bounds of order strictly smaller than $\Sigma_{11}(X)$ for $\Sigma_{21}(X)$ and $\Sigma_{31}(X)$ (see Proposition \ref{prop:sigma21sigma31}).

\subsection{Using the conditional level of distribution for $\tau_3(n)$ in AP's to evaluate the sum $\Sigma_{11}(X)$} \label{section:sigma11}

We treat the most inner sum in \eqref{eq:426} using an averaged level of distribution for $\tau_3(n)$.

The main term in \eqref{eq:412} is explicit.
\begin{lem}\label{lemma:tau}
	For any $q\ge 1$, we have
	\begin{align} \label{eq:tau3Corpime}
		\frac{1}{\varphi(q)}
		\sum_{\substack{n\le X\\ (n,q)=1}} \tau_3(n)
		&= X \left(a_1(q) \log^2 X
		+ a_2(q) \log X
		+ a_3(q) \right)
		\\&+ O\left(\frac{\tau(q) X^{2/3} \log X}{\varphi(q)} \right),
	\end{align}
	where
	\begin{align} 
		a_1(q) \label{eq:a1}
		&= \frac{1}{2}
		\frac{\varphi(q)^2}{q^3},
		\\
		a_2(q) \label{eq:a2}
		&=
		\frac{\varphi(q)^2}{q^3}
		\left(3\gamma - \frac{7}{6}
		+ \frac{7}{3}
		\sum_{p\mid q} \frac{\log p}{p-1} \right),
		\\
		a_3(q) \label{eq:a3}
		&=
		\frac{\varphi(q)^2}{q^3}
		\left(3\gamma^2 - 3 \gamma + 3\gamma_1
		+
		\sum_{p\mid q} \frac{\log p}{p-1}
		\left(4\gamma - 3 + \sum_{p\mid q} \frac{\log p}{p-1} \right)\right).
	\end{align}
\end{lem}
\begin{proof}
	See, e.g, \cite[Lemma 51, p. 153]{NguyenThesis}.
\end{proof}

Thus, assuming Conjecture \ref{conj:LevelofDistribution} for $k=3$, we get, by \eqref{eq:426} and
\eqref{eq:tau3Corpime}, that
\begin{align} \label{eq:427}
	\Sigma_{11}(X)
	\sim
	(X+1)
	\left(
	b_1(X) \log^2(X+1)
	+ b_2(X) \log(X+1)
	+ b_3(X) \right)
	+ E(X),
\end{align}
where
\begin{align}
	b_1(X)  \label{eq:b1}
	&= \sum_{\ell_1 \le X^{1/3}}
	\sum_{\ell_2 \le \frac{X^{2/3}}{\ell_1}}
	a_1(\ell_1\ell_2),
	\\
	b_2(X) \label{eq:b2}
	&=
	\sum_{\ell_1 \le X^{1/3}}
	\sum_{\ell_2 \le \frac{X^{2/3}}{\ell_1}}
	a_2(\ell_1\ell_2),
	\\
	b_3(X) \label{eq:b3b}
	&=
	\sum_{\ell_1 \le X^{1/3}}
	\sum_{\ell_2 \le \frac{X^{2/3}}{\ell_1}}
	a_3(\ell_1\ell_2),
\end{align}
and
\begin{equation} \label{eq:E(X)}
	E(X) 
	= (X+1)^{2/3} \log(X+1)
	\sum_{\ell_1 \le X^{1/3}}
	\sum_{\ell_2 \le \frac{X^{2/3}}{\ell_1}}
	\frac{\tau(\ell_1 \ell_2)}{\varphi(\ell_1 \ell_2)}.
\end{equation}
We will evaluate $b_1(X)$ asymptotically and estimate $b_2(X)$, $b_3(X)$, and $E(X)$ below.

\subsubsection{Evaluation of $b_1(X)$}

By \eqref{eq:b1} and \eqref{eq:a1}, we have
\begin{equation} \label{eq:554a}
	b_1(X)
	= \frac{1}{2}
	\sum_{\ell_1 \le X^{1/3}}
	\sum_{\ell_2 \le \frac{X^{2/3}}{\ell_1}}
	\frac{\varphi(\ell_1 \ell_2)^2}{(\ell_1 \ell_2)^3}.
\end{equation}
We evaluate $b_1(X)$ in the following lemma.
\begin{lem} \label{lemma:b1(X)}
	There are computable constants $c_1$ and $c_2$ such that
	\begin{equation} \label{eq:b_1(X)a}
		b_1(X) 
		=
		\frac{1}{12}
		\prod_p
		\left(
		1- \frac{4}{p^2}
		+ \frac{4}{p^3}
		- \frac{1}{p^4}
		\right)
		\log^2(X)
		+ c_1 \log X
		+ c_2
		+ O_\epsilon\left(X^{-\frac{2}{7} +\epsilon} \right).
	\end{equation}
\end{lem}
\begin{proof}
	We apply Perron's formula twice to \eqref{eq:554a}, first to the $\ell_2$ sum, then to the $\ell_1$ sum. Let
	\begin{equation} \label{eq:f}
		f(n) = \prod_{p \mid n}
		\left(
		1-\frac{1}{p}
		\right)^2
	\end{equation}
	and
	\begin{equation} \label{eq:g}
		g_d(n) = \frac{f(nd)}{f(d)}.
	\end{equation}
	The functions $f(n)$ and $g_d(n)$ are both multiplicative in $n$. By \eqref{eq:554a}, definition of $\varphi(n)$, \eqref{eq:f}, and \eqref{eq:g}, we have
	\begin{equation} \label{eq:b_1(X)}
		b_1(X)
		= \frac{1}{2}
		\sum_{\ell_1 \le X^{1/3} }
		\frac{f(\ell_1)}{\ell_1}
		\Sigma(X,\ell_1),
	\end{equation}
	where
	\begin{equation} \label{Sigma(X,ell_1)}
		\Sigma(X,\ell_1)
		=
		\sum_{n \le \frac{X^{2/3}}{\ell_1}}
		\frac{g_{\ell_1}(n)}{n}.
	\end{equation}
	By Euler products, we have
	\begin{equation}
		\sum_{n=1}^\infty
		\frac{g_{\ell_1}(n)}{n^{s+1}}
		= \zeta(s+1) A(s) B_{\ell_1}(s),\ (\sigma >0),
	\end{equation}
	where
	\begin{equation} \label{eq:A(s)b}
		A(s)
		= \prod_p
		\left(
		1-\frac{2}{p^{s+2}} + \frac{1}{p^{s+3}}
		\right),\ (\sigma > -1),
	\end{equation}
	\begin{equation} \label{eq:Bn}
		B_{n}(s)
		= \prod_{p \mid n}
		\frac{\left(1-\frac{1}{p}\right)^2}{1-\frac{2}{p^{s+2}} + \frac{1}{p^{s+3}}},
	\end{equation}
	and $A(s)$ and $B_{\ell_1}(s)$ are convergent in the larger regions. Thus, by \eqref{Sigma(X,ell_1)} and Perron's formula, we have
	\begin{align}
		\Sigma(X,\ell_1)
		= 
		A(0) B_{\ell_1}(0)
		\log\left(
		\frac{X^{2/3}}{\ell_1}
		\right)
		+ (A B_{\ell_1})^\prime (0)
		+ \gamma A(0) B_{\ell_1}(0)
		+ O\left(
		\frac{X^\epsilon}{T}
		+ \left(\frac{X^{2/3}}{\ell_1}\right)^{-1/2} T^{1/6 + \epsilon} \right),
	\end{align}
	for a parameter $T$ to be chosen below. Hence, by \eqref{eq:b_1(X)} and the above, we have
	\begin{equation} \label{eq:b_1(X)c}
		b_1(X)
		= b_{11}(X) \log X 
		+ b_{12}(X)
		+ b_{13}(X)
		+ O\left(
		X^\epsilon T^{-1}
		+
		X^{-\frac{1}{6} + \epsilon} T^{1/6}
		\right),
	\end{equation}
	where
	\begin{align} 
		b_{11}(X) \label{eq:b_11(X)}
		& = \frac{2}{3} A(0)
		\sum_{n\le X^{1/3}}
		\frac{B_n(0)}{n},\\
		b_{12}(X) \label{eq:b_12(X)}
		& =
		- A(0)
		\sum_{n\le X^{1/3}}
		\frac{B_n(0)}{n} \log n,\\
		b_{13}(X) \label{eq:b_13(X)}
		& = ((A B_{\ell_1})^\prime (0)
		+ \gamma A(0) B_{\ell_1}(0))
		\sum_{n\le X^{1/3}} \frac{1}{n}.
	\end{align}
	Setting $T^{-1} = X^{-1/6} T^{1/6}$, we obtain $T \gg X^{1/7 -\epsilon}$, and \eqref{eq:b_1(X)c} becomes, with this choice for $T$,
	\begin{equation} \label{eq:b_1(X)b}
		b_1(X)
		= b_{11}(X) \log X 
		+ b_{12}(X)
		+ b_{13}(X)
		+ O\left(
		X^{-\frac{1}{7} + \epsilon}
		\right),
	\end{equation}
	We now evaluate the $b$'s. By the definition \eqref{eq:Bn} and Euler products, we have
	\begin{equation} \label{eq:949}
		\sum_{n=0}^\infty 
		\frac{B_n(0)}{n^{s+1}}
		= \zeta(s+1) B(s),\ (\sigma > 0),
	\end{equation}
	where
	\begin{equation} \label{eq:B(s)}
		B(s)
		=
		\prod_p
		\left(
		1 - \frac{1}{p^{s+1}}
		+ \frac{1}{p^{s+1}}
		\frac{\left(1-\frac{1}{p}\right)^2}{1 - \frac{2}{p^2} + \frac{1}{p^3}}
		\right),\ (\sigma > -1).
	\end{equation}
	By \eqref{eq:A(s)b} and \eqref{eq:B(s)}, we have
	\begin{equation} \label{eq:A0B0}
		A(0) B(0)
		= \prod_p
		\left(
		1 - \frac{4}{p^2} + \frac{4}{p^3} - \frac{1}{p^4}
		\right).
	\end{equation}
	Thus, by \eqref{eq:b_11(X)}, Perron's formula, \eqref{eq:949}, and \eqref{eq:A0B0}, we have
	\begin{align} \label{eq:b11b}
		b_{11}(X)
		&= \frac{2}{9} 
		\prod_p
		\left(
		1 - \frac{4}{p^2} + \frac{4}{p^3} - \frac{1}{p^4}
		\right)
		\log X
		\\& \quad
		+ \frac{2}{3} A(0)
		\left(
		B^\prime(0)
		+ \gamma B(0)
		\right)
		+ O\left(X^{-\frac{1}{7} + \epsilon} \right).
	\end{align}
	Next, by \eqref{eq:b_12(X)}, partial summation, and the above, we have
	\begin{align} \label{eq:b12b}
		b_{12}(X)
		&=
		- \frac{1}{18}
		\prod_p
		\left(
		1 - \frac{4}{p^2} + \frac{4}{p^3} - \frac{1}{p^4}
		\right)
		\log^2 X
		\\& \quad
		- \frac{1}{2} A(0)
		\left(
		B^\prime(0)
		+ \gamma B(0)
		\right)
		\log X
		+ O\left(X^{-\frac{1}{7} + \epsilon} \right).
	\end{align}
	Lastly, we have, from \eqref{eq:b_13(X)}
	\begin{equation} \label{eq:b13b}
		b_{13}(X)
		= 
		((A B_{\ell_1})^\prime (0)
		+ \gamma A(0) B_{\ell_1}(0))
		\left(
		\frac{1}{3}
		\log X
		+ \gamma
		+ O\left(\frac{1}{X}\right)
		\right).
	\end{equation}
	Therefore, combining \eqref{eq:b_1(X)b}, together with \eqref{eq:b11b}, \eqref{eq:b12b}, and \eqref{eq:b13b}, the estimate \eqref{eq:b_1(X)a} follows.
\end{proof}

\subsubsection{Bounds for $b_2(X)$, $b_3(X)$, and $E_1(X)$}

\begin{lem} \label{lemma:b2b3E1}
	Let $b_2(X)$, $b_3(X)$, and $E_1(X)$ be given as in \eqref{eq:b2}, \eqref{eq:b3b}, and \eqref{eq:E(X)}, respectively.
	We have, as $X\to \infty$,
	\begin{align}
		b_{2}(X) &\ll \log^2 X,
		\\
		b_3(X) & \ll \log^2 X,
		\\
		E(X) & \ll X^{\frac{2}{3} + \epsilon}
		.
	\end{align}
\end{lem}
\begin{proof}
	We have
	\begin{equation}
		\sum_{p\mid q} \frac{\log p}{p-1}
		\ll 1.
	\end{equation}
	Thus, by \eqref{eq:b2}, \eqref{eq:a2}, the above, and \eqref{eq:b1}, we have
	\begin{equation} \label{eq:b22}
		b_2(X)
		\ll
		\sum_{\ell_1 \le X^{1/3}}
		\sum_{\ell_2 \le \frac{X^{2/3}}{\ell_1}}
		\frac{\varphi(\ell_1 \ell_2)^2}{(\ell_1 \ell_2)^3}
		\ll b_1(X)
		\ll \log^2 X,
	\end{equation}
	by \eqref{eq:b_1(X)a}. Similarly, we get
	\begin{equation} \label{eq:b33}
		b_3(X) \ll \log^2 X.
	\end{equation}
	We now estimate $E_1(X)$. We have
	\begin{equation}
		\sum_{\ell_1 \le X^{1/3}}
		\sum_{\ell_2 \le \frac{X^{2/3}}{\ell_1}}
		\frac{\tau(\ell_1 \ell_2)}{\varphi(\ell_1 \ell_2)}
		\ll X^\epsilon
		\sum_{\ell_1 \le X^{1/3}}
		\frac{1}{\ell_1}
		\sum_{\ell_2 \le \frac{X^{2/3}}{\ell_1}}
		\frac{1}{\ell_2}
		\ll X^\epsilon.
	\end{equation}
	Hence, by \eqref{eq:E(X)} and the above, we get
	\begin{equation} \label{eq:E(X)X}
		E(X) \ll X^{\frac{2}{3} + \epsilon}.
	\end{equation}
\end{proof}

Therefore, combining \eqref{eq:427}, Lemmas \ref{lemma:b1(X)} and \ref{lemma:b2b3E1}, we have, on assuming Conjecture \ref{conj:LevelofDistribution}
we obtain the following
\begin{prop} \label{prop:Sigma1111}
	Assume Conjecture \ref{conj:LevelofDistribution} for $k=3$. Then, we have, as $X\to \infty$,
	\begin{equation} \label{eq:Sigma1111}
		\Sigma_{11}(X)
		\sim 
		\frac{1}{12}
		\prod_p
		\left(
		1- \frac{4}{p^2}
		+ \frac{4}{p^3}
		- \frac{1}{p^4}
		\right) 
		(X+1) \log^2(X+1) \log^2 X,
	\end{equation}
	with $\Sigma_{11}(X)$ defined in \eqref{eq:Sigma123} and given in \eqref{eq:426}.
\end{prop}

\subsection{Applying Shiu's bound to estimate the remaining sums $\Sigma_{21}(X)$ and $\Sigma_{31}(X)$} \label{section:4.2}

We apply Shiu's bound below to unconditionally treat the last two sums $\Sigma_{21}(X)$ and $\Sigma_{31}(X)$. These two sums do not contribute to the leading order main term of order $X(\log X)^4$ and only contribute to the lower order leading terms; more precisely, of order $X(\log X)^3$ and lower.
\begin{lem}[Shiu's bound] \label{lem:Shiu'sBound}
	Suppose that $1\le N<N'<2X$, $N'-N > X^\epsilon d$, and $(a,d)=1$. Then for $j,\nu\ge 1$ we have
	\begin{equation} \label{eq:Shiu'sBound}
		\sum_{\substack{N\le n\le N'\\ n\equiv a (d)}}
		\tau_j(n)^\nu
		\ll \frac{N'-N}{\varphi(d)} 
		(\log X)^{j^\nu-1}.
	\end{equation}
	The implied constants depending on $\epsilon, j$, and $\nu$ at most.
\end{lem}
\begin{proof}
	See \cite[Theorem 2]{Shiu1980}.
\end{proof}

This is
\begin{prop} \label{prop:sigma21sigma31}
	Let $\Sigma_{21}(X)$ and $\Sigma_{31}(X)$ be given by \eqref{eq:1106b} and \eqref{eq:1106bc}, respectively.
	We have
	\begin{align}
		\Sigma_{21}(X) &\ll X \log^3X \log \log X,\\
		\Sigma_{31}(X) &\ll X \log^3X \log \log X.
	\end{align}
\end{prop}
\begin{proof}
	We treat $\Sigma_{21}(X)$ first.
	By Shiu's bound \eqref{eq:Shiu'sBound}, the most inner sum over $n$ in $\Sigma_{21}(X)$ is
	\begin{equation} \label{eq:1106}
		\ll \frac{1}{\varphi(\ell_1 \ell_3)}
		(\ell_3 X^{2/3} +1)
		\log^2(\ell_3 X^{2/3} +1).
	\end{equation}
	Thus, by \eqref{eq:1106b} and \eqref{eq:1106},
	\begin{equation} \label{eq:Sigma2111}
		\Sigma_{21}(X)
		\ll
		X^{2/3} \log^2 X
		\sum_{\ell_1 \le X^{1/3}}
		\sum_{\ell_3 \le X^{1/3}}
		\frac{\ell_3}{\varphi(\ell_1 \ell_3)}
		\ll 
		X \log^3 X \log \log X.
	\end{equation}
	Similarly, we have, from \eqref{eq:Shiu'sBound}, that
	\begin{equation} \label{eq:Sigma3111}
		\Sigma_{31}(X)
		\ll X \log^2 X \log \log X.
	\end{equation}
\end{proof}

Therefore, on assuming Conjecture \ref{conj:LevelofDistribution} for $k=3$, we obtain, by \eqref{eq:Sigma123}, Propositions \ref{prop:Sigma1111} and \ref{prop:sigma21sigma31}, the asymptotic \eqref{eq:leadingTermAsymptotic} for $h=1$.

\section{General case of mixed correlations and composite shifts: Proof of Theorem \ref{thm:2b}}
\label{section:compositeh}

In this section we derive the asymptotics \eqref{eq:h=prime} and \eqref{eq:CorrelationSumforhPrime}, and describe procedure to extract the leading order main term of the mixed correlation sum $D_{k,\ell}(X,h)$ in \eqref{eq:DklXh} with composite shifts $h$. 

Let $1 \le h \le X$ be a composite number. Write
\begin{equation}
	h = \prod_p p^{\nu_p(h)}.
\end{equation}
We replace $\tau_k(n)$ in \eqref{eq:DklXh} by Hooley's identity \eqref{eq:keyDecompositionk}, giving
\begin{align}
	D_{k,\ell}(X,h)
	&	\sim
	k
	\sum_{n\le X} \tau_{\ell} (n+h)
	\sum_{\substack{\ell_1 \ell_2 \cdots \ell_k =n\\ \ell_1\ell_2 \cdots \ell_{k-1} \le X^{(k-1)/k};\ \ell_1 \le X^{1/k}}} 1
	\\&
	= k
	\sum_{\ell_1 \le X^{1/k}}
	\sum_{\ell_2 \le \frac{X^{(k-1)/k}}{\ell_1}}
	\sum_{\ell_3 \le \frac{X^{(k-1)/k}}{\ell_1 \ell_2}}
	\cdots
	\sum_{\substack{\ell_{k-1} \le \frac{X^{(k-1)/k}}{\ell_1 \cdots \ell_{k-2}}}} 
	\sum_{\ell_k \le \frac{X}{\ell_1 \cdots \ell_{k-1}}}
	\tau_{\ell}( \ell_1 \cdots \ell_{k} +h),
\end{align}
where we have used an analogous result to Proposition \ref{prop:sigma21sigma31} to bound the lower order terms. Making a change of variables $n=\ell_1 \cdots \ell_{k} +h$ in the most inner $\ell_{k}$ sum, the above becomes
\begin{equation}
	k
	\sum_{\ell_1 \le X^{1/k}}
	\sum_{\ell_2 \le \frac{X^{(k-1)/k}}{\ell_1}}
	\sum_{\ell_3 \le \frac{X^{(k-1)/k}}{\ell_1 \ell_2}}
	\cdots
	\sum_{\substack{\ell_{k-1} \le \frac{X^{(k-1)/k}}{\ell_1 \cdots \ell_{k-2}}}} 
	\sum_{\substack{n \le X+h\\ n\equiv h (\ell_1 \cdots \ell_{k-1})}} \tau_{\ell}( n).
\end{equation}
By the bound \eqref{eq:412} for all $\ell$, the error term is negligible and the above is in turns asymptotic to
\begin{equation}
	k
	\sum_{\ell_1 \le X^{1/k}}
	\sum_{\ell_2 \le \frac{X^{(k-1)/k}}{\ell_1}}
	\sum_{\ell_3 \le \frac{X^{(k-1)/k}}{\ell_1 \ell_2}}
	\cdots
	\sum_{\substack{\ell_{k-1} \le \frac{X^{(k-1)/k}}{\ell_1 \cdots \ell_{k-2}}}} 
	\frac{1}{\varphi\left( \frac{\ell_1\cdots \ell_{k-1}}{(h, \ell_1 \cdots \ell_{k-1})} \right)}
	\sum_{\substack{n \le X+h\\ \left(n,  \frac{\ell_1\cdots \ell_{k-1}}{(h, \ell_1 \cdots \ell_{k-1})} \right) = 1}} \tau_{\ell}( n).
\end{equation}
Thus, by Perron's formula in a way similar to the proof of Proposition \ref{prop:1} in Section \ref{sectiion:D33XhFullPoly},
we obtain that
\begin{align} \label{eq:Dk1040}
	D_{k,\ell}(X,h)
	&\sim k
	\underset{\substack{s=1\\ w_1 = \cdots = w_{k-1} = 0}}{\mathrm{Res}}
	\left(
	\dfrac{X^{s+ \frac{w_1}{k} + \frac{k-1}{k} (w_2 + \cdots w_{k-1})}}{s w_1 \cdots w_{k-1}}
	T_{k,\ell}(s, w_1, \cdots, w_{k-1}; h)
	\right),
\end{align}
where
\begin{align}
	T_{k,\ell}&(s, w_1, \cdots, w_{k-1}; h)
	=
	\sum_{\ell_1, \dots, \ell_{k-1} = 1}^\infty
	\frac{1}{(h, \ell_1 \cdots \ell_{k-1})^s}
	\frac{1}{\varphi\left( \frac{\ell_1\cdots \ell_{k-1}}{(h, \ell_1 \cdots \ell_{k-1})} \right)}
	\\&
	\times
	\sum_{\left( n, \frac{\ell_1\cdots \ell_{k-1}}{(h, \ell_1 \cdots \ell_{k-1})} \right) = 1}
	\frac{\tau_\ell(n (h, \ell_1 \cdots \ell_{k-1}) )}{n^s}
	\prod_{j=1}^{k-1}
	\ell_j^{ - \sum_{i=j}^{k-1} w_i }.
\end{align}
By multiplicativity and Euler products, the above generating function $T_{k,\ell}$ can be written as
\begin{equation} \label{eq:Tk1035}
	T_{k,\ell}(s, w_1, \cdots, w_{k-1}; h)
	=
	\prod_{p\mid h}
	\frac{A_p(s; w_1, \cdots, w_{k-1}; h)}{B_p(s; w_1, \cdots, w_{k-1})}
	\prod_p
	B_p(s; w_1, \cdots, w_{k-1}),
\end{equation}
where
\begin{align} \label{eq:Ap141}
	&	A_p(s;w_1,\cdots, w_{k-1};h)
	\\&=
	\sum_{j_1,\cdots, j_{k-1}}
	\frac{1}{p^{\min(j_1 + \cdots + j_{k-1}, \nu_p(h)) s }}
	\frac{1}{\varphi(p^{j_1 + \cdots + j_{k-1} - \min(j_1 + \cdots + j_{k-1}, \nu_p(h))})}
	\\& \times
	\sum_{\left( n, p^{j_1+ \cdots + j_{k-1} - \min(j_1+ \cdots + j_{k-1}, \nu_p(h) } \right) =1 }
	\frac{\tau_\ell(n p^{\min(j_1+ \cdots + j_{k-1}, \nu_p(h)})}{n^s}
	\frac{1}{p^{ \sum_{i=1}^{k-1} j_i \sum_{\kappa=i}^{k-1} w_\kappa }},
\end{align}
and
\begin{equation} \label{eq:1023}
	B_p(s; w_1, \cdots, w_{k-1})
	=
	\zeta^\ell(s)
	\left(
	1
	+ 
	\frac{\left( 1-\frac{1}{p} \right)^{\ell}}{1-\frac{1}{p}}
	\sum_{j=1}^{k-1}
	\sum_{\sigma \in \Upxi_{j,k-1}}
	\prod_{i=1}^j
	\frac{1}{p^{w_{\sigma(i)} + 1} - 1}
	\right)
\end{equation}
(we have used a nonstandard notation here, $\Upxi_{j,n} = \{ (\alpha_1 \cdots \alpha_j) \in S_{n}: \alpha_1 < \cdots < \alpha_{j} \}$ and $\sigma(i)$ to mean $\alpha_i$, where $S_{n}$ is the usual symmetric group on $n$ letters). From \eqref{eq:1023}, we can further factor out a product of zetas from $B_p$ as
\begin{align} \label{eq:Bp1034}
	\prod_p
	B_p(s; w_1, \cdots, w_{k-1})
	&=
	\zeta^\ell(s)
	\zeta(w_1 + w_2 +\cdots w_{k-1} +1)
	\zeta(w_2 +\cdots w_{k-1} +1)
	\times \cdots
	\\& \times \zeta(w_{k-1} + 1)
	\prod_p
	BB_p(s; w_1, \cdots, w_{k-1}),
\end{align}
where
\begin{equation} \label{eq:BBp541}
	BB_p(s; w_1, \cdots, w_{k-1})
	=
	\prod_{i=1}^{k-1}
	\left(
	1
	-
	\frac{1}{p^{w_i + \cdots + w_{k-1} + 1}}
	\right)
	\left(
	1
	+ 
	\frac{\left( 1-\frac{1}{p} \right)^{\ell}}{1-\frac{1}{p}}
	\sum_{j=1}^{k-1}
	\sum_{\sigma \in \Upxi_{j,k-1}}
	\prod_{i=1}^j
	\frac{1}{p^{w_{\sigma(i)} + 1} - 1}
	\right).
\end{equation}
The product $\prod_p BB_p(s; w_1, \cdots, w_{k-1})$ converges in a wider region than $\prod_p B_p$ since we have factored out all the poles from the latter. Similarly, from Lemmas \ref{lemma:tauk(nh)} and \ref{lemma:tauknhcoprime}, the local Euler factors can be written as
\begin{align} \label{eq:Ap1034}
	A_p(s;w_1,\cdots, w_{k-1};h)
	=
	\zeta^\ell(s)
	AA_p(s;w_1,\cdots, w_{k-1};h)
\end{align}
with $AA_p(s;w_1,\cdots, w_{k-1};h)$ a nice Euler product converging in a larger region. From \eqref{eq:Ap1034} and \eqref{eq:Bp1034}, the factor $\zeta^\ell(s)$ cancels out in the ratio $\frac{A_p}{B_p}$, and that the generating function $T_{k,\ell}(s, w_1, \cdots, w_{k-1}; h)$ \eqref{eq:Tk1035}  can thus be written as
\begin{align} \label{eq:Tk1212}
	&T_{k,\ell}(s, w_1, \cdots, w_{k-1}; h)
	= \zeta^\ell(s) \zeta(w_1 + w_2 +\cdots w_{k-1} +1)
	\zeta(w_2 +\cdots w_{k-1} +1)
	\\& \quad \times \zeta(w_{k-1} + 1)
	\prod_{p\mid h}
	\frac{A_p(s; w_1, \cdots, w_{k-1}; h)}{B_p(s; w_1, \cdots, w_{k-1})}
	\prod_p
	BB_p(s; w_1, \cdots, w_{k-1}),
\end{align}
and, hence, we conclude that $T_{k,\ell}(s, w_1, \cdots, w_{k-1}; h)$ has poles at $s=1$ and $w_1 = \cdots = w_{k-1} = 0$. Therefore, by \eqref{eq:Tk1212} above, we obtain from \eqref{eq:Dk1040}, on assuming Conjecture \ref{conj:LevelofDistribution} for all $\ell$, that
\begin{equation} \label{eq:Dk138}
	D_{k, \ell}(X,h)
	\sim
	\frac{C_{k,\ell}
		f_{k,\ell}(h)}{(k-1)! (\ell-1)!}
	X (\log X)^{k+\ell-2},
\end{equation}
where
\begin{equation} \label{eq:Ckl}
	C_{k,\ell}
	= 
	\prod_p
	BB_p(0; \vec{0})
	=
	\prod_p
	\left(
	1 - \frac{1}{p}
	\right)^{k-1}
	\left(
	1
	+ 
	\left( 1-\frac{1}{p} \right)^{\ell-1}
	\sum_{j=1}^{k-1}
	\frac{\binom{k-1}{j}}{(p-1)^j}
	\right)
\end{equation}
and
\begin{equation} \label{eq:fklh}
	f_{k,\ell}(h)
	= 
	\prod_{p \mid h}
	\frac{A_p(1;\vec{0};h)}{B_p(1;\vec{0})},
\end{equation}
and where we have abbreviated $\vec{0}$ for $0,\cdots, 0$ $k-1$ times. This gives the asymptotic \eqref{eq:CorrelationSumforhPrime}.

Lastly, we show that the constant $C_{k,\ell}$  from \eqref{eq:Ckl} above matches the predicted global constant from equation (1.6) of Ng and Thom \cite{NgThom2019}.
\begin{prop}
	Let $C_{k,\ell}$ be defined as in \eqref{eq:Ckl}.
	We have
	\begin{equation} \label{eq:Ckl2}
		C_{k,\ell}
		= 
		\prod_p
		\left(
		\left(
		1 - \frac{1}{p}
		\right)^{k-1}
		+ 
		\left(
		1 - \frac{1}{p}
		\right)^{\ell-1}
		-
		\left(
		1 - \frac{1}{p}
		\right)^{k+\ell-2}
		\right),
	\end{equation}
	which matches exactly equation (1.6) of \cite{NgThom2019}.
\end{prop}
\begin{proof}
	We have the identity
	\begin{equation}
		\sum_{j=1}^{k-1}
		\frac{\binom{k-1}{j}}{(p-1)^j}
		=
		-1
		+
		\left(
		\frac{p}{p-1}
		\right)^k
		-
		\frac{1}{p}
		\left(
		\frac{p}{p-1}
		\right)^k.
	\end{equation}
	Substituting the above into the right side of \eqref{eq:Ckl} and simplifying then give the right side of \eqref{eq:Ckl2}.
\end{proof}
In the next three subsections, we compute exactly and match the local constants $f_{k,\ell}(h)$ from \eqref{eq:fklh} for the special case $k=\ell=3$ and any composite shift $h$ with \cite{NgThom2019}.

\subsection{The case $k=3$ and $\ell,h\ge 1$}
\label{section:hprime}

In this subsection, we demonstrate how to apply our general method developed above to extract the leading order main term for the case $k=3$ and $\ell, h \ge 1$, 
in particular, deriving the asymptotic \eqref{eq:leadingTermAsymptotic} and
showing that our answers match with previously conjectured values.

Let $k=3$ and fix $\ell, h\ge 1$. The procedure from previous subsection gives that
\begin{align} \label{eq:compositeh}
	\sum_{n\le X} 
	&\tau_3(n) \tau_\ell(n+h)
	\sim
	3
	\sum_{\ell_1 \le X^{1/3}}
	\sum_{\ell_2 \le \frac{X^{2/3}}{\ell_1}}
	\sum_{\substack{n\le X+h\\ n \equiv h (\ell_1\ell_2)}}
	\tau_\ell(n)
	\\&\sim
	3
	\underset{\substack{s=1\\ w_1=w_2=0}}{\mathrm{Res}}
	\left(
	\frac{X^{\frac{1}{3}w_1}}{w_1}
	\frac{X^{\frac{2}{3}w_2}}{w_2}
	\frac{(X+h)^s}{s}
	T_\ell(s,w_1,w_2;h)
	\right),
\end{align}
with
\begin{align}
	T_\ell(s,w_1,w_2;h)
	&=
	\sum_{\ell_1, \ell_2 = 1}^\infty
	\frac{1}{(h, \ell_1 \ell_2)^s}
	\frac{1}{\varphi(\ell_1 \ell_2 / (h, \ell_1 \ell_2))}
	\\ & \times
	\sum_{\left(n, \frac{\ell_1 \ell_2}{(h, \ell_1 \ell_2)} \right)=1}
	\frac{\tau_\ell(n (h, \ell_1 \ell_2))}{n^s}
	\frac{1}{\ell_1^{w_1+w_2}}
	\frac{1}{\ell_2^{w_2}}
	\\&=
	\prod_{p\mid h} \frac{A_p(s; w_1,w_2;h)}{B_p(s; w_1, w_2)}
	\prod_p B_p(s; w_1, w_2),
\end{align}
with the global Euler factor $B_p(s; w_1,w_2)$ given in \eqref{eq:1023} with $k=3$, and local factor
\begin{align}
	A_p&(s,w_1,w_2;h)
	=
	\sum_{j_1,j_2}
	\frac{1}{p^{\min(j_1 +j_2, \nu_p(h)) s }}
	\frac{1}{\varphi(p^{j_1+j_2 - \min(j_1 + j_2, \nu_p(h))})}
	\\& \times
	\sum_{\left( n, p^{j_1+j_2 - \min(j_1+j_2, \nu_p(h) } \right) =1 }
	\frac{\tau_\ell(n p^{\min(j_1+j_2, \nu_p(h)})}{n^s}
	\frac{1}{p^{j_1(w_1+w_2) + j_2w_2}}.
\end{align}
Thus, \eqref{eq:compositeh} predicts that
\begin{equation} \label{eq:hcompositeprediction}
	\sum_{n\le X} \tau_3(n) \tau_\ell(n+h)
	\sim
	\frac{1}{4}
	\prod_p
	\left(
	1- \frac{4}{p^2}
	+ \frac{4}{p^3}
	- \frac{1}{p^4}
	\right) 
	\prod_{p \mid h} f_{3,\ell}(h)
	X \log^4 X.
\end{equation}
with
\begin{equation} \label{eq:238}
	f_{3,\ell}(h)
	=
	\frac{A_p(1 ;0,0;h)}{B_p(1;0,0)}.
\end{equation}
We first evaluate $f_{3,\ell}(h)$ in \eqref{eq:238} for $\ell=3$ and $h$ prime.

\subsection{Prime shifts}

\begin{prop}
	Let $h$ be a prime. We have
	\begin{equation} \label{eq:f33h}
		f_{3,3}(h)
		=
		\frac{h^3 + 6h^2 + 3h - 4}{h (h^2+2h-1)}.
	\end{equation}
	In particular, assuming the bound \eqref{eq:412} for $k=3$, we have, for $h$ prime,
	\begin{align} \label{eq:prop3.3}
		D_{3,3}(X,h)
		\sim
		\frac{1}{4}
		\prod_p
		\left(
		1- \frac{4}{p^2}
		+ \frac{4}{p^3}
		- \frac{1}{p^4}
		\right) 
		\frac{h^3 + 6h^2 + 3h - 4}{h (h^2+2h-1)}
		X \log^4 X.
	\end{align}
\end{prop}

\begin{proof}
	By Perron's formula and \eqref{eq:412}, we have
	\begin{align} \label{eq:822}
		&\sum_{\ell_1 \le X^{1/3}}
		\sum_{\ell_2 \le \frac{X^{2/3}}{\ell_1}}
		\sum_{\substack{n\le X+h\\ n \equiv h (\ell_1\ell_2)}}
		\tau_3(n)
		\\&
		\sim \frac{1}{(2\pi i)^3}
		\int_{(2)} \int_{(2)} \int_{(2)}
		\frac{X^{\frac{1}{3}w_1}}{w_1}
		\frac{X^{\frac{2}{3}w_2}}{w_2}
		\frac{X^s}{s}
		T_3(s,w_1,w_2)
		dw_1 dw_2 ds,
	\end{align}
	where
	\begin{equation}
		T_3(s,w_1,w_2)
		= 
		\sum_{\ell_1, \ell_2=1}^\infty
		\frac{1}{\varphi(q_1)}
		\frac{1}{\delta^s}
		\sum_{(n,q_1)=1}
		\frac{\tau_3(n \delta)}{n^s}
		\frac{1}{\ell_1^{w_1+w_2}}
		\frac{1}{\ell_2^{w_2}}
	\end{equation}
	with
	\begin{equation}
		\delta = (h, \ell_1\ell_2)
	\end{equation}
	and
	\begin{equation}
		q_1 = \frac{\ell_1 \ell_2}{\delta}.
	\end{equation}
	By Euler products, we can write this function as
	\begin{align} \label{eq:T3}
		T_3(s,w_1,w_2)
		&=
		\prod_{p\mid h} \frac{A_p(s;w_1,w_2;h)}{B_p(s;w_1, w_2)}
		\prod_p B_p(s;w_1, w_2),
	\end{align}
	where
	\begin{align} \label{eq:B}
		B_p(s ;w_1, w_2)
		=
		\displaystyle\sum_{n} \frac{\tau_3(n)}{n^s}
		+
		\displaystyle\sum_{\substack{j_1,j_2\\ j_1j_2 \neq 0}}
		\frac{1}{\varphi(p^{j_1+j_2})}
		\displaystyle\sum_{(n, p^{j_1+j_2})=1}
		\frac{\tau_3(n)}{n^s}
		\frac{1}{p^{j_1(w_1+w_2) + j_2w_2} }
	\end{align}
	and
	\begin{equation} \label{eq:A}
		A_p(s;w_1,w_2;h)
		=
		\displaystyle\sum_{n} \frac{\tau_3(n)}{n^s}
		+
		\displaystyle\sum_{\substack{j_1,j_2\\ j_1j_2 \neq 0}}
		\frac{1}{\varphi(p^{j_1+j_2-1})}
		\frac{1}{h^s}
		\displaystyle\sum_{(n, p^{j_1+j_2-1})=1}
		\frac{\tau_3(nh)}{n^s}
		\frac{1}{p^{j_1(w_1+w_2) + j_2w_2} }.
	\end{equation}
	We now evaluate the functions $A$ and $B$. We start with $B$.
	
	We split the $j_i$ sums in \eqref{eq:B} into
	\begin{equation}
		\sum_{\substack{j_1,j_2\\ j_1j_2 \neq 0}}
		=
		\sum_{\substack{j_1\ge 1\\ j_2=0}}
		+
		\sum_{\substack{j_1=0 \\ j_2 \ge 1}}
		+
		\sum_{\substack{j_1 \ge 1 \\ j_2 \ge 1}}.
	\end{equation}
	We have
	\begin{align}
		\sum_{\substack{j_1\ge 1\\ j_2=0}}
		&\frac{1}{\varphi(p^{j_1+j_2})}
		\displaystyle\sum_{(n, p^{j_1+j_2})=1}
		\frac{\tau_3(n)}{n^s}
		\frac{1}{p^{j_1(w_1+w_2) + j_2w_2} }
		\\&=
		\sum_{j_1=1}^\infty
		\frac{1}{p^{j_1} \left(1-\frac{1}{p}\right)}
		\displaystyle\sum_{(n, p)=1}
		\frac{\tau_3(n)}{n^s}
		\frac{1}{p^{j_1(w_1+w_2)} }
		\\&=
		\zeta^3(s)
		\frac{\left(1-\frac{1}{p^s}\right)^3}{1-\frac{1}{p}}
		\sum_{j_1=1}^\infty \frac{1}{p^{j_1(w_1+w_2+1)} }
		\\&=
		\zeta^3(s)
		\frac{\left(1-\frac{1}{p^s}\right)^3}{1-\frac{1}{p}}
		\frac{1}{p^{w_1+w_2+1} - 1},
	\end{align}
	\begin{align}
		\sum_{\substack{j_1=0 \\ j_2 \ge 1}}
		&\frac{1}{\varphi(p^{j_1+j_2})}
		\displaystyle\sum_{(n, p^{j_1+j_2})=1}
		\frac{\tau_3(n)}{n^s}
		\frac{1}{p^{j_1(w_1+w_2) + j_2w_2} }
		\\&=
		\zeta^3(s)
		\frac{\left(1-\frac{1}{p^s}\right)^3}{1-\frac{1}{p}}
		\frac{1}{p^{w_2+1} - 1},
	\end{align}
	and
	\begin{align}
		\sum_{\substack{j_1 \ge 1 \\ j_2 \ge 1}}
		&\frac{1}{\varphi(p^{j_1+j_2})}
		\displaystyle\sum_{(n, p^{j_1+j_2})=1}
		\frac{\tau_3(n)}{n^s}
		\frac{1}{p^{j_1(w_1+w_2) + j_2w_2} }
		\\&=
		\zeta^3(s)
		\frac{\left(1-\frac{1}{p^s}\right)^3}{1-\frac{1}{p}}
		\frac{1}{p^{w_1+w_2+1} - 1}
		\frac{1}{p^{w_2+1} - 1}.
	\end{align}
	Thus,
	\begin{align} \label{eq:B_p}
		B_p(s ;w_1,w_2 )
		=
		\zeta^3(s)
		&\left(
		1
		+
		\frac{\left(1-\frac{1}{p^s}\right)^3}{1-\frac{1}{p}}
		\left(
		\frac{1}{p^{w_1+w_2+1} - 1}
		\right. \right.
		\\ & \left. \left.
		+
		\frac{1}{p^{w_2+1} - 1}
		+
		\frac{1}{p^{w_1+w_2+1} - 1}
		\frac{1}{p^{w_2+1} - 1}
		\right) \right)
	\end{align}
	and, hence,
	\begin{equation}  \label{eq:B2}
		\prod_p 
		B_p(s; w_1,w_2)
		=
		\zeta^3(s)
		\zeta(w_1+w_2+1) \zeta(w_2+1)
		BB(s; w_1,w_2),
	\end{equation}
	where
	\begin{align}
		BB(s; w_1,w_2)
		=
		\prod_p
		&\left(1-\dfrac{1}{p^{w_1+w_2+1}}\right)
		\left(1-\dfrac{1}{p^{w_2+1}}\right)
		\\& \times
		\left(
		1
		+
		\frac{\left(1-\frac{1}{p^s}\right)^3}{1-\frac{1}{p}}
		\left(
		\frac{1}{p^{w_1+w_2+1} - 1}
		+
		\frac{1}{p^{w_2+1} - 1}
		\right. \right.
		\\ & \left. \left.
		+
		\frac{1}{p^{w_1+w_2+1} - 1}
		\frac{1}{p^{w_2+1} - 1}
		\right)
		\right).
	\end{align}
	We have that
	\begin{align} \label{eq:C100}
		BB(1; 0,0)
		&=
		\prod_p
		\left(1-\frac{1}{p}\right)^2
		\left(
		1
		+
		\left(1-\frac{1}{p}\right)^2
		\left(
		\frac{1}{p-1}
		+
		\frac{1}{p-1}
		+
		\frac{1}{(p-1)^2}
		\right)
		\right)
		\\&=
		\prod_p
		\left(
		1 - \frac{4}{p^2} + \frac{4}{p^3} - \frac{1}{p^4}
		\right).
	\end{align}
	We evaluate the $dw_2$ integral in \eqref{eq:822} first, picking up a double pole at $w_2 =0$, then perform the $dw_1$ integral, collecting the triple pole at $w_1=0$, and finally the $ds$ integral, with a triple pole at $s=0$. Thus, the left side of \eqref{eq:prop3.3} is asymptotic to
	\begin{equation} \label{eq:827}
		\frac{BB(1; 0,0) A(1; 0,0; h)}{12}
		X \log^4 X.
	\end{equation}
	We next evaluate \eqref{eq:A}.
	
	Because of the exponent $j_1+j_2-1$ in \eqref{eq:A} being non-negative, we split the $j$ sums in \eqref{eq:A} into
	\begin{equation}
		\sum_{\substack{j_1,j_2\\ j_1j_2 \neq 0}}
		=
		\sum_{\substack{j_1 = 1\\ j_2=0}}
		+
		\sum_{\substack{j_1=0\\ j_2 = 1}}
		+
		\sum_{\substack{j_1 \ge 2\\ j_2=0}}
		+
		\sum_{\substack{j_1=0\\ j_2 \ge 2}}
		+
		\sum_{\substack{j_1 \ge 1\\ j_2\ge 1}}.
	\end{equation}
	We have
	\begin{align}
		\sum_{\substack{j_1 = 1\\ j_2=0}}
		&\frac{1}{\varphi(p^{j_1+j_2-1})}
		\frac{1}{h^s}
		\sum_{(n, p^{j_1+j_2-1})=1}
		\frac{\tau_3(nh)}{n^s}
		\frac{1}{p^{j_1(w_1+w_2) + j_2w_2} }
		\\&= 
		\frac{1}{h^s}
		\sum_{n}
		\frac{\tau_3(nh)}{n^s}
		\frac{1}{p^{w_1+w_2} }
		=
		\frac{1}{h^s}
		\frac{1}{p^{w_1+w_2} }
		\zeta^3(s) A_h(s),
	\end{align}
	\begin{align}
		\sum_{\substack{j_1=0\\ j_2 = 1}}
		&\frac{1}{\varphi(p^{j_1+j_2-1})}
		\frac{1}{h^s}
		\sum_{(n, p^{j_1+j_2-1})=1}
		\frac{\tau_3(nh)}{n^s}
		\frac{1}{p^{j_1(w_1+w_2) + j_2w_2} }
		\\&= 
		\frac{1}{h^s}
		\sum_{n}
		\frac{\tau_3(nh)}{n^s}
		\frac{1}{p^{w_2} }
		=
		\frac{1}{h^s}
		\frac{1}{p^{w_2} }
		\zeta^3(s) A_h(s),
	\end{align}
	\begin{align}
		\sum_{\substack{j_1\ge 2\\ j_2 = 0}}
		&\frac{1}{\varphi(p^{j_1+j_2-1})}
		\frac{1}{h^s}
		\sum_{(n, p^{j_1+j_2-1})=1}
		\frac{\tau_3(nh)}{n^s}
		\frac{1}{p^{j_1(w_1+w_2) + j_2w_2} }
		\\&= 
		\sum_{j_1=2}^\infty
		\frac{1}{\varphi(p^{j_1-1})}
		\frac{1}{h^s}
		\sum_{(n, h)=1}
		\frac{\tau_3(h)\tau_3(n)}{n^s}
		\frac{1}{p^{j_1(w_1+w_2)} }
		\\&=
		\frac{1}{h^s}
		\tau_3(h)
		\zeta^3(s)
		\prod_{p\mid h} \left(1-\frac{1}{p^s}\right)^3
		\sum_{j_1=2}^\infty
		\frac{1}{p^{j_1-1} \left(1-\frac{1}{p}\right)} \frac{1}{p^{j_1(w_1+w_2)} }
		\\&=
		\frac{3}{p^{s-1}} \zeta^3(s)
		\frac{\left(1-\dfrac{1}{p^s}\right)^3}{1-\dfrac{1}{p}}
		\sum_{j_1=2}^\infty \frac{1}{p^{j_1(w_1+w_2+1)}}
		\\&=
		\frac{3}{p^{s-1}} \zeta^3(s)
		\frac{\left(1-\dfrac{1}{p^s}\right)^3}{1-\dfrac{1}{p}}
		\frac{1}{p^{w_1+w_2+1}} \frac{1}{p^{w_1+w_2+1} - 1},
	\end{align}
	\begin{align}
		\sum_{\substack{j_1=0\\ j_2 \ge 2}}
		&\frac{1}{\varphi(p^{j_1+j_2-1})}
		\frac{1}{h^s}
		\sum_{(n, p^{j_1+j_2-1})=1}
		\frac{\tau_3(nh)}{n^s}
		\frac{1}{p^{j_1(w_1+w_2) + j_2w_2} }
		\\&= 
		\sum_{j_2=2}^\infty
		\frac{1}{\varphi(p^{j_2-1})}
		\frac{1}{h^s}
		\sum_{(n, h)=1}
		\frac{\tau_3(h)\tau_3(n)}{n^s} \frac{1}{p^{j_2w_2} }
		\\&=
		\frac{3}{p^{s-1}} \zeta^3(s)
		\frac{\left(1-\dfrac{1}{p^s}\right)^3}{1-\dfrac{1}{p}}
		\frac{1}{p^{w_2+1}} \frac{1}{p^{w_2+1} - 1},
	\end{align}
	and
	\begin{align}
		\sum_{\substack{j_1\ge 1\\ j_2 \ge 1}}
		&\frac{1}{\varphi(p^{j_1+j_2-1})}
		\frac{1}{h^s}
		\sum_{(n, p^{j_1+j_2-1})=1}
		\frac{\tau_3(nh)}{n^s}
		\frac{1}{p^{j_1(w_1+w_2) + j_2w_2} }
		\\&= 
		\sum_{\substack{j_1\ge 1\\ j_2 \ge 1}}
		\frac{1}{\varphi(p^{j_1+j_2-1})}
		\frac{1}{h^s}
		\sum_{(n, h)=1}
		\frac{\tau_3(h)\tau_3(n)}{n^s}
		\frac{1}{p^{j_1(w_1+w_2) + j_2w_2} }
		\\&=
		\frac{3}{p^{s-1}} \zeta^3(s)
		\frac{\left(1-\dfrac{1}{p^s}\right)^3}{1-\dfrac{1}{p}}
		\frac{1}{p^{w_1+w_2+1} - 1}
		\frac{1}{p^{w_2+1} - 1}.
	\end{align}
	Thus, the local Euler product $A_p(s; w_1,w_2; h)$ of $T_3(s; w_1,w_2)$ is equal to
	\begin{align}
		\zeta^3(s)
		&\left(
		1
		+
		\frac{1}{p^s}
		A_p(s)
		\left(
		\frac{1}{p^{w_1+w_2} }
		+
		\frac{1}{p^{w_2} }
		\right)
		+
		\frac{3}{p^{s-1}}
		\frac{\left(1-\dfrac{1}{p^s}\right)^3}{1-\dfrac{1}{p}}
		\left(
		\frac{1}{p^{w_2+1}} \frac{1}{p^{w_2+1} - 1}
		\right.\right.
		\\& \left.\left. \qquad\qquad\qquad
		+
		\frac{1}{p^{w_1+w_2+1}} \frac{1}{p^{w_1+w_2+1} - 1}
		+
		\frac{1}{p^{w_1+w_2+1} - 1}
		\frac{1}{p^{w_2+1} - 1}
		\right)
		\right).
	\end{align}
	Thus, by this, \eqref{eq:A} and \eqref{eq:B_p},
	\begin{align}
		&A_p(s; w_1,w_2; h)
		\\&=
		\frac{
			1
			+
			\frac{1}{p^s}
			A_p(s)
			\left(
			\frac{1}{p^{w_1+w_2} }
			+
			\frac{1}{p^{w_2} }
			\right)
			+
			\frac{3}{p^{s-1}}
			\frac{\left(1-\frac{1}{p^s}\right)^3}{1-\frac{1}{p}}
			\left(
			\frac{1}{p^{w_2+1}} \frac{1}{p^{w_2+1} - 1}
			+
			\frac{1}{p^{w_1+w_2+1}} \frac{1}{p^{w_1+w_2+1} - 1}
			+
			\frac{1}{p^{w_1+w_2+1} - 1}
			\frac{1}{p^{w_2+1} - 1}
			\right)}
		{1
			+
			\frac{\left(1-\frac{1}{p^s}\right)^3}{1-\frac{1}{p}}
			\left(
			\frac{1}{p^{w_1+w_2+1} - 1}
			+
			\frac{1}{p^{w_2+1} - 1}
			+
			\frac{1}{p^{w_1+w_2+1} - 1}
			\frac{1}{p^{w_2+1} - 1}
			\right)}.
	\end{align}
	Now, by \eqref{eq:A_h} with $k=3$, we have
	\begin{equation}
		A_p(1)
		=
		3 - \frac{3}{p} + \frac{1}{p^2}.
	\end{equation}
	Hence,
	\begin{equation}
		A_p(1; 0,0; h)
		=
		\frac{p^3 + 6p^2 + 3p - 4}{p (p^2+2p-1)}.
	\end{equation}
	This, together with \eqref{eq:827} and \eqref{eq:C100}, give the right side of \eqref{eq:prop3.3}.
\end{proof}

\subsection{Composite shift $h$} \label{sec:hcomposite}

Similarly, for any $h$ composite, Mathematica calculations\footnote{Link to Mathematica file calculation: \href{https://aimath.org/~dtn/papers/correlations/calculations\%20for\%20k=3,\%20any\%20ell\%20and\%20h.nb}{https://aimath.org/$\sim$dtn/papers/correlations/calculations for k=3, any ell and h.nb}} give
\begin{align} \label{eq:f33b}
	f_{3,3}(h)
	=
	&\prod_{p \mid h}
	\left(
	-\nu_p(h)^2 (p-1)^2 (p+1)+p^{\nu_p(h)+2}+4 p^{\nu_p(h)+3}
	\right.
	\\ & \left.
	+p^{\nu_p(h)+4}
	+\nu_p(h) \left(-4 p^3+6 p-2\right)
	-4 p^3-5 p^2+4 p-1
	\right)
	\\ & 
	/
	\left(
	p^{\nu_p(h)}(p-1)^2 \left(p^2+2 p-1\right)
	\right),
\end{align}
\begin{align} \label{eq:f34}
	f_{3,4}(h)
	&=
	\prod_{p\mid h}
	p \left(-\nu_p(h)^3 (p+1) (p-1)^3-\nu_p(h)^2 \left(7 p^2+6 p-4\right) (p-1)^2
	\right.
	\\& \left.
	+\nu_p(h) \left(-16 p^4+33 p^2-22 p+5\right)+2 \left(-p^{\nu_p(h)+2}
	+5 p^{\nu_p(h)+3}
	\right. \right.
	\\& \left. \left.
	+5 p^{\nu_p(h)+4}
	+p^{\nu_p(h)+5}-6 p^4-9 p^3+9 p^2-5 p+1\right)\right)
	\\& 
	/
	\left( 2 p^{\nu_p(h)}
	(p-1)^3 \left(p^3+2 p^2-3 p+1\right)
	\right),
\end{align}
and
\begin{align} \label{eq:f35}
	f_{3,5}(h)
	&=
	\prod_{p\mid h}
	\left(-\nu_p(h)^4 (p+1) (p-1)^4-\nu_p(h)^3 \left(11 p^2+8 p-7\right) (p-1)^3
	\right.
	\\& \left.
	-\nu_p(h)^2 \left(44 p^3+31 p^2-50 p+17\right) (p-1)^2-\nu_p(h) \left(76 p^5+p^4
	\right. \right.
	\\ & \left. \left.
	-200 p^3
	+200 p^2-94 p+17\right)+6 \left(6 p^{\nu_p(h)+4}+8 p^{\nu_p(h)+5}
	\right. \right.
	\\ & \left. \left.
	+p^{\nu_p(h)+6}
	-8 p^5-14 p^4
	+16 p^3-14 p^2+6 p-1\right)\right)
	\\ & /
	\left(
	6 p^{\nu_p(h)} (p-1)^2 \left(p^4+2 p^3-5 p^2+4 p-1\right)
	\right),
\end{align}
and so on, where $\nu_p(h)$ is the highest power of $p$ that divides $h$. The local constants \eqref{eq:f33h}, \eqref{eq:f33b}, \eqref{eq:f34}, and \eqref{eq:f35} agree with the predicted values from Ng and Thom \cite[equation (1.7)]{NgThom2019}. 

We next compare our predicted leading main term with the that from the delta method \cite{DukeFriedlanderIwaniec} of Duke, Friedlander, and Iwaniec.

\section{Comparison with a conjectural formula of Conrey and Gonek: Proof of Theorem \ref{thm:deltamethod}}
\label{section:ConreyGonek}

Two decades ago, in 2002, Conrey and Gonek predicted in \cite[Conjecture 3]{ConreyGonnek2002} that, for $k=3$ and $h=1$, we have
\begin{equation} \label{eq:m31}
	\sum_{n\le X} \tau_3(n) \tau_3(n+1)
	=
	m_3(X,1) + O\left( X^{1/2 +\epsilon} \right),
\end{equation}
where the derivative of the main term $m_3(x,1)$ from the delta method satisfies
\begin{equation} \label{eq:m3prime}
	m_3^\prime(u,1)
	=
	\sum_{q=1}^\infty 
	\frac{\mu(q)}{q^2}
	\left[
	\underset{\substack{s=0}}{\mathrm{Res}}
	\left(
	\zeta^3(s+1) G_3(s+1,q) \left(\frac{u}{q}\right)^s
	\right)
	\right]^2,
\end{equation}
and $G_3(s,q)$ is a multiplicative function in $q$ which, by \cite[Lemma 4.3, pg. 17]{BaluyotConrey2022}, at prime values, reduces to
\begin{equation} \label{eq:G3}
	G_3(s,p)
	=
	p^s
	\left(
	1
	-
	\frac{p}{p-1}
	\left(
	1- \frac{1}{p^{s}}
	\right)^3
	\right).
\end{equation}

In this section, we will compute this main term $m_3(X,1)$ by working out the residue in \eqref{eq:m3prime} using the simplified version for $G_3(s,q)$ in \eqref{eq:G3}. After that, we comment on the behavior of the error term in \eqref{eq:m31}. For ease of comparing, we restate the main result of this section below, with digits that match with our prediction \eqref{eq:m3u1bc} highlighted in bold, and give a proof below.
\setcounter{thm}{2}
\begin{thm} \label{thm:deltamethod}
	We have, with at least 71 digits accuracy in the coefficients,
	\begin{align} \label{eq:m3u1}
		&m_3(X,1)
		=
		{\mathbf{0.05444467915488409458075187852986170328269943875033898441206}}
		\\& \qquad
		{\mathbf{910088090 66227780631551394813609558909414229584839437008}} X \log ^4(X)
		\\&+
		{\mathbf{0.710113929053644747553958926673505372958197119463757504939845715359}}
		\\& \qquad
		{\mathbf{73}}9076661971842253983213149206 X \log ^3(X)
		\\&+
		{\mathbf{2.021196057879877779433242407847538094670915083699177892670406035438}}
		\\& \qquad
		{\mathbf{8}}0548628848354775122568369734 X \log ^2(X)
		\\&+
		{\mathbf{0.677863310832980388541571083062733656003222322704135348688102425159}}
		\\& \qquad
		{\mathbf{89}}727867201461267995359769 X \log (X)
		\\&+ {\mathbf{0.287236647746619417221664617814645950166036274397222249618913907447}}
		\\& \qquad
		31664345218868780687078219 X
		+ O(X^\epsilon).
	\end{align}
\end{thm}
\begin{proof}
	To evaluate \eqref{eq:m3prime}, we bring the $q$ sum inside and evaluate the residues afterwards. Then integrating the resulting expression will give us the polynomial $m_3(X,1)$. Thus, we rewrite \eqref{eq:m3prime} as
	\begin{align}
		m_3'(u,1)
		&=
		\underset{\substack{s=0\\ w=0}}{\mathrm{Res}}\
		\zeta^3(s+1)\zeta^3(w+1)
		u^{s+w}
		\sum_{q=1}^\infty 
		\frac{\mu(q)}{q^{2+s+w}}
		G_3(s+1,q)G_3(w+1,q)
		\\&
		=
		\underset{\substack{s=0\\ w=0}}{\mathrm{Res}}\
		\zeta^3(s+1)\zeta^3(w+1)
		u^{s+w}
		A(s,w),
	\end{align}
	where
	\begin{align} 
		A(s,w)
		&=
		\prod_p
		\left(
		1
		-
		\frac{G_3(s+1,p)}{p^{s+1}}
		\frac{G_3(w+1,p)}{p^{w+1}}
		\right)
		\\& = \label{eq:Asw}
		\prod_p
		\left(
		1
		-
		\left(
		1
		-
		\frac{p}{p-1}
		\left(
		1- \frac{1}{p^{s+1}}
		\right)^3
		\right)
		\left(
		1
		-
		\frac{p}{p-1}
		\left(
		1- \frac{1}{p^{w+1}}
		\right)^3
		\right)
		\right)
	\end{align}
	by \eqref{eq:G3}. Hence,
	\begin{align} \label{eq:m3'u1}
		m_3'&(u,1)
		=
		\frac{1}{4} A(0,0) \log^4 u
		+ 
		\log^3 u 
		\frac{1}{2} \left(6 \gamma  A(0,0) + 2 A^{(1,0)}(0,0) \right) 
		\\&
		+
		\log^2 u
		\frac{1}{4} \left(
		(48 \gamma ^2 -12 \gamma _1) A(0,0)
		+36 \gamma  A^{(1,0)}(0,0)
		+4 A^{(1,1)}(0,0)
		+2 A^{(2,0)}(0,0)
		\right) 
		\\&
		+
		\log u
		\frac{1}{2} \left(
		(36 \gamma ^3 - 36 \gamma \gamma _1 ) 
		A(0,0)
		+ (48 \gamma ^2 -12 \gamma _1 - \left(18 \gamma \gamma _1 \right)) 
		A^{(1,0)}(0,0)
		\right.
		\\& \left. \quad
		+12 \gamma  
		A^{(1,1)}(0,0)
		+2 
		A^{(1,2)}(0,0)
		+6 \gamma A^{(2,0)}(0,0) \right) 
		\\&
		+ (9 \gamma ^4 +9 \left(\gamma _1\right){}^2 -  \left( 18 \gamma _1 \gamma ^2 \right) ) 
		A(0,0)
		+ 18 \gamma ^3 
		A^{(1,0)}(0,0)
		+ \left( 9 \gamma ^2 \right)
		A^{(1,1)}(0,0)
		\\& \quad
		+ \left(3 \gamma \right)
		A^{(1,2)}(0,0)
		+
		(3 \gamma ^2 -3 \gamma _1 )
		A^{(2,0)}(0,0)
		+\frac{1}{4} 
		A^{(2,2)}(0,0).
	\end{align}
	
	\begin{lem}
		We have
		\begin{align} \label{eq:productp}
			\prod_p
			&\frac{p^4-4 p^2+4 p-1}{p^4}
			\\&
			\approx
			0.21777871661953637832300751411944681313079775500136,
		\end{align}
		\begin{align}
			\label{eq:sump}
			\sum_p
			&\frac{3 (2 p-1) \log (p)}{p^3+p^2-3 p+1}
			\\&
			\approx
			2.5290661735809299292595871293018945923000922399444,
		\end{align}
		\begin{align}
			\sum_p 
			&\frac{9 p^4 \log ^2(p)}{\left(p^3+p^2-3 p+1\right)^2}
			\\&
			\approx
			6.4892240868025807879695316031935594971438999573128,
		\end{align}
	\begin{align}
			\sum_p
			&\frac{3 p (2 p-1) \left(p^2-p-1\right) \log ^2(p)}{\left(p^3+p^2-3 p+1\right)^2}
			\\& \approx
			2.7937396327899498121176904230895393701540841938169,
		\end{align}
	\begin{align}
			\sum_p
			&\frac{9 p^4 \left(p^3-p^2+5 p-3\right) \log ^3(p)}{\left(p^3+p^2-3 p+1\right)^3}
			\\& \approx
			13.924949838246429023458888451222757226018087649990,
		\end{align}
		and
		\begin{align}
			\sum_p
			&\frac{9 p^4 \left(p^6-2 p^5+29 p^4-16 p^3+31 p^2-30 p+9\right) \log ^4(p)}{\left(p^3+p^2-3 p+1\right)^4}
			\\&
			\approx
			51.561612317854622568503183873771816289674440542631.
		\end{align}
	\end{lem}
	\begin{proof}
		We show \eqref{eq:productp} and \eqref{eq:sump}; the remaining four estimates follow similarly.
		Let
		\begin{equation} \label{eq:P(s)}
			P(s) = \sum_{p} \frac{1}{p^s},\
			(\Re s>1).
		\end{equation}
		The command {\tt PrimeZetaP[s]} in Mathematica evaluates the function $P(s)$ to arbitrary numerical precision. The idea is thus to write the above product and sums over primes as linear combinations of $P(s)$. Let $A$ and $B$ denote the left side of \eqref{eq:productp} and \eqref{eq:sump}, respectively. For convergence issues, we separate out the prime $p=2$. We have
		\begin{equation}
			A 
			=
			\frac{7}{16} 
			\exp \left(
			\sum_{p>2} \log 
			\left(
			1- \frac{4}{p^2}
			+ \frac{4}{p^3}
			- \frac{1}{p^4}
			\right) 
			\right).
		\end{equation}
		We expand $\log$ as a series in powers of $1/p$, say
		\begin{equation}
			\log 
			\left(
			1- \frac{4}{p^2}
			+ \frac{4}{p^3}
			- \frac{1}{p^4}
			\right) 
			= 
			\sum_{N=1}^\infty
			a_N p^{-N}.
		\end{equation}
		Since $p>2$, the above series converges absolutely. Thus, interchanging the order of the summations, we get, by \eqref{eq:P(s)},
		\begin{equation}
			A
			=
			\frac{7}{16}
			\left(
			\sum_{N=1}^\infty
			a_N
			\sum_{p>2} \frac{1}{p^N}
			\right)
			=
			\frac{7}{16}
			\left(
			\sum_{N=1}^\infty
			a_N
			\left(
			P(N)
			-
			\frac{1}{2^N}
			\right)
			\right).
		\end{equation}
		Taking the first 1000 terms in the above in Mathematica gives $A$ to 100 digits accuracy.
		
		Next, if we took derivatives of \eqref{eq:P(s)}, we get
		\begin{equation}
			P^{(\ell)}(s)
			=
			(-1)^\ell
			\sum_p
			\frac{\log^\ell (p)}{p^s},\
			\Re s >1.
		\end{equation}
		Thus, we can rewrite $B$ as
		\begin{equation}
			B
			= 
			\frac{9 \log (2)}{7}
			-
			\sum_{N=1}^\infty
			b_N 
			\left(
			P'(N) - \frac{\log 2}{2^N}
			\right),
		\end{equation}
		where
		\begin{equation}
			\frac{3 (2 p-1)}{p^3+p^2-3 p+1}
			= \sum_{N=1}^\infty
			b_N p^{-N}.
		\end{equation}
		The first 1000 terms gives $B$ to 75 digits precision. A sample Mathematica code used to compute the constant $B$ is include below.
		
		\bigskip
		{\tt \color{purple}
		Block[\{\$MaxExtraPrecision = 1000\},
		
		Do[CC = Join[\{0\}, Series[(3 (-1 + 2 p))/(1 - 3 p + p\^{}2 + p\^{}3) //. p -> 1/x, \{x, 0,	t\}][[3]]];
		
		Print[N[-Sum[CC[[k]]*(PrimeZetaP'[k] + Log[2]/2\^{}k), \{k, 1, Length[CC]\}] + 9 Log[2]/7, 75]], \{t, 500, 1000, 100\}]]
		}\\

		In particular, this constant \eqref{eq:sump} is sequence \href{https://oeis.org/A354709}{\path{A354709}} in the On-Line Encyclopedia of Integer Sequences.
	\end{proof}

	From this Lemma, we get
	
	\begin{lem} 
		\label{lemma:sixEstimates}
		We have the following six estimates, with $A(s,w)$ given in \eqref{eq:Asw},
		\begin{align}
			A(0,0)
			&=
			\prod_p
			\frac{p^4-4 p^2+4 p-1}{p^4}
			\\&
			\approx
			0.21777871661953637832300751411944681313079775500136,
		\end{align}
		\begin{align}
			A^{(1,0)}(0,0)
			&=
			A^{(0,1)}(0,0)
			=
			A(0,0)
			\sum_p
			\frac{3 (2 p-1) \log (p)}{p^3+p^2-3 p+1}
			\\&
			\approx
			0.5507767855283365397996797117267309614310491736309,
		\end{align}
		\begin{align}
			A^{(1,1)}(0,0)
			&=
			A^{(1,0)}(0,0)
			\sum_p
			\frac{3 (2 p-1) \log (p)}{p^3+p^2-3 p+1}
			- A(0,0)
			\sum_p 
			\frac{9 p^4 \log ^2(p)}{\left(p^3+p^2-3 p+1\right)^2}
			\\&
			\approx
			-0.0202639560070943835323319895802569693120443555261,
		\end{align}
		\begin{align}
			A^{(2,0)}(0,0)
			&=
			A^{(1,0)}(0,0)
			\sum_p
			\frac{3 (2 p-1) \log (p)}{p^3+p^2-3 p+1}
			-
			A(0,0)
			\sum_p
			\frac{3 p (2 p-1) \left(p^2-p-1\right) \log ^2(p)}{\left(p^3+p^2-3 p+1\right)^2}
			\\ &
			\approx
			0.7845339056752244929584711968462575268503571131850,
		\end{align}
		\begin{align}
			A^{(2,1)}(0,0)
			&=
			A^{(1,1)}(0,0)
			\sum_p
			\frac{3 (2 p-1) \log (p)}{p^3+p^2-3 p+1}
			-
			A^{(1,0)}(0,0)
			\sum_p
			\frac{9 p^4 \log ^2(p)}{\left(p^3+p^2-3 p+1\right)^2}
			\\&
			-
			A^{(0,1)}(0,0)
			\sum_p
			\frac{3 p (2 p-1) \left(p^2-p-1\right) \log ^2(p)}{\left(p^3+p^2-3 p+1\right)^2}
			+
			A(0,0)
			\sum_p
			\frac{9 p^4 \left(p^3-p^2+5 p-3\right) \log ^3(p)}{\left(p^3+p^2-3 p+1\right)^3}
			\\& 
			\approx
			-2.131532098569090941134519992703368488331974362859,
		\end{align}
		\begin{align}
			A^{(2,2)}(0,0)
			&=
			A^{(1,2)}(0,0)
			\sum_p
			\frac{3 (2 p-1) \log (p)}{p^3+p^2-3 p+1}
			- 2
			A^{(1,1)}(0,0)
			\sum_p
			\frac{9 p^4 \log ^2(p)}{\left(p^3+p^2-3 p+1\right)^2}
			\\& \quad
			+ 	A^{(1,0)}(0,0)
			\left(
			\sum_p
			\frac{3 p (2 p-1) \left(p^2-p-1\right) \log ^2(p)}{\left(p^3+p^2-3 p+1\right)^2}
			+ 2
			\sum_p
			\frac{9 p^4 \left(p^3-p^2+5 p-3\right) \log ^3(p)}{\left(p^3+p^2-3 p+1\right)^3}
			\right)
			\\& \quad
			-
			A^{(0,2)}(0,0)
			\sum_p
			\frac{3 p (2 p-1) \left(p^2-p-1\right) \log ^2(p)}{\left(p^3+p^2-3 p+1\right)^2}
			\\& \quad
			-
			A(0,0)
			\sum_p
			\frac{9 p^4 \left(p^6-2 p^5+29 p^4-16 p^3+31 p^2-30 p+9\right) \log ^4(p)}{\left(p^3+p^2-3 p+1\right)^4}
			\\&
			\approx	
			-1.67079109287503595276150635884376764502678366004.
		\end{align}
	\end{lem}
	Thus, by Lemma \ref{lemma:sixEstimates}, equation \eqref{eq:m3'u1} becomes
	\begin{align}
		m_3'(u,1)
		&=
		0.05444467915488409458075187852986170328269943875033898441206910088090
		\\& \qquad
		662277806315513948136095589094142 \log ^4(u)
		\\&+ 0.92789264567318112587696644079295218608899487446511344258812211888336
		\\& \qquad
		5567774 \log ^3(u)
		\\&+ 4.15153784504081202209511918786805421354550644209045040748994318151802
		\\& \qquad
		271627 \log ^2(u)
		\\&+ 4.72025542659273594740805589875780984534505249010249113402891449603750
		\\& \qquad
		82512 \log (u)
		\\& + 0.965099958579599805763235700877379606169258597101357598307016332607213922.
	\end{align}
	Hence, integrating the above gives the right side of \eqref{eq:m3u1}, ignoring the constant  and the power-saving error terms.
\end{proof}

The error term in \eqref{eq:m31} is plotted in Figure \ref{figure:deltamethod}, showing that it is bounded by $\pm 1050 X^{0.51}$ for $1\le X\le 10^6$. This data thus shows that Conjecture \ref{conj:LevelofDistribution} agrees with the evaluation of $m_3(X,1)$ in Theorem \ref{thm:D33XhFullPoly}.

In the next section, we investigate the error term in the classical correlation of the usual divisor function $\tau(n)$.

\section{Proof of Theorem \ref{theorem:D22Xh} and numerical evidence for Conjecture \ref{conj:LevelofDistribution}: Square-root cancellation in the error term of the classical correlation $\sum_{n\le X} \tau(n) \tau(n+1)$}
\label{section:D22Xh}

It is a classic result of Ingham \cite{Ingham1927} from 1927 that, as $X\to \infty$,
\begin{equation} \label{eq:Ingham137}
	D_{2,2}(X,h) 
	\sim 
	\frac{6}{\pi^2} \sum_{d \mid h} \frac{1}{d}
	X \log^2 X.
\end{equation}
A little more than half-century latter, Heath-Brown \cite[Theorem 2]{HeathBrown1979} in 1979 refined Ingham's asymptotic to an equality with all lower order terms of the form
\begin{equation} \label{eq:D22209}
	D_{2,2} (X,h)
	= 
	m(X, h) + E(X,h),
\end{equation}
where
\begin{equation} \label{eq:mXh135}
	m(X,h)
	=
	\sum_{i=0}^2
	c_i(h) X \log^i X,
\end{equation}
and, for any $\epsilon>0$,
\begin{equation}
	E(X,h)
	\ll 
	X^{5/6+\epsilon},\ ( h  \le X^{5/6}),
\end{equation}
for some absolute constants $c_i(h)$. In this last section, we apply the procedure in Section \ref{section:compositeh} to refine \eqref{eq:D22209} by explicitly computing the three constants $c_i(h)$ from our $M_{2,2}(X,h)$, in particular, recovering the asymptotic \eqref{eq:Ingham137}. We also discuss the behavior of the error term $E_{2,2}(X,1)$, showing that it exhibits square root cancellation, supported by numerical evidence.

Fortunately, when $k=\ell=2$, the bound \eqref{eq:412} is known unconditionally, with an error term of size $\ll X^{\frac{1}{2}+ \frac{1}{3} + \epsilon} = O(X^{5/6+\epsilon})$.
\begin{theorem} \label{lemma:tau2/3230}
	Let $\epsilon>0$. Then, we have, uniformly for $1\le q \le X^{2/3}$,
	\begin{equation}
		\Delta(X,q,h)
		\ll X^{1/3+\epsilon}.
	\end{equation}
\end{theorem}
\begin{proof}
	This is a classic unpublished result of Selberg, Hooley, and others all from the mid 1950's. A formal proof can be found in \cite[Corollary 1, pg. 409]{HeathBrown1979}.
\end{proof}
While only a level of distribution $1/2$ for $\tau(n)$ is needed to prove \eqref{eq:D22209}, Theorem \ref{lemma:tau2/3230} gives that the divisor function is actually well distributed in arithmetic progressions to a higher level of $2/3$. Using Theorem \ref{lemma:tau2/3230}, we derive in this last section the following unconditional

\setcounter{thm}{3}
\begin{thm} \label{theorem:D22Xh}
	Let $\epsilon>0$. We have, uniformly for all $1\le h\le X^{1/2}$, the asymptotic equality
	\begin{equation} \label{eq:D22Xh}
		\sum_{n\le X}
		\tau(n) \tau(n+h)
		=
		M_{2,2}(X,h)
		+ 
		E_{2,2}(X,h),
	\end{equation}
	where
	\begin{equation} \label{eq:M22Xh}
		M_{2,2}(X,h)
		=X \left(
		c_2(h) \log^2 X
		+ c_1(h) \log X
		+ c_0(h)
		\right),
	\end{equation}
	with
	\begin{equation}
		c_2(h)
		=
		\frac{6}{\pi^2} \sum_{d \mid h} \frac{1}{d},
	\end{equation}
	\begin{equation}
		c_1(h)
		=
		(4 \gamma - 2) f_h(1,0)
		+
		2 f_h^{(0,1)}(1,0)+f_h^{(1,0)}(1,0),
	\end{equation}
	and
	\begin{align}
		c_0(h)
		&=
		2 \left(-f_h^{(0,1)}(1,0)+\gamma  \left(2 f_h^{(0,1)}(1,0)+f_h^{(1,0)}(1,0)-f_h(1,0)\right)+f_h^{(1,1)}(1,0)+2 \gamma ^2 f_h(1,0)\right)
		\\&  \quad
		+ 
		f_h^{(1,0)}(1,0)
		+2 (\gamma -1) f_h(1,0),
	\end{align}
	with the constants $f_h$, $f_h^{(0,1)}$, $f_h^{(1,0)}$, and $f_h^{(1,1)}$ at $(1,0)$ depending only on $h$ given in Lemmas \ref{lemma:fh10} and \ref{lemma:fhDerivatives} below, and with the error term satisfying
	\begin{equation} \label{eq:E22unconditional}
		E_{2,2}(X,h)
		\ll
		X^{5/6+\epsilon}.
	\end{equation}
\end{thm}

\begin{proof}
	From \eqref{eq:Dk138} with \eqref{eq:Ckl2}, \eqref{eq:fklh}, \eqref{eq:Ap141}, \eqref{eq:Bp1034}, $k=\ell=2$, and by Lemma \ref{lemma:tau2/3230}, we have
	\begin{align} \label{eq:426pm}
		D_{2,2}(X,h)
		&=
		2\underset{\substack{s=1\\ w=0}}{\mathrm{Res}}
		\left(
		\frac{X^{\frac{1}{2}w} (X+h)^s}{w s}
		\sum_{n=1}^\infty
		\frac{F_h(s; n)}{n^{w}}
		\right)
		\\&
		-
		\underset{\substack{s=1\\ w=1}}{\mathrm{Res}}
		\left(
		\frac{X^{ \frac{1}{2}w }}{w s}
		\sum_{n=1}^\infty
		\frac{F_h(s; n) (nX^{1/2} + h)^{s}}{n^{w} }
		\right)
		+ O\left( X^{5/6 + \epsilon}  \right),
	\end{align}
	where
	\begin{equation}
		F_h(s; n)
		=
		\frac{1}{\varphi\left( \frac{n}{(h, n)} \right)}
		\sum_{ \substack{\ell =1\\ \left( \ell, \frac{n}{(h, n) =1 } \right) = 1} }^\infty
		\frac{\tau(\ell (h, n))}{(\ell (h, n) )^s}.
	\end{equation}
	By multiplicativity and Euler products, we have, from \eqref{eq:Ap141}, \eqref{eq:1023}, \eqref{eq:Bp1034}, and \eqref{eq:BBp541}, with $k=\ell=2$,
	\begin{equation} \label{eq:Fh435}
		\sum_{n=1}^\infty
		\frac{F_h(s; n)}{n^{w}}
		=
		\zeta^2(s) \zeta(w+1)
		f_h(s;w),
	\end{equation}
	where
	\begin{equation} \label{eq:fh550}
		f_h(s;w)
		=
		\prod_{p \mid h}
		\frac{A_p(s; w; h)}{B_p(s;w)}
		\prod_p BB_p(s;w)
	\end{equation}
	with
	\begin{align} \label{eq:Ap442}
		A_p(s;w;h)
		=
		\zeta^2(s)
		&\left(
		1
		+
		\frac{  2p (p-1) \left(p^{\nu_p(h)}-1\right)-\nu_p(h) (p-1)^2 }{ p^{\nu_p(h)+1}(p-1)^2}
		\right.
		\\& \left. \quad
		+
		(\nu_p(h)+1)
		\frac{\left(1-\frac{1}{p^s}\right)^2}{1-\frac{1}{p}}
		\frac{p^{-\nu_p(h) (s+w)}}{p^{w+1} - 1}
		\right),
	\end{align}
	\begin{equation} \label{eq:Bp442}
		B_p(s;w)
		=
		\zeta^2(s)
		\left(
		1
		+
		\frac{\left(1-\frac{1}{p^s}\right)^2}{1-\frac{1}{p}} \frac{1}{p^{w+1}-1}
		\right)
	\end{equation}
	and
	\begin{equation} \label{eq:BBp638}
		BB_p(s;w)
		=
		\left( 1- \frac{1}{p^{w+1}} \right)
		\left(
		1
		+
		\frac{\left(1-\frac{1}{p^s}\right)^2}{1-\frac{1}{p}} \frac{1}{p^{w+1}-1}
		\right),
	\end{equation}
	with $f_h(s;w)$ converging in a wider region. Hence, by \eqref{eq:Fh435}, \eqref{eq:426pm} becomes
	\begin{align} \label{eq:427530}
		D_{2,2}(X,h)
		&=
		2\underset{\substack{s=1\\ w=0}}{\mathrm{Res}}
		\left(
		\frac{X^{\frac{1}{2}w + s}}{w s}
		\zeta^2(s) \zeta(w+1) f_h(s;w)
		\right)
		\\&-
		\underset{\substack{s=1\\ w=1}}{\mathrm{Res}}
		\left(
		\frac{X^{ \frac{1}{2}(w +s) }}{w s}
		\zeta^2(s) \zeta(w-s+1) f_h(s;w-s)
		\right)
		+ O\left( X^{5/6 + \epsilon}  \right).
	\end{align}
	The first residue of the above is equal to
	\begin{align} \label{eq:1stresidue}
		\frac{1}{2} X 
		&\left(
		f_h(1,0) \log ^2(X)
		+ \left( 2 f_h^{(0,1)}(1,0)+f_h^{(1,0)}(1,0)+(4 \gamma -1) f_h(1,0) \right) \log (X)
		\right.
		\\& \left.
		+
		2 \left(-f_h^{(0,1)}(1,0)+\gamma  \left(2 f_h^{(0,1)}(1,0)+f_h^{(1,0)}(1,0)-f_h(1,0)\right)+f_h^{(1,1)}(1,0)+2 \gamma ^2 f_h(1,0)\right)
		\right)
	\end{align}
	and the second
	\begin{align} \label{eq:2ndresidue}
		X \left(
		f_h(1,0) \log (X)
		+ f_h^{(1,0)}(1,0)
		+2 (\gamma -1) f_h(1,0)\right).
	\end{align}
	Thus, by \eqref{eq:1stresidue} and \eqref{eq:2ndresidue}, \eqref{eq:427530} becomes
	\begin{align}
		D_{2,2}(X,h)
		&=
		X 
		\left(
		f_h(1,0) \log ^2(X)
		+ \left( 
		(4 \gamma - 2) f_h(1,0)
		+ 2 f_h^{(0,1)}(1,0)+f_h^{(1,0)}(1,0) \right) \log (X)
		\right.
		\\& \left.
		+
		2 \left(-f_h^{(0,1)}(1,0)+\gamma  \left(2 f_h^{(0,1)}(1,0)+f_h^{(1,0)}(1,0)-f_h(1,0)\right)+f_h^{(1,1)}(1,0)+2 \gamma ^2 f_h(1,0)\right)
		\right.
		\\& \left. \quad
		+ 
		f_h^{(1,0)}(1,0)
		+2 (\gamma -1) f_h(1,0)
		\right)
		+ O\left(X^{5/6 +\epsilon}\right).
	\end{align}
	It remains to compute the function $f_h$ and its derivatives at $(1,0)$. We do so in the following two lemmas, which will complete the proof of Theorem \ref{theorem:D22Xh}.
	
	\begin{lem} \label{lemma:fh10}
		We have
		\begin{equation} \label{eq:fh10512}
			f_h(1,0)
			=
			\frac{6}{\pi^2}
			\sum_{d \mid h} \frac{1}{d}.
		\end{equation}
	\end{lem}
	\begin{proof}
		By \eqref{eq:fh550}, \eqref{eq:Ckl} and \eqref{eq:Ckl2}, we have
		\begin{align}
			f_h(1,0)
			=
			\prod_{p \mid h} \frac{A_p(1; 0; h)}{B_p(1;0)}
			\prod_{p} BB_p(1;0)
			=
			\prod_{p\mid h}
			\frac{p^{-\nu_p(h)} \left(p^{\nu_p(h)+1}-1\right)}{p-1}
			\prod_p \left( 1- \frac{1}{p^2} \right).
		\end{align}
		But
		\begin{equation}
			\prod_p \left( 1- \frac{1}{p^2} \right)
			= 
			\frac{1}{\zeta(2)}
			=
			\frac{6}{\pi^2},
		\end{equation}
		and
		\begin{equation}
			\prod_{p\mid h}
			\frac{p^{-\nu_p(h)} \left(p^{\nu_p(h)+1}-1\right)}{p-1}
			=
			\prod_{p\mid h}
			\frac{p^{-(\nu_p(h) + 1)} - 1}{p^{-\nu_p(h)} - 1}
			= 
			\sum_{d \mid h} \frac{1}{d},
		\end{equation}
		where the last equality follows from \cite[Theorem 274, pg. 311]{HardyWright1938}. Hence, \eqref{eq:fh10512} follows.
	\end{proof}

	\begin{lem} \label{lemma:fhDerivatives}
		We have the following three estimates
		\begin{align}
			f_h^{(0,1)}(1,0)
			&=
			\frac{6}{\pi^2}
			\sum_{d \mid h} \frac{1}{d}
			\left(
			\sum_p \frac{\log (p)}{p^2-1}
			+
			\sum_{p \mid h}		
			\frac{\left(\nu_p(h) (p-1)-p \left(p^{\nu_p(h)}-1\right)\right) \log (p)}{(p-1) \left(p^{\nu_p(h)+1}-1\right)}
			\right),
		\end{align}
		\begin{align}
			f_h^{(1,0)}(1,0)
			&=
			\frac{6}{\pi^2}
			\sum_{d \mid h} \frac{1}{d}
			\left(
			\sum_p \frac{2\log (p)}{p^2-1}
			-
			\sum_{p\mid h}
			\frac{2 \left(p \left(p^{\nu_p(h)}-1\right)-\nu_p(h) (p)+\nu_p(h)\right) \log (p)}{(p-1) \left(p^{\nu_p(h)+1}-1\right)}
			\right),
		\end{align}
	and
		\begin{align} \label{eq:fh11}
			&f_h^{(1,1)}(1,0)
			=
			\frac{6}{\pi^2}
			\sum_{d \mid h} \frac{1}{d}
			\left(
			\sum_p \frac{\log (p)}{p^2-1}
			+
			\sum_{p \mid h}		
			\frac{\left(\nu_p(h) (p-1)-p \left(p^{\nu_p(h)}-1\right)\right) \log (p)}{(p-1) \left(p^{\nu_p(h)+1}-1\right)}
			\right)
			\\& \quad \times
			\left(
			\sum_p \frac{2\log (p)}{p^2-1}
			-
			\sum_{p\mid h}
			\frac{2 \left(p \left(p^{\nu_p(h)}-1\right)-\nu_p(h) (p)+\nu_p(h)\right) \log (p)}{(p-1) \left(p^{\nu_p(h)+1}-1\right)}
			\right)
			\\& +
			\frac{6}{\pi^2}
			\sum_{d \mid h} \frac{1}{d}
			\left(
			- \sum_p \frac{2 p^2 \log ^2(p)}{\left(p^2-1\right)^2}
			\right.
			\\& \left.
			+
			\prod_{p \mid h}
			\frac{2 p \left(2 \nu_p(h) (h+2) p^{\nu_p(h)+1}-(\nu_p(h)+1)^2 p^{\nu_p(h)+2}+p^{2 \nu_p(h)+2}-(\nu_p(h)+1)^2 p^{\nu_p(h)}+1\right) \log ^2(p)}{(p-1)^2 \left(p^{\nu_p(h)+1}-1\right)^2}
			\right).
		\end{align}	
	\end{lem}
	\begin{proof}
		By \eqref{eq:BBp638}, we have
		\begin{equation} \label{eq:BBp01}
			\sum_p
			\frac{d}{dw} \log BB_p(1,0)
			=
			\sum_p \frac{\log (p)}{p^2-1},
		\end{equation}
		\begin{equation} \label{eq:BBp10}
			\sum_p
			\frac{d}{ds} \log BB_p(1,0)
			=
			\sum_p \frac{2 \log (p)}{p^2-1},
		\end{equation}
		and
		\begin{equation} \label{eq:BBp11}
			\sum_p
			\frac{d^2}{dsdw} \log BB_p(1,0)
			=
			- \sum_p \frac{2 p^2 \log ^2(p)}{\left(p^2-1\right)^2}.
		\end{equation}
		Thus, by \eqref{eq:fh550}, \eqref{eq:fh10512}, \eqref{eq:BBp01}, \eqref{eq:BBp10}, and \eqref{eq:BBp11}, we get
		\begin{align}
			f_h^{(0,1)}(1,0)
			&= 
			f_h(1,0)
			\frac{d}{dw} \log f_h(s,w) |_{(s,w) = (1,0)}
			\\&=
			f_h(1,0)
			\left(
			\sum_p \frac{d}{dw} \log BB_p(s;w)
			+
			\sum_{p \mid h}
			\frac{d}{dw} \log \frac{A_p(s; w; h)}{B_p(s;w)}
			\right)_{(s,w) = (1,0)}
			\\&
			=
			\frac{6}{\pi^2}
			\sum_{d \mid h} \frac{1}{d}
			\left(
			\sum_p \frac{\log (p)}{p^2-1}
			+
			\sum_{p \mid h}		
			\frac{\left(\nu_p(h) (p-1)-p \left(p^{\nu_p(h)}-1\right)\right) \log (p)}{(p-1) \left(p^{\nu_p(h)+1}-1\right)}
			\right),
		\end{align}
		\begin{align}
			f_h^{(1,0)}(1,0)
			&= 
			f_h(1,0)
			\frac{d}{ds} \log f_h(s,w) |_{(s,w) = (1,0)}
			\\&=
			f_h(1,0)
			\left(
			\sum_p \frac{d}{ds} \log BB_p(s;w)
			+
			\sum_{p \mid h}
			\frac{d}{ds} \log \frac{A_p(s; w; h)}{B_p(s;w)}
			\right)_{(s,w) = (1,0)}
			\\&
			=
			\frac{6}{\pi^2}
			\sum_{d \mid h} \frac{1}{d}
			\left(
			\sum_p \frac{2\log (p)}{p^2-1}
			-
			\sum_{p\mid h}
			\frac{2 \left(p \left(p^{\nu_p(h)}-1\right)-\nu_p(h) (p)+\nu_p(h)\right) \log (p)}{(p-1) \left(p^{\nu_p(h)+1}-1\right)}
			\right),
		\end{align}
		and
		\begin{align}
			&f_h^{(1,1)}(1,0)
			= 
			\frac{d}{dw} f_h^{(1,0)}(s,w)|_{(s,w)=(1,0)}
			=
			\frac{d}{dw}
			\left(
			f_h(s,w)
			\frac{d}{ds} \log f_h(s,w) \right)_{(s,w) = (1,0)}
			\\&
			=
			\left(
			f_h^{(0,1)}(s,w) \frac{d}{ds} \log f_h(s,w) 
			+
			f_h(s,w)
			\frac{d^2}{dsdw} \log f_h(s,w)
			\right)_{(s,w) = (1,0)},
		\end{align}
		which gives the right side of \eqref{eq:fh11}.
	\end{proof}
	This completes the proof of Theorem \ref{theorem:D22Xh}.
\end{proof}

In particular, we have the following consequence to Theorem \ref{theorem:D22Xh} for $h=1$.
\setcounter{cor}{1}
\begin{cor}
	We have, for any $\epsilon>0$, with at least 148 digits accuracy in the coefficients,
	\begin{align}
		&M_{2,2}(X,1)
		=
		X \left(\frac{6}{\pi ^2}\log ^2(X)
		\right.
		\\& \left. 
		+1.5737449203324910789070569280484417010544014980534581993991047787172106559673
		\right.
		\\& \left. \quad
		1173018329789033856157663793482022187619702084359231966550508901828044158 \log (X)
		\right.
		\\& \left. 
		-0.5243838319228249988207213304174247109766097340170991428485246582967458363611
		\right.
		\\& \left. \quad
		4606090215515124475866524185215534024889460792901985996741204565400064583\right)
		+ O(X^\epsilon).
	\end{align}
\end{cor}
\begin{proof}
	We have
	\begin{align}
		\sum_p &\frac{\log (p)}{p^2-1}
		\\& \approx
		0.569960993094532806399864360019730002403482280806930979558125010990350610050
	\end{align}
	and
	\begin{align}
		\sum_p &\frac{p^2 \log^2(p)}{(p^2-1)^2}
		\\&\approx
		0.884481833963523885196536153870651168588667332638711335184294712832630231963.
	\end{align}
	When $h=1$, \eqref{eq:fh550} reduces to
	\begin{equation}
		f_1(s;w)
		=
		\prod_p BB_p(s;w)
	\end{equation}
	and there is no local factor. Hence, the estimates in Lemmas \ref{lemma:fh10} and \ref{lemma:fhDerivatives} simplify to
	\begin{equation}
		f_1(1,0)
		=
		\frac{6}{\pi^2},
	\end{equation}
	\begin{align}
		&f_1^{(0,1)}(0,1)
		=
		\frac{6}{\pi^2}
		\sum_p \frac{\log (p)}{p^2-1}
		\\&
		\approx 
		0.346494734701802213346160816867709151548899264204041698651043406973780662935,
	\end{align}
	\begin{align}
		&f_1^{(1,0)}(0,1)
		= 2 f_1^{(0,1)}(0,1)
		\\&
		\approx
		0.692989469403604426692321633735418303097798528408083397302086813947561325869,
	\end{align}
	and
	\begin{align}
		&f_1^{(1,1)}(0,1)
		=
		\frac{12}{\pi^2} \left(
		\left(\sum_p \frac{\log (p)}{p^2-1}\right)^2
		-
		\sum_p \frac{p^2 \log^2(p)}{(p^2-1)^2}
		\right)
		\\&
		\approx
		-0.68042398974262717192610795266802886217030580133549111824673457509413466415.
	\end{align}
	Hence, with the four estimates above, \eqref{eq:M22Xh} simplifies to give \eqref{eq:D22X1}.
\end{proof}

The error term $E_{2,2}(X,1) = D_{2,2}(X,1) - M_{2,2}(X,1)$ is plotted in Figure \ref{figure:D22X1}, showing a fluctuating behavior, but seems to be bounded by a constant times a fractional power of $X$. In Figure \ref{figure:D22X1LogLogPlot}, a log-log-plot of the error term $E_{2,2}(X,1)$ is graphed to numerically determine the constants $\alpha$ and $C$ such that $|E_{2,2}(X,1)| \le C X^\alpha$.
\begin{figure}
	\includegraphics[scale=1]{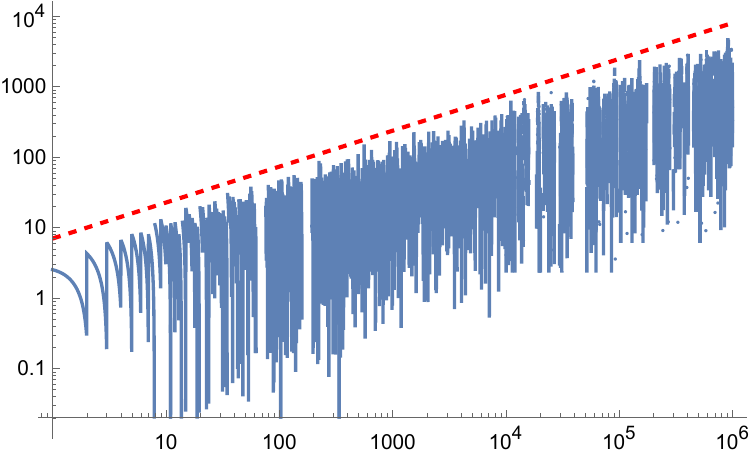}
	\caption{A log-log-plot of the error term $E_{2,2}(X,1)$, for $1\le X\le 10^6$, with slope of dashed line approximately 0.51 and $y$-intercept around 7, which numerically suggests that $|E_{2,2}(X,1)| \le 7 X^{0.51}$.}
	\label{figure:D22X1LogLogPlot}
\end{figure}
This is simply because, if we took log's of both sides of this equation, then the exponent $\alpha$ is equal to the slope and $C$ is given by the $y$-intercept of this straight line. Thus, from Figure \ref{figure:D22X1LogLogPlot}, pick two best points we compute $\alpha \approx 0.51$ and $C\approx 7$.
This suggests that
\begin{equation} \label{eq:E22X1}
	|E_{2,2}(X,1)| \le 7  X^{0.51},
\end{equation}
which, in particular, is much sharper than \eqref{eq:E22unconditional}. Therefore, \eqref{eq:E22X1} shows that the corresponding error term exhibits square-root cancellation, which
provides numerical evidence to support  Conjecture \ref{conj:LevelofDistribution}.

\section*{Acknowledgments}

I am very grateful to Brian Conrey for his suggestion to investigate the shifted convolution $\tau_3(n)\tau_3(n+1)$ and for helpful conversations, in particular, pointing my attention to \cite{BaluyotConrey2022}, and for thorough reading of Section \ref{section:ConreyGonek}. Many thanks to Nathan Ng and Brad Rodgers for useful comments and suggestions. Special thanks to Siegfred Baluyot for sending \cite{BaluyotConrey2022}, from which \eqref{eq:G3} appears. I also benefited from Mathematica code from V. Kotesovec in an OEIS comment (entry \href{https://oeis.org/A256392}{\path{A256392}}), which permits arbitrary precision in computing sums and products over primes. My gratitude also goes to everyone at AIM for great environment.

\section*{Appendix: Proof of Corollary \ref{cor:mypolynomialM33}}
\label{section:appendix}



\section*{Supplementary material}

# Proof of Corollary 1

## Summary of the proof

---

### All three residues

In[257]:= `Simplify[3 Residue[Residue[Residue[f1[s, w1, w2], {s, 1}], {w2, 0}], {w1, 0}] -`
`3 Residue[Residue[Residue[f2[s, w1, w2], {s, 1}], {w2, 1}], {w1, 0}] +`
`Residue[Residue[Residue[f3[s, w1, w2], {s, 1}], {w2, 1}], {w1, 1}]]`

Out[257]= $X$ (0.2872366477466194172216646178146459501660362743972222496189139 07447198 +
0.6778633108329803885415710830627336560032223227041353486881024 25159897 Log[X] +
2.0211960578798777794332424078475380946709150836991778926704060 3543881 Log[X]$^2$ +
0.7101139290536447475539589266735053729581971194637575049398457 15359739 Log[X]$^3$ +
0.0544446791548840945807518785298617032826994387503389844120691 008809066 Log[X]$^4$)

## Summary of the constants

In[174]:= `f1[s_, w1_, w2_] := X^((w1 + 2 w2 + 3 s) / 3) / w1 / w2 / s *`
`Zeta[s]^3 * Zeta[w1 + w2 + 1] * Zeta[w2 + 1] * A1[s, w1, w2];`
`f2[s_, w1_, w2_] := X^((w1 + 2 w2 + s) / 3) / w1 / w2 / s *`
`Zeta[s]^3 * Zeta[w1 + w2 + 1 - s] * Zeta[w2 + 1 - s] * A2[s, w1, w2];`
`f3[s_, w1_, w2_] := X^((w1 + w2 + s) / 3) / w1 / w2 / s *`
`Zeta[s]^3 * Zeta[w1 + 1 - s] * Zeta[w2 + 1 - s] * A3[s, w1, w2];`

`A1`$^{(0,0,1)}$`[1, 0, 0] := A1[1, 0, 0] * A111`$^{(0,0,1)}$`[1, 0, 0];`
`A1`$^{(0,1,0)}$`[1, 0, 0] := A1[1, 0, 0] * A111`$^{(0,1,0)}$`[1, 0, 0];`
`A1`$^{(1,0,0)}$`[1, 0, 0] := A1[1, 0, 0] * A111`$^{(1,0,0)}$`[1, 0, 0];`
`A1`$^{(0,1,1)}$`[1, 0, 0] :=`
`A1[1, 0, 0] A111`$^{(0,0,1)}$`[1, 0, 0] A111`$^{(0,1,0)}$`[1, 0, 0] + A1[1, 0, 0] A111`$^{(0,1,1)}$`[1, 0, 0];`
`A1`$^{(1,1,1)}$`[1, 0, 0] := A1[1, 0, 0] A111`$^{(0,0,1)}$`[1, 0, 0] A111`$^{(0,1,0)}$`[1, 0, 0] A111`$^{(1,0,0)}$`[1, 0, 0] +`
`A1[1, 0, 0] A111`$^{(0,1,1)}$`[1, 0, 0] A111`$^{(1,0,0)}$`[1, 0, 0] +`
`A1[1, 0, 0] A111`$^{(0,1,0)}$`[1, 0, 0] A111`$^{(1,0,1)}$`[1, 0, 0] +`
`A1[1, 0, 0] A111`$^{(0,0,1)}$`[1, 0, 0] A111`$^{(1,1,0)}$`[1, 0, 0] + A1[1, 0, 0] A111`$^{(1,1,1)}$`[1, 0, 0];`
`A1`$^{(0,0,2)}$`[1, 0, 0] := A1[1, 0, 0] A111`$^{(0,0,1)}$`[1, 0, 0]$^2$ + A1[1, 0, 0] A111`$^{(0,0,2)}$`[1, 0, 0];`
`A1`$^{(0,2,0)}$`[1, 0, 0] := A1[1, 0, 0] A111`$^{(0,1,0)}$`[1, 0, 0]$^2$ + A1[1, 0, 0] A111`$^{(0,2,0)}$`[1, 0, 0];`
`A1`$^{(1,0,1)}$`[1, 0, 0] :=`
`A1[1, 0, 0] A111`$^{(0,0,1)}$`[1, 0, 0] A111`$^{(1,0,0)}$`[1, 0, 0] + A1[1, 0, 0] A111`$^{(1,0,1)}$`[1, 0, 0];`
`A1`$^{(1,1,0)}$`[1, 0, 0] :=`

$A1[1, 0, 0] \, A111^{(0,1,0)}[1, 0, 0] \, A111^{(1,0,0)}[1, 0, 0] + A1[1, 0, 0] \, A111^{(1,1,0)}[1, 0, 0];$

$A1^{(1,2,0)}[1, 0, 0] := A1[1, 0, 0] \, A111^{(0,1,0)}[1, 0, 0]^2 \, A111^{(1,0,0)}[1, 0, 0] +$
$\quad A1[1, 0, 0] \, A111^{(0,2,0)}[1, 0, 0] \, A111^{(1,0,0)}[1, 0, 0] +$
$\quad 2 \, A1[1, 0, 0] \, A111^{(0,1,0)}[1, 0, 0] \, A111^{(1,1,0)}[1, 0, 0] + A1[1, 0, 0] \, A111^{(1,2,0)}[1, 0, 0];$

$A1^{(2,0,0)}[1, 0, 0] := A1[1, 0, 0] \, A111^{(1,0,0)}[1, 0, 0]^2 + A1[1, 0, 0] \, A111^{(2,0,0)}[1, 0, 0];$

$A1^{(2,0,1)}[1, 0, 0] := A1[1, 0, 0] \, A111^{(0,0,1)}[1, 0, 0] \, A111^{(1,0,0)}[1, 0, 0]^2 +$
$\quad 2 \, A1[1, 0, 0] \, A111^{(1,0,0)}[1, 0, 0] \, A111^{(1,0,1)}[1, 0, 0] +$
$\quad A1[1, 0, 0] \, A111^{(0,0,1)}[1, 0, 0] \, A111^{(2,0,0)}[1, 0, 0] + A1[1, 0, 0] \, A111^{(2,0,1)}[1, 0, 0];$

$A1^{(2,1,0)}[1, 0, 0] := A1[1, 0, 0] \, A111^{(0,1,0)}[1, 0, 0] \, A111^{(1,0,0)}[1, 0, 0]^2 +$
$\quad 2 \, A1[1, 0, 0] \, A111^{(1,0,0)}[1, 0, 0] \, A111^{(1,1,0)}[1, 0, 0] +$
$\quad A1[1, 0, 0] \, A111^{(0,1,0)}[1, 0, 0] \, A111^{(2,0,0)}[1, 0, 0] + A1[1, 0, 0] \, A111^{(2,1,0)}[1, 0, 0];$

$A1^{(2,1,1)}[1, 0, 0] := A1[1, 0, 0] \, A111^{(0,0,1)}[1, 0, 0] \, A111^{(0,1,0)}[1, 0, 0] \, A111^{(1,0,0)}[1, 0, 0]^2 +$
$\quad A1[1, 0, 0] \, A111^{(0,1,1)}[1, 0, 0] \, A111^{(1,0,0)}[1, 0, 0]^2 +$
$\quad 2 \, A1[1, 0, 0] \, A111^{(0,1,0)}[1, 0, 0] \, A111^{(1,0,0)}[1, 0, 0] \, A111^{(1,0,1)}[1, 0, 0] +$
$\quad 2 \, A1[1, 0, 0] \, A111^{(0,0,1)}[1, 0, 0] \, A111^{(1,0,0)}[1, 0, 0] \, A111^{(1,1,0)}[1, 0, 0] +$
$\quad 2 \, A1[1, 0, 0] \, A111^{(1,0,1)}[1, 0, 0] \, A111^{(1,1,0)}[1, 0, 0] +$
$\quad 2 \, A1[1, 0, 0] \, A111^{(1,0,0)}[1, 0, 0] \, A111^{(1,1,1)}[1, 0, 0] +$
$\quad A1[1, 0, 0] \, A111^{(0,0,1)}[1, 0, 0] \, A111^{(0,1,0)}[1, 0, 0] \, A111^{(2,0,0)}[1, 0, 0] +$
$\quad A1[1, 0, 0] \, A111^{(0,1,1)}[1, 0, 0] \, A111^{(2,0,0)}[1, 0, 0] +$
$\quad A1[1, 0, 0] \, A111^{(0,1,0)}[1, 0, 0] \, A111^{(2,0,1)}[1, 0, 0] +$
$\quad A1[1, 0, 0] \, A111^{(0,0,1)}[1, 0, 0] \, A111^{(2,1,0)}[1, 0, 0] + A1[1, 0, 0] \, A111^{(2,1,1)}[1, 0, 0];$

$A1^{(2,2,0)}[1, 0, 0] := A1[1, 0, 0] \, A111^{(0,1,0)}[1, 0, 0]^2 \, A111^{(1,0,0)}[1, 0, 0]^2 +$
$\quad A1[1, 0, 0] \, A111^{(0,2,0)}[1, 0, 0] \, A111^{(1,0,0)}[1, 0, 0]^2 +$
$\quad 4 \, A1[1, 0, 0] \, A111^{(0,1,0)}[1, 0, 0] \, A111^{(1,0,0)}[1, 0, 0] \, A111^{(1,1,0)}[1, 0, 0] +$
$\quad 2 \, A1[1, 0, 0] \, A111^{(1,1,0)}[1, 0, 0]^2 + 2 \, A1[1, 0, 0] \, A111^{(1,0,0)}[1, 0, 0] \, A111^{(1,2,0)}[1, 0, 0] +$
$\quad A1[1, 0, 0] \, A111^{(0,1,0)}[1, 0, 0]^2 \, A111^{(2,0,0)}[1, 0, 0] +$
$\quad A1[1, 0, 0] \, A111^{(0,2,0)}[1, 0, 0] \, A111^{(2,0,0)}[1, 0, 0] +$
$\quad 2 \, A1[1, 0, 0] \, A111^{(0,1,0)}[1, 0, 0] \, A111^{(2,1,0)}[1, 0, 0] + A1[1, 0, 0] \, A111^{(2,2,0)}[1, 0, 0];$


$A2^{(0,0,1)}[1, 0, 1] := A2[1, 0, 1] * A222^{(0,0,1)}[1, 0, 1];$

$A2^{(0,1,0)}[1, 0, 1] := A2[1, 0, 1] * A222^{(0,1,0)}[1, 0, 1];$

$A2^{(1,0,0)}[1, 0, 1] := A2[1, 0, 1] * A222^{(1,0,0)}[1, 0, 1];$

$A2^{(0,1,1)}[1, 0, 1] :=$
$\quad A2[1, 0, 1] \, A222^{(0,0,1)}[1, 0, 1] \, A222^{(0,1,0)}[1, 0, 1] + A2[1, 0, 1] \, A222^{(0,1,1)}[1, 0, 1];$

$A2^{(0,1,2)}[1, 0, 1] := A2[1, 0, 1] \, A222^{(0,0,1)}[1, 0, 1]^2 \, A222^{(0,1,0)}[1, 0, 1] +$
$\quad A2[1, 0, 1] \, A222^{(0,0,2)}[1, 0, 1] \, A222^{(0,1,0)}[1, 0, 1] +$
$\quad 2 \, A2[1, 0, 1] \, A222^{(0,0,1)}[1, 0, 1] \, A222^{(0,1,1)}[1, 0, 1] + A2[1, 0, 1] \, A222^{(0,1,2)}[1, 0, 1];$

$A2^{(0,0,2)}[1, 0, 1] := A2[1, 0, 1] \, A222^{(0,0,1)}[1, 0, 1]^2 + A2[1, 0, 1] \, A222^{(0,0,2)}[1, 0, 1];$

$A2^{(1,0,1)}[1, 0, 1] :=$
$\quad A2[1, 0, 1] \, A222^{(0,0,1)}[1, 0, 1] \, A222^{(1,0,0)}[1, 0, 1] + A2[1, 0, 1] \, A222^{(1,0,1)}[1, 0, 1];$

$A2^{(1,1,0)}[1, 0, 1] :=$
$\quad A2[1, 0, 1] \, A222^{(0,1,0)}[1, 0, 1] \, A222^{(1,0,0)}[1, 0, 1] + A2[1, 0, 1] \, A222^{(1,1,0)}[1, 0, 1];$

$A2^{(1,1,1)}[1, 0, 1] := A2[1, 0, 1] \, A222^{(0,0,1)}[1, 0, 1] \, A222^{(0,1,0)}[1, 0, 1] \, A222^{(1,0,0)}[1, 0, 1] +$
$\quad A2[1, 0, 1] \, A222^{(0,1,1)}[1, 0, 1] \, A222^{(1,0,0)}[1, 0, 1] +$

```
     A2[1, 0, 1] A222^(0,1,0)[1, 0, 1] A222^(1,0,1)[1, 0, 1] +
     A2[1, 0, 1] A222^(0,0,1)[1, 0, 1] A222^(1,1,0)[1, 0, 1] + A2[1, 0, 1] A222^(1,1,1)[1, 0, 1];
A2^(2,0,0)[1, 0, 1] := A2[1, 0, 1] A222^(1,0,0)[1, 0, 1]^2 + A2[1, 0, 1] A222^(2,0,0)[1, 0, 1];
A2^(2,1,0)[1, 0, 1] := A2[1, 0, 1] A222^(0,1,0)[1, 0, 1] A222^(1,0,0)[1, 0, 1]^2 +
     2 A2[1, 0, 1] A222^(1,0,0)[1, 0, 1] A222^(1,1,0)[1, 0, 1] +
     A2[1, 0, 1] A222^(0,1,0)[1, 0, 1] A222^(2,0,0)[1, 0, 1] + A2[1, 0, 1] A222^(2,1,0)[1, 0, 1];


A3^(0,0,1)[1, 1, 1] := A3[1, 1, 1] * A333^(0,0,1)[1, 1, 1];
A3^(0,1,0)[1, 1, 1] := A3[1, 1, 1] * A333^(0,1,0)[1, 1, 1];
A3^(1,0,0)[1, 1, 1] := A3[1, 1, 1] * A333^(1,0,0)[1, 1, 1];
A3^(0,0,2)[1, 1, 1] := A3[1, 1, 1] A333^(0,0,1)[1, 1, 1]^2 + A3[1, 1, 1] A333^(0,0,2)[1, 1, 1];
A3^(0,1,1)[1, 1, 1] :=
     A3[1, 1, 1] A333^(0,0,1)[1, 1, 1] A333^(0,1,0)[1, 1, 1] + A3[1, 1, 1] A333^(0,1,1)[1, 1, 1];
A3^(0,2,0)[1, 1, 1] := A3[1, 1, 1] A333^(0,1,0)[1, 1, 1]^2 + A3[1, 1, 1] A333^(0,2,0)[1, 1, 1];
A3^(1,0,1)[1, 1, 1] :=
     A3[1, 1, 1] A333^(0,0,1)[1, 1, 1] A333^(1,0,0)[1, 1, 1] + A3[1, 1, 1] A333^(1,0,1)[1, 1, 1];
A3^(1,1,0)[1, 1, 1] :=
     A3[1, 1, 1] A333^(0,1,0)[1, 1, 1] A333^(1,0,0)[1, 1, 1] + A3[1, 1, 1] A333^(1,1,0)[1, 1, 1];
A3^(2,0,0)[1, 1, 1] := A3[1, 1, 1] A333^(1,0,0)[1, 1, 1]^2 + A3[1, 1, 1] A333^(2,0,0)[1, 1, 1];


A1[1, 0, 0] :=
     0.217778716619536378323007514119446813130797755001355937648276403523626491112252620557\
       9254438235637656918339357748032`100.;
A2[1, 0, 1] := A1[1, 0, 0];
A3[1, 1, 1] := A1[1, 0, 0];


A111^(0,1,0)[1, 0, 0] :=
     0.843022057860309976419862376433964864100030746648133253729664187567119266887621591241\
       9165565392`75.;
A111^(0,0,1)[1, 0, 0] := 2 * A111^(0,1,0)[1, 0, 0];
A111^(0,1,1)[1, 0, 0] :=
     -1.190727816059283072434053614108579419745932058046617974308297869696307135808149094290\
       370550743993`75.;
A111^(0,2,0)[1, 0, 0] :=
     -1.450209087855249540828877087187312382773836051487588631246608261693943620465558872760\
       9919442550403`75.;
A111^(1,0,0)[1, 0, 0] := 3 * A111^(0,1,0)[1, 0, 0];
A111^(1,0,1)[1, 0, 0] :=
     -4.326149391201720525313021068795706331429266638208531355130930119053118430049057281918\
       885962161142`75.;
A111^(1,1,1)[1, 0, 0] :=
     4.922546395398470985919121032181312118733181980067924933193051055404874059727834172602\
       287213077508`75.;
A111^(1,0,2)[1, 0, 0] := 2 * A111^(1,1,1)[1, 0, 0];
A111^(1,1,0)[1, 0, 0] := A111^(1,0,1)[1, 0, 0] / 2;
```

$A111^{(1,2,0)}[1, 0, 0] :=$
  $5.203442844714798964018612580621705162127001410139188853357799078660730898691806653152 \backslash$
  $754725469967\text{`}75.$ ;

$A111^{(2,0,0)}[1, 0, 0] :=$
  $-2.793739632789949812117690423089539370154084193816941952109962433096011953452217955508 \backslash$
  $1858689463119\text{`}75.$ ;

$A111^{(2,1,0)}[1, 0, 0] :=$
  $4.641649946082143007819629483740919075339362549996661013028303032149017220763861683754 \backslash$
  $510494643289\text{`}75.$ ;

$A111^{(2,0,1)}[1, 0, 0] := 2 * A111^{(2,1,0)}[1, 0, 0]$;
$A111^{(0,0,2)}[1, 0, 0] := 2 * A111^{(0,1,1)}[1, 0, 0]$;
$A111^{(2,1,1)}[1, 0, 0] :=$
  $-16.662094909253587645981626246039311879645863234835104449613586577081516802089469692375 \backslash$
  $4264885014504\text{`}75.$

$A111^{(2,2,0)}[1, 0, 0] :=$
  $-16.136985712555634435795524534154684996066912955460012951066400738858144247842042449049 \backslash$
  $94068963096076\text{`}75.$ ;


$A222^{(0,1,0)}[1, 0, 1] := A111^{(0,1,0)}[1, 0, 0]$;
$A222^{(0,0,1)}[1, 0, 1] := 2 * A111^{(0,1,0)}[1, 0, 0]$;
$A222^{(1,0,0)}[1, 0, 1] := A111^{(0,1,0)}[1, 0, 0]$;
$A222^{(0,1,1)}[1, 0, 1] := A111^{(0,1,1)}[1, 0, 0]$;
$A222^{(0,0,2)}[1, 0, 1] := A111^{(0,0,2)}[1, 0, 0]$;
$A222^{(0,1,2)}[1, 0, 1] :=$
  $2.253473304856102095489192873464602805750241458417489508846347471733613639206760865244 \backslash$
  $491245099618\text{`}75.$ ;
$A222^{(1,0,1)}[1, 0, 1] := 2 * A222^{(1,1,0)}[1, 0, 1]$;
$A222^{(1,1,0)}[1, 0, 1] :=$
  $-0.9723468795415771902224569202892737459687012610576477032571671898302520792163795482 \backslash$
  $1505572162349\text{`}75.$ ;
$A222^{(2,0,0)}[1, 0, 1] :=$
  $3.477103517494925093640244486284714453212584966506884809535302065617610635029598419913 \backslash$
  $002810379413\text{`}75.$ ;
$A222^{(1,1,1)}[1, 0, 1] :=$
  $2.669073090542368890429928158716709312982940521650435424346703583671623039891572526148 \backslash$
  $00129547396\text{`}75.$ ;
$A222^{(2,1,0)}[1, 0, 1] :=$
  $-2.949969539858696868529419707157102356376759951721699344511451606927479878855554500889 \backslash$
  $4380831862888\text{`}75.$ ;


$A333^{(0,0,1)}[1, 1, 1] := A111^{(0,1,0)}[1, 0, 0]$;
$A333^{(0,1,0)}[1, 1, 1] := A333^{(0,0,1)}[1, 1, 1]$;
$A333^{(0,0,2)}[1, 1, 1] := A111^{(0,2,0)}[1, 0, 0]$;
$A333^{(0,2,0)}[1, 1, 1] := A111^{(0,2,0)}[1, 0, 0]$;
$A333^{(0,1,1)}[1, 1, 1] :=$
  $0.2594812717959664683948234730787329630279039934409706569383103919976364846574097767 24 \backslash$

```
        613359110604`75.;
A333⁽¹,¹,⁰⁾ [1, 1, 1] := A222⁽¹,¹,⁰⁾ [1, 0, 1];
A333⁽¹,⁰,¹⁾ [1, 1, 1] := A333⁽¹,¹,⁰⁾ [1, 1, 1];
A333⁽¹,⁰,⁰⁾ [1, 1, 1] := A111⁽⁰,¹,⁰⁾ [1, 0, 0];
A333⁽²,⁰,⁰⁾ [1, 1, 1] := A222⁽²,⁰,⁰⁾ [1, 0, 1];
```

# Definitions

$$\mathit{In[\cdot]:=} \ \ \texttt{A11[s\_, w1\_, w2\_]} := \left(1 - p^{-1-w2}\right) \ \left(1 - p^{-1-w1-w2}\right)$$

$$\left(1 + \frac{\left(1 - p^{-s}\right)^3 \left(\frac{1}{-1+p^{1+w2}} + \frac{1}{-1+p^{1+w1+w2}} + \frac{1}{(-1+p^{1+w2})\ (-1+p^{1+w1+w2})}\right)}{1 - \frac{1}{p}}\right) \ //. \ p \to \texttt{Prime[n]};$$

$$\texttt{A22[s\_, w1\_, w2\_]} := \left(1 - p^{s-1-w2}\right) \ \left(1 - p^{s-1-w1-w2}\right)$$

$$\left(1 + \frac{\left(1 - p^{-s}\right)^3 \left(\frac{1}{-1+p^{1+w2-s}} + \frac{1}{-1+p^{1+w1+w2-s}} + \frac{1}{(-1+p^{1+w2-s})\ (-1+p^{1+w1+w2-s})}\right)}{1 - \frac{1}{p}}\right) \ //. \ p \to \texttt{Prime[n]};$$

$$\texttt{A33[s\_, w1\_, w2\_]} := \left(1 - p^{s-1-w2}\right) \ \left(1 - p^{s-1-w1}\right)$$

$$\left(1 + \frac{\left(1 - p^{-s}\right)^3 \left(\frac{1}{-1+p^{1+w2-s}} + \frac{1}{-1+p^{1+w1-s}} + \frac{1}{(-1+p^{1+w2-s})\ (-1+p^{1+w1-s})}\right)}{1 - \frac{1}{p}}\right) \ //. \ p \to \texttt{Prime[n]};$$

# Computations of the constants

## Constants from the first residue

```
A111⁽⁰,¹,⁰⁾ [1, 0, 0] :=
  0.843022057860309976419862376433964864100030746648133253729664187567119266887621591241591`
    65565392`75.;
```

$$\mathit{In[\cdot]:=} \ \ \texttt{Simplify[D[Log[A11[s, w1, w2]], w1] //. \{s \to 1, w1 \to 0, w2 \to 0\}]}$$

$$\mathit{Out[\cdot]=} \ \ \frac{\texttt{Log[Prime[n]]}\ (-1 + 2\ \texttt{Prime[n]})}{1 - 3\ \texttt{Prime[n]} + \texttt{Prime[n]}^2 + \texttt{Prime[n]}^3}$$

```mathematica
In[•]:= N[Sum[(Log[Prime[n]] (-1 + 2 Prime[n]))/(1 - 3 Prime[n] + Prime[n]^2 + Prime[n]^3), {n, 1, 10^4}], 10]
        Block[{$MaxExtraPrecision = 1000}, Do[CC = Join[{0},
            Series[(-1 + 2 Prime[n])/(1 - 3 Prime[n] + Prime[n]^2 + Prime[n]^3) //. Prime[n] → 1/x, {x, 0, t}][[3]]];
          Print[N[-Sum[CC[[k]] * (PrimeZetaP'[k] + Log[2]/2^k), {k, 1, Length[CC]}] +
            (Log[Prime[n]] (-1 + 2 Prime[n]))/(1 - 3 Prime[n] + Prime[n]^2 + Prime[n]^3) //. Prime[n] → 2, 75]], {t, 1000, 1500, 100}]]]
```

```
Out[•]= 0.8430029907

        0.843022057860309976419862376433964864100030746648133253729664187567119266888
        0.843022057860309976419862376433964864100030746648133253729664187567119266888
```

```
Out[•]= $Aborted
```

```mathematica
In[•]:= A111^(0,1,0)[1, 0, 0] :=
        0.843022057860309976419862376433964864100030746648133253729664187567119266887621591241`
        59165565392`75.;

        A111^(0,0,1)[1, 0, 0] := 2 * A111^(0,1,0)[1, 0, 0];
```

```mathematica
In[•]:= Simplify[D[Log[A11[s, w1, w2]], w2] //. {s → 1, w1 → 0, w2 → 0}]
```

```
Out[•]= (2 Log[Prime[n]] (-1 + 2 Prime[n]))/(1 - 3 Prime[n] + Prime[n]^2 + Prime[n]^3)
```

```mathematica
In[•]:= A111^(0,0,1)[1, 0, 0] := 2 * A111^(0,1,0)[1, 0, 0];

        A111^(0,1,1)[1, 0, 0] :=
        -1.190727816059283072434053614108579419745932058046617974308297869696307135808149094290`
        70550743993`75.;
```

```mathematica
In[•]:= Simplify[D[D[Log[A11[s, w1, w2]], w1], w2] //. {s → 1, w1 → 0, w2 → 0}]
```

```
Out[•]= -(Log[Prime[n]]^2 Prime[n] (-1 + Prime[n] + 2 Prime[n]^2))/((-1 + Prime[n]) (-1 + 2 Prime[n] + Prime[n]^2)^2)
```

```
In[•]:= N[Sum[-(Log[Prime[n]]^2 Prime[n] (-1 + Prime[n] + 2 Prime[n]^2))/((-1 + Prime[n]) (-1 + 2 Prime[n] + Prime[n]^2)^2), {n, 1, 10^4}], 10]

      Block[{$MaxExtraPrecision = 1000}, Do[CC = Join[{0}, Series[
          -(Prime[n] (-1 + Prime[n] + 2 Prime[n]^2))/((-1 + Prime[n]) (-1 + 2 Prime[n] + Prime[n]^2)^2) //. Prime[n] -> 1/x, {x, 0, t}][[3]]];

        Print[
          N[Sum[CC[[k]] * (PrimeZetaP''[k] - Log[2]^2/2^k - Log[3]^2/3^k), {k, 1, Length[CC]}] -
            (18 Log[2]^2)/49 - (15 Log[3]^2)/98, 75]], {t, 1000, 1200, 50}]]]
```

Out[•]= -1.190488334

     -1.19072781605928307243405361410857941974593205804661797430829786969630713581

     -1.19072781605928307243405361410857941974593205804661797430829786969630713581

     -1.19072781605928307243405361410857941974593205804661797430829786969630713581

Out[•]= $Aborted

```
      A111^(0,1,1)[1, 0, 0] :=
        -1.19072781605928307243405361410857941974593205804661797430829786969630713580814909429\
          0370550743993`75.;

      A111^(0,2,0)[1, 0, 0] :=
        -1.4502090878552495408288770871873123827738360514875886312466082616939436204655588727699\
          19442550403`75.;
```

```
In[•]:= Simplify[D[D[Log[A11[s, w1, w2]], w1], w1] //. {s -> 1, w1 -> 0, w2 -> 0}]
```

Out[•]= $-\dfrac{\text{Log}[\text{Prime}[n]]^2 \, \text{Prime}[n] \, (1 - 3\,\text{Prime}[n] + \text{Prime}[n]^2 + 2\,\text{Prime}[n]^3)}{(1 - 3\,\text{Prime}[n] + \text{Prime}[n]^2 + \text{Prime}[n]^3)^2}$

$-\dfrac{\text{Log}[\text{Prime}[n]]^2 \, \text{Prime}[n] \, (1 - 3\,\text{Prime}[n] + \text{Prime}[n]^2 + 2\,\text{Prime}[n]^3)}{(1 - 3\,\text{Prime}[n] + \text{Prime}[n]^2 + \text{Prime}[n]^3)^2}$

*In[•]:=* `N[Sum[-`$\frac{\text{Log[Prime[n]]}^2 \text{Prime[n]} \left(1 - 3\,\text{Prime[n]} + \text{Prime[n]}^2 + 2\,\text{Prime[n]}^3\right)}{\left(1 - 3\,\text{Prime[n]} + \text{Prime[n]}^2 + \text{Prime[n]}^3\right)^2}$`, {n, 1, 10^5}], 10]`

`Block[{$MaxExtraPrecision = 1000},`

`Do[CC = Join[{0}, Series[-`$\frac{\text{Prime[n]} \left(1 - 3\,\text{Prime[n]} + \text{Prime[n]}^2 + 2\,\text{Prime[n]}^3\right)}{\left(1 - 3\,\text{Prime[n]} + \text{Prime[n]}^2 + \text{Prime[n]}^3\right)^2}$` //.`

`Prime[n] → 1 / x, {x, 0, t}][[3]]];`

`Print[N[Sum[CC[[k]] * (PrimeZetaP''[k] - Log[2]^2 / 2^k), {k, 1, Length[CC]}] -`

$\frac{\text{Log[Prime[n]]}^2 \text{Prime[n]} \left(1 - 3\,\text{Prime[n]} + \text{Prime[n]}^2 + 2\,\text{Prime[n]}^3\right)}{\left(1 - 3\,\text{Prime[n]} + \text{Prime[n]}^2 + \text{Prime[n]}^3\right)^2}$` //.`

`n → 1, 75]], {t, 950, 1000, 25}]]`

*Out[•]=* `- 1. 450185894`

`-1.450209087855249540828877087187312382773836051487588631246608261693943620047`

`-1.450209087855249540828877087187312382773836051487588631246608261693943620047`

`-1.450209087855249540828877087187312382773836051487588631246608261693943620047`

`A111`$^{(0,2,0)}$`[1, 0, 0] :=`

`-1.4502090878552495408288770871873123827738360514875886312466082616939436204655588727`⟍
`9919442550403`75.;`

`A111`$^{(1,0,0)}$`[1, 0, 0] := 3 * A111`$^{(0,1,0)}$`[1, 0, 0];`

*In[•]:=* `Simplify[D[Log[A11[s, w1, w2]], s] //. {s → 1, w1 → 0, w2 → 0}]`

*Out[•]=* $\frac{3\,\text{Log[Prime[n]]}\,(-1 + 2\,\text{Prime[n]})}{1 - 3\,\text{Prime[n]} + \text{Prime[n]}^2 + \text{Prime[n]}^3}$

*In[•]:=* `A111`$^{(1,0,0)}$`[1, 0, 0] := 3 * A111`$^{(0,1,0)}$`[1, 0, 0];`

`A111`$^{(1,0,1)}$`[1, 0, 0] :=`

`-4.3261493912017205253130210687957063314292666382085313551309301190531184300490572819188`⟍
`85962161142`75.;`

*In[•]:=* `Simplify[D[D[Log[A11[s, w1, w2]], s], w2] //. {s → 1, w1 → 0, w2 → 0}]`

*Out[•]=* $-\frac{6\,\text{Log[Prime[n]]}^2\,\text{Prime[n]}^4}{\left(1 - 3\,\text{Prime[n]} + \text{Prime[n]}^2 + \text{Prime[n]}^3\right)^2}$

*In[ ]:=* `N[Sum[-` $\dfrac{6 \, \text{Log}[\text{Prime}[n]]^2 \, \text{Prime}[n]^4}{(1 - 3 \, \text{Prime}[n] + \text{Prime}[n]^2 + \text{Prime}[n]^3)^2}$ `, {n, 1, 10^4}], 10]`

`Block[{$MaxExtraPrecision = 1000}, Do[CC = Join[{0},`

`Series[-` $\dfrac{6 \, \text{Prime}[n]^4}{(1 - 3 \, \text{Prime}[n] + \text{Prime}[n]^2 + \text{Prime}[n]^3)^2}$ `//. Prime[n] → 1 / x, {x, 0, t}][[3]]];`

`Print[N[Sum[CC[[k]] * (PrimeZetaP''[k] - Log[2]^2 / 2^k), {k, 1, Length[CC]}] -`

$\dfrac{6 \, \text{Log}[\text{Prime}[n]]^2 \, \text{Prime}[n]^4}{(1 - 3 \, \text{Prime}[n] + \text{Prime}[n]^2 + \text{Prime}[n]^3)^2}$ `//. Prime[n] → 2, 75]], {t, 1000, 1100, 50}]]`

*Out[ ]=* −4.325430944

−4.3261493912017205253130210687957063314292666382085313551309301190531184 3005

−4.3261493912017205253130210687957063314292666382085313551309301190531184 3005

−4.3261493912017205253130210687957063314292666382085313551309301190531184 3005

`A111`$^{(1,0,1)}$`[1, 0, 0] :=`

   `−4.326149391201720525313021068795706331429266638208531355130930119053118430049057281915`
   `8885962161142`75. ;

`A111`$^{(1,1,1)}$`[1, 0, 0] :=`

   `4.92254639539847098591912103218131211873318198006792493319305105540487405972783417260228 7`
   `213077508`75. ;

*In[ ]:=* `Simplify[D[D[D[Log[A11[s, w1, w2]], s], w1], w2] //. {s → 1, w1 → 0, w2 → 0}]`

*Out[ ]=* $\dfrac{3 \, \text{Log}[\text{Prime}[n]]^3 \, \text{Prime}[n]^4 \left(-2 + 3 \, \text{Prime}[n] + \text{Prime}[n]^3\right)}{(1 - 3 \, \text{Prime}[n] + \text{Prime}[n]^2 + \text{Prime}[n]^3)^3}$

*In[ ]:=* `N[Sum[` $\dfrac{3 \, \text{Log}[\text{Prime}[n]]^3 \, \text{Prime}[n]^4 \left(-2 + 3 \, \text{Prime}[n] + \text{Prime}[n]^3\right)}{(1 - 3 \, \text{Prime}[n] + \text{Prime}[n]^2 + \text{Prime}[n]^3)^3}$ `, {n, 1, 10^4}], 10]`

`Block[{$MaxExtraPrecision = 1250}, Do[CC = Join[{0},`

`Series[` $\dfrac{3 \, \text{Prime}[n]^4 \left(-2 + 3 \, \text{Prime}[n] + \text{Prime}[n]^3\right)}{(1 - 3 \, \text{Prime}[n] + \text{Prime}[n]^2 + \text{Prime}[n]^3)^3}$ `//. Prime[n] → 1 / x, {x, 0, t}][[3]]];`

`Print[N[-Sum[CC[[k]] * (PrimeZetaP'''[k] + Log[2]^3 / 2^k), {k, 1, Length[CC]}] +`

$\dfrac{3 \, \text{Log}[\text{Prime}[n]]^3 \, \text{Prime}[n]^4 \left(-2 + 3 \, \text{Prime}[n] + \text{Prime}[n]^3\right)}{(1 - 3 \, \text{Prime}[n] + \text{Prime}[n]^2 + \text{Prime}[n]^3)^3}$ `//.`

`Prime[n] → 2, 75]], {t, 1000, 1100, 50}]]`

*Out[ ]=* 4.918005965

4.9225463953984709859191210321813121187331819800679249331930510554048740597 3

4.9225463953984709859191210321813121187331819800679249331930510554048740597 3

4.9225463953984709859191210321813121187331819800679249331930510554048740597 3

**A111$^{(1,1,1)}$[1, 0, 0] :=
4.9225463953984709859191210321813121187331819800679249331930510554048740597278341726026‹
287213077508`75. ;**

**A111$^{(1,0,2)}$[1, 0, 0] := 2 * A111$^{(1,1,1)}$[1, 0, 0];**

*In[ ]:=* **Simplify[D[D[D[Log[A11[s, w1, w2]], s], w2], w2] //. {s → 1, w1 → 0, w2 → 0}]**

*Out[ ]=*
$$\frac{6 \operatorname{Log}[\operatorname{Prime}[n]]^3 \operatorname{Prime}[n]^4 \left(-2 + 3 \operatorname{Prime}[n] + \operatorname{Prime}[n]^3\right)}{\left(1 - 3 \operatorname{Prime}[n] + \operatorname{Prime}[n]^2 + \operatorname{Prime}[n]^3\right)^3}$$

*In[ ]:=* **A111$^{(1,0,2)}$[1, 0, 0] := 2 * A111$^{(1,1,1)}$[1, 0, 0];**

**A111$^{(1,1,0)}$[1, 0, 0] := A111$^{(1,0,1)}$[1, 0, 0] / 2;**

*In[ ]:=* **Simplify[D[D[Log[A11[s, w1, w2]], s], w1] //. {s → 1, w1 → 0, w2 → 0}]**

*Out[ ]=*
$$-\frac{3 \operatorname{Log}[\operatorname{Prime}[n]]^2 \operatorname{Prime}[n]^4}{\left(1 - 3 \operatorname{Prime}[n] + \operatorname{Prime}[n]^2 + \operatorname{Prime}[n]^3\right)^2}$$

*In[ ]:=* **A111$^{(1,1,0)}$[1, 0, 0] := A111$^{(1,0,1)}$[1, 0, 0] / 2;**

**A111$^{(1,2,0)}$[1, 0, 0] :=
5.203442844714798964018612580621705162127001410139188853357790786607308986918066531527546‹
725469967`75. ;**

*In[ ]:=* **Simplify[D[D[D[Log[A11[s, w1, w2]], s], w1], w1] //. {s → 1, w1 → 0, w2 → 0}]**

*Out[ ]=*
$$\frac{3 \operatorname{Log}[\operatorname{Prime}[n]]^3 \operatorname{Prime}[n]^4 \left(-1 + \operatorname{Prime}[n] + \operatorname{Prime}[n]^2 + \operatorname{Prime}[n]^3\right)}{\left(1 - 3 \operatorname{Prime}[n] + \operatorname{Prime}[n]^2 + \operatorname{Prime}[n]^3\right)^3}$$

*In[ ]:=* `N[Sum[`$\frac{3\,\text{Log}[\text{Prime}[n]]^3\,\text{Prime}[n]^4\,(-1+\text{Prime}[n]+\text{Prime}[n]^2+\text{Prime}[n]^3)}{(1-3\,\text{Prime}[n]+\text{Prime}[n]^2+\text{Prime}[n]^3)^3}$`, {n, 1, 10^4}], 10]`

`Block[{$MaxExtraPrecision = 1250},`

`Do[CC = Join[{0}, Series[`$\frac{3\,\text{Prime}[n]^4\,(-1+\text{Prime}[n]+\text{Prime}[n]^2+\text{Prime}[n]^3)}{(1-3\,\text{Prime}[n]+\text{Prime}[n]^2+\text{Prime}[n]^3)^3}$` //.`

`Prime[n] → 1 / x, {x, 0, t}][[3]]];`

`Print[N[-Sum[CC[[k]] * (PrimeZetaP'''[k] + Log[2]^3 / 2^k), {k, 1, Length[CC]}] +`

$\frac{3\,\text{Log}[\text{Prime}[n]]^3\,\text{Prime}[n]^4\,(-1+\text{Prime}[n]+\text{Prime}[n]^2+\text{Prime}[n]^3)}{(1-3\,\text{Prime}[n]+\text{Prime}[n]^2+\text{Prime}[n]^3)^3}$` //.`

`Prime[n] → 2, 75]], {t, 1000, 1100, 50}]]`

*Out[ ]=* 5. 198902394

5. 20344284471479896401861258062170516212700141013918885335779907866073089869

5. 20344284471479896401861258062170516212700141013918885335779907866073089869

5. 20344284471479896401861258062170516212700141013918885335779907866073089869

*In[ ]:=* **A111**$^{(1,2,0)}$ **[1, 0, 0] :=**

  **5.2034428447147989640186125806217051621270014101391888533577990786607308986918066531527**
  **54725469967`75.;**

  **A111**$^{(2,0,0)}$ **[1, 0, 0] :=**
  **−2.7937396327899498121176904230895393701540841938169419521099624330960119534522179550818**
  **58689463119`75.;**

*In[ ]:=* **Simplify[D[D[Log[A11[s, w1, w2]], s], s] //. {s → 1, w1 → 0, w2 → 0}]**

*Out[ ]=* $-\dfrac{3\,\text{Log}[\text{Prime}[n]]^2\,\text{Prime}[n]\,(-1+2\,\text{Prime}[n])\,(-1-\text{Prime}[n]+\text{Prime}[n]^2)}{(1-3\,\text{Prime}[n]+\text{Prime}[n]^2+\text{Prime}[n]^3)^2}$

```mathematica
In[*]:= N[Sum[-(3 Log[Prime[n]]^2 Prime[n] (-1 + 2 Prime[n]) (-1 - Prime[n] + Prime[n]^2))/
        (1 - 3 Prime[n] + Prime[n]^2 + Prime[n]^3)^2, {n, 1, 10^4}],
      10]
    Block[{$MaxExtraPrecision = 1250},
      Do[CC = Join[{0}, Series[-(3 Prime[n] (-1 + 2 Prime[n]) (-1 - Prime[n] + Prime[n]^2))/
            (1 - 3 Prime[n] + Prime[n]^2 + Prime[n]^3)^2 //.
          Prime[n] -> 1/x, {x, 0, t}][[3]]];
        Print[N[Sum[CC[[k]] * (PrimeZetaP''[k] - Log[2]^2 / 2^k), {k, 1, Length[CC]}] -
            (3 Log[Prime[n]]^2 Prime[n] (-1 + 2 Prime[n]) (-1 - Prime[n] + Prime[n]^2))/
            (1 - 3 Prime[n] + Prime[n]^2 + Prime[n]^3)^2 //.
          Prime[n] -> 2, 75]], {t, 1000, 1100, 50}]]
Out[*]= -2.793021191
    -2.79373963278994981211769042308953937015408419381694195210996243309601195345
    -2.79373963278994981211769042308953937015408419381694195210996243309601195345
    -2.79373963278994981211769042308953937015408419381694195210996243309601195345
    A111^(2,0,0)[1, 0, 0] :=
      -2.79373963278994981211769042308953937015408419381694195210996243309601195345221795508
        1858689463119`75. ;

    A111^(2,1,0)[1, 0, 0] :=
      4.641649946082143007819629483740919075339362549996661013028303032149017220763861683754510
        494643289`75. ;

In[*]:= Simplify[D[D[D[Log[A11[s, w1, w2]], s], s], w1] //. {s -> 1, w1 -> 0, w2 -> 0}]
Out[*]= (3 Log[Prime[n]]^3 Prime[n]^4 (-3 + 5 Prime[n] - Prime[n]^2 + Prime[n]^3))/
      (1 - 3 Prime[n] + Prime[n]^2 + Prime[n]^3)^3
```

```
In[•]:= N[Sum[ (3 Log[Prime[n]]^3 Prime[n]^4 (-3 + 5 Prime[n] - Prime[n]^2 + Prime[n]^3)) / ((1 - 3 Prime[n] + Prime[n]^2 + Prime[n]^3)^3), {n, 1, 10^4}], 10]

     Block[{$MaxExtraPrecision = 1250},

       Do[CC = Join[{0}, Series[ (3 Prime[n]^4 (-3 + 5 Prime[n] - Prime[n]^2 + Prime[n]^3)) / ((1 - 3 Prime[n] + Prime[n]^2 + Prime[n]^3)^3) //.

          Prime[n] -> 1 / x, {x, 0, t}][[3]];

        Print[N[-Sum[CC[[k]] * (PrimeZetaP'''[k] + Log[2]^3 / 2^k), {k, 1, Length[CC]}] +

          (3 Log[Prime[n]]^3 Prime[n]^4 (-3 + 5 Prime[n] - Prime[n]^2 + Prime[n]^3)) / ((1 - 3 Prime[n] + Prime[n]^2 + Prime[n]^3)^3) //.

         Prime[n] -> 2, 75]], {t, 1000, 1100, 50}]]]
```

```
Out[•]= 4.637109535
```

```
     4.64164994608214300781962948374091907533936254999666101302830303214901722076
     4.64164994608214300781962948374091907533936254999666101302830303214901722076
     4.64164994608214300781962948374091907533936254999666101302830303214901722076
```

```
In[•]:= A111^(2,1,0)[1, 0, 0] :=
     4.641649946082143007819629483740919075339362549996661013028303032149017220763861683754
       510494643289`75.;
```

```
A111^(2,0,1)[1, 0, 0] := 2 * A111^(2,1,0)[1, 0, 0];
```

```
In[•]:= Simplify[D[D[D[Log[A11[s, w1, w2]], s], s], w2] //. {s -> 1, w1 -> 0, w2 -> 0}]
```

```
Out[•]= (6 Log[Prime[n]]^3 Prime[n]^4 (-3 + 5 Prime[n] - Prime[n]^2 + Prime[n]^3)) / ((1 - 3 Prime[n] + Prime[n]^2 + Prime[n]^3)^3)
```

```
In[•]:= A111^(2,0,1)[1, 0, 0] := 2 * A111^(2,1,0)[1, 0, 0];
```

```
A111^(0,0,2)[1, 0, 0] := 2 * A111^(0,1,1)[1, 0, 0];
```

```
In[•]:= Simplify[D[D[D[Log[A11[s, w1, w2]], w2], w2] //. {s -> 1, w1 -> 0, w2 -> 0}]
```

```
Out[•]= -(2 Log[Prime[n]]^2 Prime[n] (-1 + Prime[n] + 2 Prime[n]^2)) / ((-1 + Prime[n]) (-1 + 2 Prime[n] + Prime[n]^2)^2)
```

```
In[•]:= A111^(0,0,2)[1, 0, 0] := 2 * A111^(0,1,1)[1, 0, 0];
```

```mathematica
A111^(2,1,1) [1, 0, 0] :=
  -16. 662094909253587645981626246039311879645863234835104449613586577081516802089496923754`
    264885014504`75. ;
```

*In[ ]:=* `Simplify[D[D[D[D[Log[A11[s, w1, w2]], s], s], w1], w2] //. {s → 1, w1 → 0, w2 → 0}]`

*Out[ ]=* $-\left(\left(3 \log[\text{Prime}[n]]^4 \text{Prime}[n]^4\right.\right.$

$\left(6 - 19\,\text{Prime}[n] + 17\,\text{Prime}[n]^2 - 8\,\text{Prime}[n]^3 + 20\,\text{Prime}[n]^4 - \text{Prime}[n]^5 + \text{Prime}[n]^6)\right) \Big/$

$\left.\left(1 - 3\,\text{Prime}[n] + \text{Prime}[n]^2 + \text{Prime}[n]^3\right)^4\right)$

*In[ ]:=* 
```mathematica
N[Sum[-((3 Log[Prime[n]]^4 Prime[n]^4
      (6 - 19 Prime[n] + 17 Prime[n]^2 - 8 Prime[n]^3 + 20 Prime[n]^4 - Prime[n]^5 + Prime[n]^6))/
     (1 - 3 Prime[n] + Prime[n]^2 + Prime[n]^3)^4), {n, 1, 10^4}], 10]
Block[{$MaxExtraPrecision = 1500},
 Do[CC = Join[{0}, Series[-((3 Prime[n]^4 (6 - 19 Prime[n] + 17 Prime[n]^2 -
        8 Prime[n]^3 + 20 Prime[n]^4 - Prime[n]^5 + Prime[n]^6))/
      (1 - 3 Prime[n] + Prime[n]^2 + Prime[n]^3)^4) //. Prime[n] → 1/x, {x, 0, t}][[3]]];
   Print[N[Sum[CC[[k]] * (PrimeZetaP''''[k] - Log[2]^4 / 2^k), {k, 1, Length[CC]}] +
     -((3 Log[Prime[n]]^4 Prime[n]^4 (6 - 19 Prime[n] + 17 Prime[n]^2 -
        8 Prime[n]^3 + 20 Prime[n]^4 - Prime[n]^5 + Prime[n]^6))/
      (1 - 3 Prime[n] + Prime[n]^2 + Prime[n]^3)^4) //. Prime[n] → 2, 75]], {t, 950, 975, 25}]]
```

*Out[ ]=* -16.60429171

-16.662094909253587645981626246039311879645863234835104449613586577081516168021

-16.662094909253587645981626246039311879645863234835104449613586577081516168021

*In[ ]:=*
```mathematica
A111^(2,1,1) [1, 0, 0] :=
  -16. 662094909253587645981626246039311879645863234835104449613586577081516168020894969237`
    54264885014504`75. ;
```

```mathematica
A111^(2,2,0) [1, 0, 0] :=
  -16. 1369857125556344357955245341546849960669129554600129510664007388581442478420424490947`
    068963096076`75. ;
```

*In[ ]:=* `Simplify[D[D[D[D[Log[A11[s, w1, w2]], s], s], w1], w1] //. {s → 1, w1 → 0, w2 → 0}]`

*Out[ ]=* $-\dfrac{3 \log[\text{Prime}[n]]^4 \text{Prime}[n]^4 \left(3 - 8\,\text{Prime}[n] + 3\,\text{Prime}[n]^2 + 11\,\text{Prime}[n]^4 + \text{Prime}[n]^6\right)}{\left(1 - 3\,\text{Prime}[n] + \text{Prime}[n]^2 + \text{Prime}[n]^3\right)^4}$

```mathematica
In[ ]:= N[Sum[-(3 Log[Prime[n]]^4 Prime[n]^4 (3 - 8 Prime[n] + 3 Prime[n]^2 + 11 Prime[n]^4 + Prime[n]^6)) / (1 - 3 Prime[n] + Prime[n]^2 + Prime[n]^3)^4,
    {n, 1, 10^4}], 10]
  Block[{$MaxExtraPrecision = 1500}, Do[CC =
    Join[{0}, Series[-(3 Prime[n]^4 (3 - 8 Prime[n] + 3 Prime[n]^2 + 11 Prime[n]^4 + Prime[n]^6)) / (1 - 3 Prime[n] + Prime[n]^2 + Prime[n]^3)^4 //.
      Prime[n] -> 1 / x, {x, 0, t}][[3]]];
    Print[N[Sum[CC[[k]] * (PrimeZetaP''''[k] - Log[2]^4 / 2^k), {k, 1, Length[CC]}] +
      -(3 Log[Prime[n]]^4 Prime[n]^4 (3 - 8 Prime[n] + 3 Prime[n]^2 + 11 Prime[n]^4 + Prime[n]^6)) / (1 - 3 Prime[n] + Prime[n]^2 + Prime[n]^3)^4 //.
      Prime[n] -> 2, 75]], {t, 950, 975, 25}]]]
```

```
Out[ ]= -16. 07918227
```

```
-16. 1369857125556344435795524534154684996066912955460012951066400738858144 2478

-16. 1369857125556344435795524534154684996066912955460012951066400738858144 2478
```

```mathematica
In[ ]:= A111^(2,2,0) [1, 0, 0] :=
  -16. 1369857125556344435795524534154684996066912955460012951066400738858144 2478420424490\
  94068963096076`75. ;
```

## Computing the 1st residue

```mathematica
In[253]:= R1 :=
  N[Simplify[3 Residue[Residue[Residue[f1[s, w1, w2], {s, 1}], {w2, 0}], {w1, 0}]], 100];
  R1
```

```
Out[254]= X (0. 21624056962947197947530794007676246063032030126961119593279154282375 55 +
  1. 4966102272251050511899031516827078178885409224771541691270934334028 51 Log[X] +
  2. 8685882348408085224411732837383493409774130078934582124318982049547 88 Log[X]^2 +
  0. 8190032873634129367154626837332287795235599696443547376398391712155 2 Log[X]^3 +
  0. 0544446791548840945807518785298617032826994387503389844120691008809 0662 Log[X]^4)
```

# Constants from the 2nd residue

```mathematica
A222^(0,1,0) [1, 0, 1] := A111^(0,1,0) [1, 0, 0];
```

```mathematica
In[ ]:= Simplify[D[Log[A22[s, w1, w2]], w1] //. {s -> 1, w1 -> 0, w2 -> 1}]
```

```
Out[ ]= (Log[Prime[n]] (-1 + 2 Prime[n])) / (1 - 3 Prime[n] + Prime[n]^2 + Prime[n]^3)
```

*In[ ]:=* **A222$^{(0,1,0)}$[1, 0, 1] := A111$^{(0,1,0)}$[1, 0, 0];**

**A222$^{(0,0,1)}$[1, 0, 1] := 2 * A111$^{(0,1,0)}$[1, 0, 0];**

*In[ ]:=* **Simplify[D[Log[A22[s, w1, w2]], w2] //. {s → 1, w2 → 1, w1 → 0}]**

*Out[ ]=* $\dfrac{2 \, \text{Log}[\text{Prime}[n]] \, (-1 + 2 \, \text{Prime}[n])}{1 - 3 \, \text{Prime}[n] + \text{Prime}[n]^2 + \text{Prime}[n]^3}$

**A222$^{(0,0,1)}$[1, 0, 1] := 2 * A111$^{(0,1,0)}$[1, 0, 0];**

**A222$^{(1,0,0)}$[1, 0, 1] := A111$^{(0,1,0)}$[1, 0, 0];**

*In[ ]:=* **Simplify[D[Log[A22[s, w1, w2]], s] //. {s → 1, w1 → 0, w2 → 1}]**

*Out[ ]=* $\dfrac{\text{Log}[\text{Prime}[n]] \, (-1 + 2 \, \text{Prime}[n])}{1 - 3 \, \text{Prime}[n] + \text{Prime}[n]^2 + \text{Prime}[n]^3}$

*In[ ]:=* **A222$^{(1,0,0)}$[1, 0, 1] := A111$^{(0,1,0)}$[1, 0, 0];**

**A222$^{(0,1,1)}$[1, 0, 1] := A111$^{(0,1,1)}$[1, 0, 0];**

*In[ ]:=* **Simplify[D[D[Log[A22[s, w1, w2]], w1], w2] //. {s → 1, w1 → 0, w2 → 1}]**

*Out[ ]=* $-\dfrac{\text{Log}[\text{Prime}[n]]^2 \, \text{Prime}[n] \, \left(-1 + \text{Prime}[n] + 2 \, \text{Prime}[n]^2\right)}{(-1 + \text{Prime}[n]) \, \left(-1 + 2 \, \text{Prime}[n] + \text{Prime}[n]^2\right)^2}$

*In[ ]:=* **A222$^{(0,1,1)}$[1, 0, 1] := A111$^{(0,1,1)}$[1, 0, 0];**

**A222$^{(0,1,2)}$[1, 0, 1] :=**
**2.2534733048561020954891928734646028057502414584174895088463474717336136392067608652444917**
**245099618`75. ;**

*In[ ]:=* **Simplify[D[D[D[Log[A22[s, w1, w2]], w1], w2], w2] //. {s → 1, w1 → 0, w2 → 1}]**

*Out[ ]=* $\left(\text{Log}[\text{Prime}[n]]^3 \, \text{Prime}[n] \right.$
$\left. \left(-1 + \text{Prime}[n] + 6 \, \text{Prime}[n]^2 - 6 \, \text{Prime}[n]^3 - 3 \, \text{Prime}[n]^4 - 3 \, \text{Prime}[n]^5 + 2 \, \text{Prime}[n]^6\right)\right) \Big/$
$\left(1 - 3 \, \text{Prime}[n] + \text{Prime}[n]^2 + \text{Prime}[n]^3\right)^3$

```mathematica
In[*]:= N[Sum[(Log[Prime[n]]^3 Prime[n]
         (-1 + Prime[n] + 6 Prime[n]^2 - 6 Prime[n]^3 - 3 Prime[n]^4 - 3 Prime[n]^5 + 2 Prime[n]^6)) /
        (1 - 3 Prime[n] + Prime[n]^2 + Prime[n]^3)^3, {n, 1, 10^4}], 10]
   Block[{$MaxExtraPrecision = 1250},
    Do[CC = Join[{0}, Series[(Prime[n] (-1 + Prime[n] + 6 Prime[n]^2 -
            6 Prime[n]^3 - 3 Prime[n]^4 - 3 Prime[n]^5 + 2 Prime[n]^6)) /
          (1 - 3 Prime[n] + Prime[n]^2 + Prime[n]^3)^3 //. Prime[n] -> 1 / x, {x, 0, t}][[3]];
     Print[N[-Sum[CC[[k]] * (PrimeZetaP'''[k] + Log[2]^3 / 2^k), {k, 1, Length[CC]}] +
        (Log[Prime[n]]^3 Prime[n]
         (-1 + Prime[n] + 6 Prime[n]^2 - 6 Prime[n]^3 - 3 Prime[n]^4 - 3 Prime[n]^5 + 2 Prime[n]^6)) /
        (1 - 3 Prime[n] + Prime[n]^2 + Prime[n]^3)^3 //. Prime[n] -> 2, 75]], {t, 750, 750, 50}]]]
```

Out[*]= 2.250446371

   2.25347330485610209548919287346460280575024145841748950884634747173361363921

```mathematica
In[*]:= A222^(0,1,2) [1, 0, 1] :=
     2.25347330485610209548919287346460280575024145841748950884634747173361363920676086524\
       491245099618`75.;
```

```mathematica
A222^(0,0,2) [1, 0, 1] := A111^(0,0,2) [1, 0, 0];
```

```mathematica
In[*]:= Simplify[D[D[Log[A22[s, w1, w2]], w2], w2] //. {s -> 1, w1 -> 0, w2 -> 1}]
```

Out[*]= $-\dfrac{2 \log[\text{Prime}[n]]^2 \, \text{Prime}[n] \, (-1 + \text{Prime}[n] + 2 \, \text{Prime}[n]^2)}{(-1 + \text{Prime}[n]) \, (-1 + 2 \, \text{Prime}[n] + \text{Prime}[n]^2)^2}$

```mathematica
A222^(0,0,2) [1, 0, 1] := A111^(0,0,2) [1, 0, 0];
```

```mathematica
A222^(1,0,1) [1, 0, 1] := 2 * A222^(1,1,0) [1, 0, 1];
```

```mathematica
In[*]:= Simplify[D[D[Log[A22[s, w1, w2]], s], w2] //. {s -> 1, w1 -> 0, w2 -> 1}]
```

Out[*]= $-\dfrac{2 \log[\text{Prime}[n]]^2 \, \text{Prime}[n] \, (-1 + 2 \, \text{Prime}[n] + \text{Prime}[n]^2 + \text{Prime}[n]^3)}{(1 - 3 \, \text{Prime}[n] + \text{Prime}[n]^2 + \text{Prime}[n]^3)^2}$

```mathematica
A222^(1,0,1) [1, 0, 1] := 2 * A222^(1,1,0) [1, 0, 1];
```

```mathematica
A222^(1,1,0) [1, 0, 1] :=
   -0.97234687954157719022245692028927374596870126105764770325716718983025207921637954823150\
     5572162349`75.;
```

```mathematica
In[*]:= Simplify[D[D[Log[A22[s, w1, w2]], s], w1] //. {s -> 1, w1 -> 0, w2 -> 1}]
```

$In[\circ]:=$ $-\dfrac{\text{Log}[\text{Prime}[n]]^2\,\text{Prime}[n]\,\left(-1+2\,\text{Prime}[n]+\text{Prime}[n]^2+\text{Prime}[n]^3\right)}{\left(1-3\,\text{Prime}[n]+\text{Prime}[n]^2+\text{Prime}[n]^3\right)^2}$

$-\dfrac{\text{Log}[\text{Prime}[n]]^2\,\text{Prime}[n]\,\left(-1+2\,\text{Prime}[n]+\text{Prime}[n]^2+\text{Prime}[n]^3\right)}{\left(1-3\,\text{Prime}[n]+\text{Prime}[n]^2+\text{Prime}[n]^3\right)^2}$ //. $\text{Prime}[n]\to 2$

$Out[\circ]=$ $-\dfrac{\text{Log}[\text{Prime}[n]]^2\,\text{Prime}[n]\,\left(-1+2\,\text{Prime}[n]+\text{Prime}[n]^2+\text{Prime}[n]^3\right)}{\left(1-3\,\text{Prime}[n]+\text{Prime}[n]^2+\text{Prime}[n]^3\right)^2}$

$Out[\circ]=$ $-\dfrac{30}{49}\,\text{Log}[2]^2$

$In[\circ]:=$ $\text{N}\Big[\text{Sum}\Big[-\dfrac{\text{Log}[\text{Prime}[n]]^2\,\text{Prime}[n]\,\left(-1+2\,\text{Prime}[n]+\text{Prime}[n]^2+\text{Prime}[n]^3\right)}{\left(1-3\,\text{Prime}[n]+\text{Prime}[n]^2+\text{Prime}[n]^3\right)^2},\ \{n,\,1,\,10^\wedge4\}\Big],\ 10\Big]$

$\text{Block}\Big[\{\$\text{MaxExtraPrecision}=1000\},$

$\text{Do}\Big[\text{CC}=\text{Join}\Big[\{0\},\ \text{Series}\Big[-\dfrac{\text{Prime}[n]\,\left(-1+2\,\text{Prime}[n]+\text{Prime}[n]^2+\text{Prime}[n]^3\right)}{\left(1-3\,\text{Prime}[n]+\text{Prime}[n]^2+\text{Prime}[n]^3\right)^2}\ //.$

$\text{Prime}[n]\to 1\,/\,x,\ \{x,\,0,\,t\}\Big]\llbracket 3\rrbracket\Big];$

$\text{Print}\Big[\text{N}\Big[\text{Sum}[\text{CC}\llbracket k\rrbracket * (\text{PrimeZetaP''}[k]-\text{Log}[2]^\wedge 2\,/\,2^\wedge k),\ \{k,\,1,\,\text{Length}[\text{CC}]\}]-\dfrac{30}{49}\,\text{Log}[2]^2,$

$75\Big]\Big],\ \{t,\,900,\,950,\,25\}\Big]\Big]$

$Out[\circ]=$ $-0.9722271378$

$-0.97234687954157719022245692028927374596870126105764770325716718983025207 9216$

$-0.97234687954157719022245692028927374596870126105764770325716718983025207 9216$

$-0.97234687954157719022245692028927374596870126105764770325716718983025207 9216$

$In[\circ]:=$ $\text{A222}^{(1,1,0)}[1,\,0,\,1]:=$
$-0.97234687954157719022245692028927374596870126105764770325716718983025207921637954823\backslash$
$1505572162349`75.\ ;$

$\text{A222}^{(1,1,1)}[1,\,0,\,1]:=$
$2.669073090542368890429928158716709312982940521650435424346703583671623039891572526148001\backslash$
$29547396`75.\ ;$

$\text{A222}^{(2,0,0)}[1,\,0,\,1]:=$
$3.477103517494925093640244486284714453212584966506884809535302065617610635029598419913002\backslash$
$810379413`75.\ ;$

$In[\circ]:=$ $\text{Simplify}[D[D[\text{Log}[\text{A22}[s,\,w1,\,w2]],\,s],\,s]\ //.\ \{s\to 1,\,w1\to 0,\,w2\to 1\}]$

$Out[\circ]=$ $\dfrac{\text{Log}[\text{Prime}[n]]^2\,\text{Prime}[n]\,\left(-5+7\,\text{Prime}[n]+11\,\text{Prime}[n]^2+2\,\text{Prime}[n]^3\right)}{\left(1-3\,\text{Prime}[n]+\text{Prime}[n]^2+\text{Prime}[n]^3\right)^2}$

```mathematica
In[•]:= N[Sum[
      Log[Prime[n]]^2 Prime[n] (-5 + 7 Prime[n] + 11 Prime[n]^2 + 2 Prime[n]^3)
      ───────────────────────────────────────────────────────────────────────, {n, 1, 10^4}],
              (1 - 3 Prime[n] + Prime[n]^2 + Prime[n]^3)^2
      10]

    Block[{$MaxExtraPrecision = 1250},

      Do[CC = Join[{0}, Series[
          Prime[n] (-5 + 7 Prime[n] + 11 Prime[n]^2 + 2 Prime[n]^3)
          ──────────────────────────────────────────────────────── //.
              (1 - 3 Prime[n] + Prime[n]^2 + Prime[n]^3)^2
          Prime[n] → 1 / x, {x, 0, t}][[3]]];
        Print[N[Sum[CC[[k]] * (PrimeZetaP''[k] - Log[2]^2 / 2^k), {k, 1, Length[CC]}] +
          Log[Prime[n]]^2 Prime[n] (-5 + 7 Prime[n] + 11 Prime[n]^2 + 2 Prime[n]^3)
          ──────────────────────────────────────────────────────────────────── //.
              (1 - 3 Prime[n] + Prime[n]^2 + Prime[n]^3)^2
          Prime[n] → 2, 75]], {t, 900, 1000, 50}]]]
```

```
Out[•]= 3.476864029

       3.47710351749492509364024448628471445321258496650688480953530206561761063503

       3.47710351749492509364024448628471445321258496650688480953530206561761063503

       3.47710351749492509364024448628471445321258496650688480953530206561761063503
```

```mathematica
In[•]:= A222^(2,0,0)[1, 0, 1] :=
       3.4771035174949250936402444862847144532125849665068848095353020656176106350350295984199134
         002810379413`75.;
```

A222^(2,1,0)[1, 0, 1] :=
  -2.9499695398586968685294197071571023563767599517216993445114516069274798788555450088943ͺ
     80831862888`75.;

## Computing the 2nd residue

*In[ ]:=* `Simplify[-3 Residue[Residue[Residue[f2[s, w1, w2], {s, 1}], {w2, 1}], {w1, 0}]]`

*Out[ ]=* $-\dfrac{1}{2}$ X $\big($A2[1, 0, 1] Log[X]$^3$ + Log[X]$^2$ $\big(($-4 + 9 EulerGamma$)$ A2[1, 0, 1] +

2 A2$^{(0,0,1)}$[1, 0, 1] + 3 A2$^{(0,1,0)}$[1, 0, 1] + 2 A2$^{(1,0,0)}$[1, 0, 1]$\big)$ +

Log[X] $\big($6 A2[1, 0, 1] $\big($1 - 4 EulerGamma + 4 EulerGamma$^2$ - StieltjesGamma[1]$\big)$ +

4 $($-1 + 3 EulerGamma$)$ A2$^{(0,0,1)}$[1, 0, 1] + A2$^{(0,0,2)}$[1, 0, 1] - 12 A2$^{(0,1,0)}$[1, 0, 1] +

18 EulerGamma A2$^{(0,1,0)}$[1, 0, 1] + 6 A2$^{(0,1,1)}$[1, 0, 1] - 4 A2$^{(1,0,0)}$[1, 0, 1] + 12 EulerGamma

A2$^{(1,0,0)}$[1, 0, 1] + 2 A2$^{(1,0,1)}$[1, 0, 1] + 6 A2$^{(1,1,0)}$[1, 0, 1] + A2$^{(2,0,0)}$[1, 0, 1]$\big)$ +

3 $\big($6 EulerGamma$^3$ A2[1, 0, 1] - 6 $($-1 + StieltjesGamma[1]$)$ A2$^{(0,1,0)}$[1, 0, 1] -

4 A2$^{(0,1,1)}$[1, 0, 1] + A2$^{(0,1,2)}$[1, 0, 1] +

6 EulerGamma$^2$ $\big($-2 A2[1, 0, 1] + A2$^{(0,0,1)}$[1, 0, 1] + A2$^{(0,1,0)}$[1, 0, 1] + A2$^{(1,0,0)}$[1, 0, 1]$\big)$ -

4 A2$^{(1,1,0)}$[1, 0, 1] + 2 A2$^{(1,1,1)}$[1, 0, 1] +

EulerGamma $\big($-6 A2[1, 0, 1] $($-1 + StieltjesGamma[1]$)$ - 4 A2$^{(0,0,1)}$[1, 0, 1] +

A2$^{(0,0,2)}$[1, 0, 1] - 12 A2$^{(0,1,0)}$[1, 0, 1] + 6 A2$^{(0,1,1)}$[1, 0, 1] - 4 A2$^{(1,0,0)}$[1, 0, 1] +

2 A2$^{(1,0,1)}$[1, 0, 1] + 6 A2$^{(1,1,0)}$[1, 0, 1] + A2$^{(2,0,0)}$[1, 0, 1]$\big)$ + A2$^{(2,1,0)}$[1, 0, 1]$\big)\big)$

*In[255]:=* `R2 :=`
`N[Simplify[-3 Residue[Residue[Residue[f2[s, w1, w2], {s, 1}], {w2, 1}], {w1, 0}]], 75];`
`R2`

*Out[256]=* X $\big($-0.06326660892676760197688927317822860811567388270468340907940215810351 5197 -

1.0933034122066966535562759670545859085819202092340644500822839165554399 2 Log[X] -

0.9562815352706989321694346329505346528718968016949582885856303712777953 7 Log[X]$^2$ -

0.1088893583097681891615037570597234065653988775006779688241382017618132 46 Log[X]$^3\big)$

---

# Constants from the last residue

*In[ ]:=* `Simplify[A3[1, 1, 1]]`

*Out[ ]=* 0.2177787166195363783230075141194468131307977550013559376482764035236264 9111225262055792 :
54438235637657

`A333`$^{(0,0,1)}$`[1, 1, 1] := A111`$^{(0,1,0)}$`[1, 0, 0];`

*In[ ]:=* `Simplify[D[Log[A33[s, w1, w2]], w2] //. {s → 1, w1 → 1, w2 → 1}]`

*Out[ ]=* $\dfrac{\text{Log[Prime[n]] } (-1 + 2 \text{ Prime[n]})}{1 - 3 \text{ Prime[n]} + \text{Prime[n]}^2 + \text{Prime[n]}^3}$

*In[ ]:=* `A333`$^{(0,0,1)}$`[1, 1, 1] := A111`$^{(0,1,0)}$`[1, 0, 0];`

```mathematica
A333^(0,1,0)[1, 1, 1] := A333^(0,0,1)[1, 1, 1];
```

```mathematica
In[·]:= Simplify[D[Log[A33[s, w1, w2]], w1] //. {s → 1, w1 → 1, w2 → 1}]
```

$$Out[·]= \frac{\text{Log}[\text{Prime}[n]]\ (-1 + 2\,\text{Prime}[n])}{1 - 3\,\text{Prime}[n] + \text{Prime}[n]^2 + \text{Prime}[n]^3}$$

```mathematica
In[·]:= A333^(0,1,0)[1, 1, 1] := A333^(0,0,1)[1, 1, 1];
```

```mathematica
A333^(0,0,2)[1, 1, 1] := A111^(0,2,0)[1, 0, 0];
```

```mathematica
In[·]:= Simplify[D[D[Log[A33[s, w1, w2]], w2], w2] //. {s → 1, w1 → 1, w2 → 1}]
```

$$Out[·]= -\frac{\text{Log}[\text{Prime}[n]]^2\,\text{Prime}[n]\ \left(1 - 3\,\text{Prime}[n] + \text{Prime}[n]^2 + 2\,\text{Prime}[n]^3\right)}{\left(1 - 3\,\text{Prime}[n] + \text{Prime}[n]^2 + \text{Prime}[n]^3\right)^2}$$

```mathematica
In[·]:= A333^(0,0,2)[1, 1, 1] := A111^(0,2,0)[1, 0, 0];
```

```mathematica
A333^(0,2,0)[1, 1, 1] := A111^(0,2,0)[1, 0, 0];
```

```mathematica
In[·]:= Simplify[D[D[Log[A33[s, w1, w2]], w1], w1] //. {s → 1, w1 → 1, w2 → 1}]
```

$$Out[·]= -\frac{\text{Log}[\text{Prime}[n]]^2\,\text{Prime}[n]\ \left(1 - 3\,\text{Prime}[n] + \text{Prime}[n]^2 + 2\,\text{Prime}[n]^3\right)}{\left(1 - 3\,\text{Prime}[n] + \text{Prime}[n]^2 + \text{Prime}[n]^3\right)^2}$$

```mathematica
In[·]:= A333^(0,2,0)[1, 1, 1] := A111^(0,2,0)[1, 0, 0];
```

```mathematica
A333^(0,1,1)[1, 1, 1] :=
  0.2594812717959664683948234730787329630279039934409706569383103919976364846574097767246135\
   359110604`75.;
```

```mathematica
In[·]:= Simplify[D[D[Log[A33[s, w1, w2]], w1], w2] //. {s → 1, w1 → 1, w2 → 1}]
```

$$Out[·]= \frac{\text{Log}[\text{Prime}[n]]^2\,\text{Prime}[n]^2\ (-1 + 2\,\text{Prime}[n])}{\left(1 - 3\,\text{Prime}[n] + \text{Prime}[n]^2 + \text{Prime}[n]^3\right)^2}$$

```mathematica
In[·]:= N[Sum[(Log[Prime[n]]^2 Prime[n]^2 (-1 + 2 Prime[n]))/(1 - 3 Prime[n] + Prime[n]^2 + Prime[n]^3)^2, {n, 1, 10^5}], 10]

  Block[{$MaxExtraPrecision = 1000}, Do[CC = Join[{0, 0},
      Series[(Prime[n]^2 (-1 + 2 Prime[n]))/(1 - 3 Prime[n] + Prime[n]^2 + Prime[n]^3)^2 //. Prime[n] → 1/x, {x, 0, t}][[3]]];
    Print[N[Sum[CC[[k]] * (PrimeZetaP''[k] - Log[2]^2/2^k), {k, 1, Length[CC]}] +
      (Log[Prime[n]]^2 Prime[n]^2 (-1 + 2 Prime[n]))/(1 - 3 Prime[n] + Prime[n]^2 + Prime[n]^3)^2 //. n → 1, 75]], {t, 1000, 1100, 50}]]
```

$$Out[·]= 0.\,2594812718$$

0.259481271795966468394823473078732963027903993440970656938310391997636484657

0.259481271795966468394823473078732963027903993440970656938310391997636484657

*Out[•]=* $Aborted

A333$^{(0,1,1)}$ [1, 1, 1] :=
  0.259481271795966468394823473078732963027903993440970656938310391997636484657409776724`
    613359110604`75.;

A333$^{(1,1,0)}$ [1, 1, 1] := A222$^{(1,1,0)}$ [1, 0, 1];

*In[•]=* **Simplify[D[D[Log[A33[s, w1, w2]], w1], s] //. {s → 1, w1 → 1, w2 → 1}]**

*Out[•]=* $-\dfrac{\text{Log}[\text{Prime}[n]]^2 \, \text{Prime}[n] \, (-1 + 2 \, \text{Prime}[n] + \text{Prime}[n]^2 + \text{Prime}[n]^3)}{(1 - 3 \, \text{Prime}[n] + \text{Prime}[n]^2 + \text{Prime}[n]^3)^2}$

*In[•]=* **A333$^{(1,1,0)}$ [1, 1, 1] := A222$^{(1,1,0)}$ [1, 0, 1];**

A333$^{(1,0,1)}$ [1, 1, 1] := A333$^{(1,1,0)}$ [1, 1, 1];

*In[•]=* **Simplify[D[D[Log[A33[s, w1, w2]], w2], s] //. {s → 1, w1 → 1, w2 → 1}]**

*Out[•]=* $-\dfrac{\text{Log}[\text{Prime}[n]]^2 \, \text{Prime}[n] \, (-1 + 2 \, \text{Prime}[n] + \text{Prime}[n]^2 + \text{Prime}[n]^3)}{(1 - 3 \, \text{Prime}[n] + \text{Prime}[n]^2 + \text{Prime}[n]^3)^2}$

*In[•]=* **A333$^{(1,0,1)}$ [1, 1, 1] := A333$^{(1,1,0)}$ [1, 1, 1];**

A333$^{(1,0,0)}$ [1, 1, 1] := A111$^{(0,1,0)}$ [1, 0, 0];

*In[•]=* **Simplify[D[Log[A33[s, w1, w2]], s] //. {s → 1, w1 → 1, w2 → 1}]**

*Out[•]=* $\dfrac{\text{Log}[\text{Prime}[n]] \, (-1 + 2 \, \text{Prime}[n])}{1 - 3 \, \text{Prime}[n] + \text{Prime}[n]^2 + \text{Prime}[n]^3}$

A333$^{(1,0,0)}$ [1, 1, 1] := A111$^{(0,1,0)}$ [1, 0, 0];

A333$^{(2,0,0)}$ [1, 1, 1] := A222$^{(2,0,0)}$ [1, 0, 1];

*In[•]=* **Simplify[D[D[Log[A33[s, w1, w2]], s], s] //. {s → 1, w1 → 1, w2 → 1}]**

*Out[•]=* $\dfrac{\text{Log}[\text{Prime}[n]]^2 \, \text{Prime}[n] \, (-5 + 7 \, \text{Prime}[n] + 11 \, \text{Prime}[n]^2 + 2 \, \text{Prime}[n]^3)}{(1 - 3 \, \text{Prime}[n] + \text{Prime}[n]^2 + \text{Prime}[n]^3)^2}$

*In[•]=* **A333$^{(2,0,0)}$ [1, 1, 1] := A222$^{(2,0,0)}$ [1, 0, 1];**

## Computing the 3rd residue

In[251]:= **R3 := N[Simplify[Residue[Residue[Residue[f3[s, w1, w2], {s, 1}], {w2, 1}], {w1, 1}]], 75];**
**R3**

Out[252]= X (0. 134262687043915039723245950916112097651389855832944627655245227226957915 +
      0. 274556495814571990907943898434611746696601609461045629643292908311248609 Log[X] +
      0. 108889358309768189161503757059723406565398877500677968824138201761813246 Log[X]$^2$)

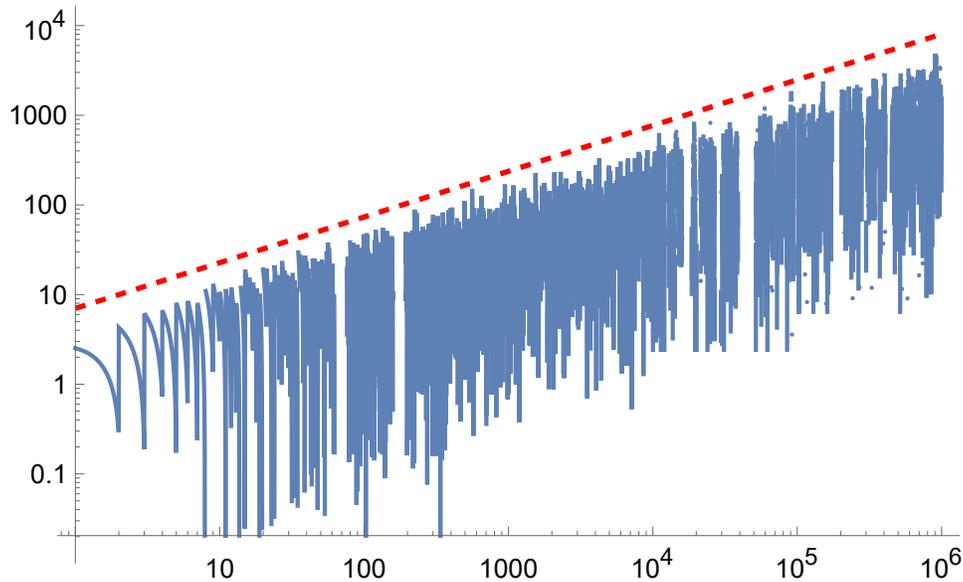

FIGURE 4. A log-log-plot of the error term $E_{2,2}(X, 1)$, for $1 \leq X \leq 10^6$, with slope of dashed line approximately 0.51 and $y$-intercept around 7, which numerically suggests that $|E_{2,2}(X, 1)| \leq 7X^{0.51}$.

PREVIOUS ADDRES: AMERICAN INSTITUTE OF MATHEMATICS, 600 E. BROKAW RD., SAN JOSE, CA 95112, USA.

*Email address:* `dtn@aimath.org`

CURRENT ADDRESS: DEPARTMENT OF MATHEMATICS AND STATISTICS, QUEEN'S UNIVERSITY, JEFFERY HALL, 48 UNIVERSITY AVE, KINGSTON, ONTARIO, K7L-3N6, CANADA.

*Email address:* `d.nguyen@queensu.ca`



\end{document}